\documentclass[a4paper,11pt]{amsart}
\usepackage[margin=1.2in]{geometry}
\usepackage[dvipsnames]{xcolor}

\usepackage{mathpazo,amsmath,amssymb,amscd,color,xcolor,mathrsfs,mathtools}
\usepackage{tikz}
\usepackage{float}
\usepackage{comment}
\definecolor{green}{RGB}{0,144,0}
\definecolor{bluegreen}{RGB}{17,100,180}
\usepackage[linkcolor = bluegreen, citecolor = bluegreen, colorlinks = True]{hyperref}
\numberwithin{equation}{section}
\numberwithin{figure}{section}
\numberwithin{table}{section}
\usepackage[all]{hypcap}
\usepackage{mathtools}

\newtheorem{theorem}{Theorem}[section]
\newtheorem*{theorem*}{Theorem}
\newtheorem{conjecture}[theorem]{Conjecture}
\newtheorem{lemma}[theorem]{Lemma}
\newtheorem{corollary}[theorem]{Corollary}
\newtheorem{proposition}[theorem]{Proposition}

\newtheorem{question}[theorem]{Question}

\newtheorem{algorithm}[theorem]{Algorithm}

\newcommand\zz{\mathbb{Z}}

\newtheorem{remark}[theorem]{Remark}
\newtheorem{example}[theorem]{Example}

\newcommand{\teichmuller}{Teichm{\"u}ller }
\newcommand{\calH}{\mathcal{H}}
\newcommand{\odd}{\mathcal{H}^{odd}}
\newcommand{\even}{\mathcal{H}^{even}}
\newcommand{\hyp}{\mathcal{H}^{hyp}}

\newcommand{\C}{\mathbb{C}}
\newcommand{\Z}{\mathbb{Z}}
\newcommand{\R}{\mathbb{R}}
\DeclareMathOperator{\SL}{\text{SL}}
\DeclareMathOperator{\Sp}{\text{Sp}}
\DeclareMathOperator{\PSL}{\text{PSL}}
\DeclareMathOperator{\PGL}{\text{PGL}}
\DeclareMathOperator{\PGamL}{\text{P$\Gamma$L}}
\DeclareMathOperator{\PG}{\text{PG}}

\DeclareMathOperator{\Sym}{\text{Sym}}
\DeclareMathOperator{\Alt}{\text{Alt}}
\DeclareMathOperator{\Mon}{\text{Mon}}
\DeclareMathOperator{\Pic}{\text{Pic}}
\DeclareMathOperator{\ind}{\text{ind}}

\DeclareMathOperator{\Gal}{\text{Gal}}
\DeclareMathOperator{\Mod}{\text{Mod}}
\DeclareMathOperator{\Aff}{\text{Aff}}
\DeclareMathOperator{\Aut}{\text{Aut}}
\DeclareMathOperator{\Spin}{\text{Spin}}

\title[Monodromy and spin parity of single-cylinder origamis]{On the monodromy and spin parity of single-cylinder origamis in the minimal stratum}

\author{Tarik Aougab}
\address{Department of Mathematics, Haverford College, Haverford, PA 19041, USA}
\curraddr{}
\email{taougab@haverford.edu}

\author{Adam Friedman-Brown}
\address{Department of Mathematics, University of Virginia, Charlottesville, VA 2290, USA}
\curraddr{}
\email{asm4pr@virginia.edu}

\author{Luke Jeffreys}
\address{School of Mathematics, University of Bristol, Fry Building, Woodland Road, Bristol BS8 1UG, UK}
\curraddr{}
\email{luke.jeffreys@bristol.ac.uk}
\thanks{The third author is a Leverhulme Early Career
Fellow (ECF-2023-553) and thanks the Leverhulme Trust for their support.}

\author{Jiajie Ma}
\address{Department of Mathematics, Duke University, Durham, NC 27708-0320, USA}
\curraddr{}
\email{jason.ma@duke.edu}

\date{}


\begin{document}

\begin{abstract}
In a paper with Menasco-Nieland, the first author constructed factorially many origamis in the minimal stratum of the moduli space of translation surfaces having simultaneously a single vertical cylinder and a single horizontal cylinder. Moreover, these origamis were constructed using the minimal number of squares required for origamis in the minimal stratum. We shall call such origamis minimal $[1,1]$-origamis. 

In this work, we calculate all of the spin parities of the Aougab-Menasco-Nieland origamis, and we therefore determine the connected component of the minimal stratum within which each is contained. Motivated by understanding the $\SL(2,\Z)$-orbits of these origamis, we investigate their monodromy groups, in particular proving that all of them are alternating or projective special linear groups. In fact, we prove more generally that the monodromy group of a minimal $[1,1]$-origami must almost always be a finite simple group. Finally, we determine the Kontsevich-Zorich monodromies of these origamis in low genus and give a conjecture in general.

Note that previous works in the literature (e.g., \cite{EKZ}, \cite{FFM}, \cite{G-R}, \cite{KanMat}, \cite{MYZ}, \cite{Zm}, \cite{Z3}) often chose to discuss just one of these $\SL(2,\Z)$-invariants at a time: in particular, to the best our knowledge, this is one of the first places where all of these $\SL(2,\Z)$-invariants are computed explicitly in a single paper for such a large family of origamis.
\end{abstract}

\maketitle
\tableofcontents

%%%%%%%%%%%%%%%%%%%%%%%%%%%%%%%%%%%%%%%%%%%%%%%%%%%%%%

\section{Introduction}\label{sec:intro}

An \emph{origami} (or \emph{square-tiled surface}) is a closed, connected and orientable surface obtained from a collection of unit squares in $\R^{2}$ by identifying via translation left-hand sides with right-hand sides, and top sides with bottom sides. The torus, obtained by identifying the opposite sides of a single unit square, is the basic example and, in fact, all origamis are branched covers of the square-torus. Indeed, an origami can be equivalently defined as a closed, connected Riemann surface $X$ equipped with a non-zero holomorphic 1-form $\omega$ such that $X$ is realised as a finite branched cover $f:X\to\C/(\Z+i\Z)$, branched only over $\textbf{0}\in\C/(\Z+i\Z)$, such that $\omega = f^{*}(dz)$.

From the latter definition, we see that origamis are special cases of \emph{translation surfaces} -- closed, connected Riemann surfaces $X$ equipped with a non-zero holomorphic 1-form $\omega$. The moduli space $\calH$ of genus $g$ translation surfaces is an important object of modern research. It is naturally stratified by the orders of the zeros of the 1-form $\omega$, and each stratum can have up to three connected components. We direct the reader to Section~\ref{sec:prelim} for more details. For now, we denote by $\calH(2g-2)$ the subset of $\calH$ of those translation surfaces $(X,\omega)$ for which $\omega$ has a single zero of order $2g-2$. This subset is called the \emph{minimal stratum}. It can have three connected components: a hyperelliptic component $\hyp(2g-2)$, a component $\odd(2g-2)$ called the odd component containing origamis with odd spin parity, and a component $\even(2g-2)$ called the even component containing origamis with even spin parity. Again, we direct the reader to Section~\ref{sec:prelim} for further details.

A \emph{cylinder} in an origami is a maximally embedded flat annulus. Equivalently, it is a maximal collection of freely homotopic closed geodesics in the (singular) flat metric determined by $\omega$ that avoid the zeros of $\omega$. An important piece of geometric/combinatorial information is the number of cylinders in the vertical and horizontal directions. The simplest (yet also the most combinatorially restricted) origamis are those with simultaneously a single horizontal cylinder and a single vertical cylinder. We shall call such an origami a \emph{$[1,1]$-origami}.

The third author~\cite{J1} has determined the minimal number of squares required to construct a $[1,1]$-origami in any connected component of any stratum of the moduli space $\calH$ of translation surfaces. For the odd and even components of $\calH(2g-2)$, a $[1,1]$-origami requires at least $2g-1$ squares, while such an origami requires $4g-4$ squares in the hyperelliptic component. We shall call a $[1,1]$-origami in $\calH(2g-2)$ \emph{minimal} if it is built from $2g-1$ squares.

The above mentioned work of the third author constructed only a single $[1,1]$-origami realising the minimal number of squares in each connected component of the moduli space. It is therefore natural to ask how many such $[1,1]$-origamis exist in each connected component. The first author in joint work with Menasco-Nieland~\cite{AMN} constructed factorially many minimal $[1,1]$-origamis in the minimal stratum $\calH(2g-2)$, for $g\geq 3$. However, this work did not determine which connected components of $\calH(2g-2)$ the constructed origamis were contained in. That is, one can ask the following.

\begin{question}
Which connected components of $\calH(2g-2)$ contain the minimal $[1,1]$-origamis constructed by Aougab-Menasco-Nieland?
\end{question}

Note that the hyperelliptic component is ruled out by the discussion above (a $[1,1]$-origami here requiring at least $4g-4$ squares). We answer the above question with the following theorem.

\begin{theorem}\label{thm:AMN}
All of the $(g-2)!$ origamis of odd genus in the construction of Aougab-Menasco-Nieland have odd spin parity and hence lie in the odd component $\odd(2g-2)$ of the minimal stratum $\calH(2g-2)$.

Of the $(g-3)(g-3)!$ origamis of even genus in the construction of Aougab-Menasco-Nieland, $\frac{1}{4}(3(g-3)^{2}+1)(g-4)!$ have odd spin parity, so lie in the odd component $\odd(2g-2)$, while the remaining $(\frac{g}{2}-1)(\frac{g}{2}-2)(g-4)!$ have even spin parity and so lie in the even component $\even(2g-2)$.
\end{theorem}

We remark that much more is known about counting origamis and square-tiled surfaces \textit{in the asymptotics}, as the number of squares goes to infinity (see for instance the work of Delecroix-Goujard-Zograf-Zorich \cite{DGZZ1}, \cite{DGZZ2}, \cite{DGZZ3}). In the complex structure on the moduli space of quadratic differentials induced by period coordinates, square-tiled surfaces are the integer points; for this reason, computing such asymptotic counts allows for the calculation of Masur-Veech volumes (see also the works of Eskin-Okounkov~\cite{EO} and Zorich~\cite{Z1}). On the other hand, Theorem \ref{thm:AMN} addresses a somewhat orthogonal question: how many origamis are there in a particular connected component of the moduli space (namely, the non-hyperelliptic components of $\calH(2g-2)$) with precisely $2g-1$ squares? 

We see that all of the odd genus origamis are contained in the odd connected component, and the even genus origamis are partitioned between the odd and even components in a ratio that limits to 3:1. It is therefore natural to ask whether there exist factorially many odd genus minimal origamis in the even component. Moreover, as detailed in Section~\ref{sec:AMN} below, the even genus construction of Aougab-Menasco-Nieland requires (when understood geometrically) a surgery-like operation while the odd genus construction remains naturally algebraic.

In this paper, we give a natural generalisation of the odd genus construction of Aougab-Menasco-Nieland that allows us to construct factorially many odd genus minimal $[1,1]$-origamis in the even component and also avoids the surgery-like operation in even genus. However, as a cost of avoiding the surgery-like operation, we cannot produce as many origamis.

\begin{theorem}\label{thm:gen-const}
The natural generalisations of the odd genus Aougab-Menasco-Nieland construction given in Section~\ref{sec:gen-const} produce, for odd $g\geq 5$, $(g-5)!$ minimal $[1,1]$-origamis in the even component $\even(2g-2)$ of the minimal stratum, and, for even $g\geq 4$, $(g-4)!$ many minimal $[1,1]$-origamis in each of $\odd(2g-2)$ and $\even(2g-2)$.
\end{theorem}

The \emph{monodromy group} of an origami is a subgroup of the symmetric group of degree equal to the number of squares of the origami. An origami is said to be \emph{primitive} if its monodromy group is a primitive subgroup of the symmetric group. Equivalently, an origami is primitive if it is not a non-trivial proper cover of another origami. We direct the reader to Section~\ref{sec:prelim} for more details. In the case of minimal $[1,1]$-origamis in $\calH(2g-2)$, the monodromy group must be a subgroup of the alternating group $\Alt_{2g-1}$ of degree $2g-1$.

The well-studied $\SL(2,\R)$ action on translation surfaces (see, for example, the works of Eskin-Mirzakhani~\cite{EM} and Eskin-Mirzakhani-Mohammadi~\cite{EMM}) restricts to an action of $\SL(2,\Z)$ on primitive origamis. The monodromy group of an origami is a weak $\SL(2,\Z)$-invariant in the sense that it is preserved by the action of $\SL(2,\Z)$, but two origamis can have the same monodromy group even if they do not lie in the same $\SL(2,\Z)$-orbit. The action of $\SL(2,\Z)$ also respects the connected components of strata. The classification of $\SL(2,\Z)$-orbits of primitive origamis has only been carried out in the stratum $\calH(2)$ by Hubert-Leli\`evre~\cite{HL} and McMullen~\cite{McM} and remains open in general.

We provide the following analysis of the monodromy groups of the origamis considered in this paper.

\begin{theorem}\label{thm:monodromy}
    All of the minimal $[1,1]$-origamis constructed by Aougab-Menasco-Nieland and their generalisations constructed in Section~\ref{sec:gen-const} are primitive.

    Furthermore, all of the odd genus origamis constructed by Aougab-Mensaco-Nieland have monodromy group isomorphic to the alternating group $\Alt_{2g-1}$. The remaining origamis have monodromy groups isomorphic to the alternating group $\Alt_{2g-1}$ or to $PSL(d,q)$ for some $d$ and $q$ with $\frac{q^{d}-1}{q-1} = 2g-1$ and $\gcd(d,q-1) = 1$.
\end{theorem}

The latter sentence follows from Theorem~\ref{thm:simple} below, however we are not able to determine which of the two groups (alternating or projective special linear) occur. Despite this, and based on computer experiments, we make the following conjecture:

\begin{conjecture}
    All of the $[1,1]$-origamis in Theorem~\ref{thm:monodromy} have monodromy group isomorphic to $\Alt_{2g-1}$.
\end{conjecture}
We also remark that~\cite[Theorem 3.12]{Zm} states that the monodromy group of any origami in the minimal stratum must contain the alternating group once the number of squares is large enough. Our $2g-1$ squares lies below this threshold.

Theorem~\ref{thm:monodromy} alone is not strong enough to allow us to separate $\SL(2,\Z)$-orbits. However, in Section~\ref{sec:double-covers} we prove the following result that implies that the odd genus Aougab-Menasco-Nieland origamis lie in at least two $\SL(2,\Z)$-orbits (for $g\geq 5$).

\begin{theorem}\label{thm:double-covers}
    For all odd $g\geq 3$, exactly $(g-3)!!$ of the $(g-2)!$ odd genus origamis constructed by Aougab-Menasco-Nieland are orientation double covers of quadratic differentials in the stratum $\mathcal{Q}_{\frac{g-1}{2}}(2g-3,-1^{3})$.

    None of the remaining origamis considered in this paper are orientation double covers.
\end{theorem}

In Section~\ref{sec:SL2Z}, we present computational evidence that leads to the following conjecture.

\begin{conjecture}\label{con:SL2Z}
    For odd genus $g\geq 5$, the origamis contructed by Aougab-Menasco-Nieland lie in exactly two $\SL(2,\Z)$-orbits inside the odd component (there is only one orbit in genus three as there is a single origami in the construction).
    
    The odd genus generalisations of Subsection~\ref{subsec:gen-odd-even} lie in a single $\SL(2,\Z)$-orbit inside the even component.

    The even genus Aougab-Menasco-Nieland origamis and the even genus generalisations of Section~\ref{sec:gen-const} lie in a single $\SL(2,\Z)$-orbit in each of the odd and even components.
\end{conjecture}

In the final paragraph of the above conjecture we are saying that the origamis of Aougab-Menasco-Nieland and the generalisations of Section~\ref{sec:gen-const} that lie in the same connected component are in fact contained in the same $\SL(2,\Z)$-orbit.

Investigating which groups can arise as the monodromy group of a primitive minimal $[1,1]$-origami in $\calH(2g-2)$, in Section~\ref{sec:simple}, we prove the following result.

\begin{theorem}\label{thm:simple}
    Let $G$ be the monodromy group of a primitive minimal $[1,1]$-origami in the stratum $\calH(2g-2)$, $g\geq 3$. Then, with the exception of $\PGamL(2,8)$ that occurs in $\odd(8)$, $G$ is simple and isomorphic to $\Alt_{2g-1}$, $PSL(d,q)$ for some $d$ and $q$ with $\frac{q^{d}-1}{q-1} = 2g-1$ and $\gcd(d,q-1) = 1$, or to one of the Mathieu groups $M_{11}$ or $M_{23}$.
\end{theorem}

Group theoretically, this is equivalent to the statement that any primitive permutation group of degree $n$ that is generated by two $n$-cycles whose commutator is also an $n$-cycle is one of the groups listed in the theorem.

Given an $\SL(2,\R)$-orbit closure $\mathcal{M}$ of translation surfaces of genus $g$, one can consider the (orbifold) vector bundle over $\mathcal{M}$ induced by $H^{1}(S,\R)$ where $S$ is a topological surface of genus $g$. This bundle is called the \emph{Hodge bundle} over $\mathcal{M}$. The \emph{Kontsevich-Zorich cocycle} is a dynamical cocycle on the Hodge bundle induced by the action of $\SL(2,\R)$. The \emph{Kontsevich-Zorich monodromy group} is a subgroup of the symplectic group that records the action of $\SL(2,\R)$ on the fibers of the Hodge bundle. The algebraic structure of Kontsevich-Zorich monodromy group is strongly related to the Lyapunov spectrum of the Kontsevich-Zorich cocycle. See, for example, the works of Filip~\cite{Filip} and Matheus-Yoccoz-Zmiaikou~\cite{MYZ}.

In Section~\ref{sec:KZ-calcs}, we investigate the Kontsevich-Zorich monodromy groups of our origamis in genus 3 through 7. In particular, we prove the following theorem.

\begin{theorem}\label{thm:KZ-monodromy}
    Let $g = 3, 5$, or 7. The Zariski closure of the Kontsevich-Zorich monodromy group (restricted to the zero-holonomy subspace) of any of the genus $g$ orientation double covers of Theorem~\ref{thm:double-covers} is $\Sp(g-1,\R)^{2}$.

    Let $g = 4,5,$ or 6. The Zariski closure of the Kontsevich-Zorich monodromy group (restricted to the zero-holonomy subspace) of any origami of genus $g$ considered in this paper that is not an orientation double cover is $\Sp(2g-2,\R)$.
\end{theorem}

These monodromy groups are as large as possible given any restrictions forced by symmetries of the origami. We conjecture that this behaviour persists in all genera.

\subsection{Connection to filling pairs of curves}

A pair of simple closed curves $\{\alpha,\beta\}$ in minimal position on a closed surface $S_{g}$ of genus $g$ is said to be a \emph{filling pair} if its complement $S_{g}\setminus(\alpha\cup\beta)$ is a disjoint union of topological discs. Filling pairs are useful objects of study in low-dimensional topology. For example, Thurston used filling pairs in his construction of pseudo-Anosov homemorphisms of a surface~\cite{FLP}.

Given a $[1,1]$-origami, the core curves of the vertical and horizontal cylinders form a filling pair on the underlying surface. In fact, since the equivalence relation defining the moduli space of translation surfaces corresponds to the action of the mapping class group $\Mod(S_{g})$, a $[1,1]$-origami corresponds to a $\Mod(S_{g})$ orbit of a filling pair on the closed surface $S_{g}$ of genus $g$. A filling pair obtained from a minimal $[1,1]$-origami in the minimal stratum will have geometric intersection number $2g-1$ (the minimum possible for a filling pair) and its complement will be a single topological disc. The above theorems can therefore be translated into the language of mapping class group orbits of minimally intersecting filling pairs. This is also the language used for the main theorem in the work of the first author and Menasco-Nieland discussed above.

\subsection{Upper bounds}

It follows from work of Chang~\cite{Chang}, on counting the mapping class group orbits of minimally intersecting filling pairs, that the number of minimal $[1,1]$-origamis in $\calH(2g-2)$ is bounded above by a function that roughly behaves like $(g-1)(2g-2)!$ (we direct the reader to Chang's paper for the exact details). There are no specialised upper bounds for the individual connected components of $\calH(2g-2)$.

\subsection{Sketch of the proof of Theorem~\ref{thm:AMN}}

We will show that the odd genus origamis constructed by Aougab-Menasco-Nieland are related to each other by a sequence of combinatorial moves. We will prove (Lemma~\ref{lem:key}) that the spin parity of an origami is preserved under these combinatorial moves. It therefore suffices to prove that a single odd genus origami has odd spin. We carry out such a calculation in Proposition~\ref{prop:odd-odd-2}.

This argument also works for the generalised constructions of Section~\ref{sec:gen-const}.

To calculate the spin parity of the even genus Aougab-Menasco-Nieland origamis, we study the steps of the surgery-like construction that takes an odd genus origami to an even genus origami. We determine how the spin parity of the origami changes under each step of the construction. This is carried out in Subsection~\ref{subsec:AMN-even-spin}.

\subsection{Outline of the paper}

In Section~\ref{sec:prelim}, we give an introduction to translation surfaces, origamis, and spin structures. In Section~\ref{sec:AMN}, we recall the constructions of the first author and Menasco-Nieland of the minimal $[1,1]$-origamis in $\calH(2g-2)$. In Section~\ref{sec:gen-const}, we give the generalisations of these constructions that we use to prove Theorem~\ref{thm:gen-const}. In Section~\ref{sec:spin}, we determine the spin parities of the origamis . In Section~\ref{sec:mono} we investigate their monodromy groups. In Section~\ref{sec:double-covers}, we investigate which of the origamis are orientation double covers. In Section~\ref{sec:SL2Z}, we summarise our investigation of the $\SL(2,\Z)$-invariants of the origamis and present a conjecture about their orbits. In Section~\ref{sec:simple}, we prove Theorem~\ref{thm:simple} concerning simple monodromy groups and state some open questions. Finally, in Section~\ref{sec:KZ-calcs}, we provide the calculations required to prove Theorem~\ref{thm:KZ-monodromy}.

\subsection{Acknowledgements}

We thank Tim Burness for very helpful discussions regarding the proof of Theorem~\ref{thm:simple} and related references. We also thank Jayadev Athreya for suggesting that we investigate the Kontsevich-Zorich monodromy groups of these origamis. For the purpose of open access, the authors have applied a Creative Commons Attribution (CC BY) licence to any Author Accepted Manuscript version arising from this submission.

%%%%%%%%%%%%%%%%%%%%%%%%%%%%%%%%%%%%%%%%%%%%%%%%%%%%%%

\section{Preliminaries}\label{sec:prelim}

Here we give the required preliminaries on translation surfaces, origamis, and spin structures. We direct the reader to the surveys of Forni-Matheus~\cite{FM} and Zorich~\cite{Z2} for more details.

\subsection{Translation surfaces}

For $g\geq 2$, we define the space $\calH_{g}$ (which we will often abbreviate to just $\calH$) to be the {\em moduli space of translation surfaces} of genus $g$. That is, $\calH$ is the quotient by the action of the mapping class group on the set of pairs $(X,\omega)$ where $X$ is a closed connected Riemann surface of genus $g$ and $\omega$ is a non-zero holomorphic 1-form on $X$, also called an \emph{Abelian differential}. Locally, $\omega = f(z)dz$ for some holomorphic function $f$.

Given an Abelian differential $\omega$ on a Riemann surface $X$, let $\Sigma$ be the set of zeros of $\omega$. On $X\setminus\Sigma$, contour integration gives a collection of charts to $\C$ with transition maps given by translations $z\mapsto z+c$. Equivalently, given a finite collection of polygons in $\C$ with parallel sides identified by translation, one can define a Riemann surface structure on the surface $X$ obtained from the quotient of these polygons by the side identifications. The local pullback of ${\rm d}z$ will give rise to a well-defined Abelian differential on $X$.

By the Riemann-Roch theorem, the sum of the orders of the zeros of an Abelian differential on a Riemann surface of genus $g$ is equal to $2g-2$ and this data can be used to stratify $\calH$. The stratum $\calH(k_{1},\ldots,k_{n})\subset\calH$, with $k_{i}\geq 1$ and $\sum k_{i} = 2g-2$, is the subset of $\calH$ consisting of Abelian differentials with $n$ distinct zeros of orders $k_{1},\ldots,k_{n}$. The stratum $\calH(2g-2)$ is known as the \emph{minimal stratum}.

Each stratum $\calH(k_{1},\ldots,k_{n})$ is a complex orbifold of dimension $2g+n-1$ (recall that local coordinates are given by the period map to $H^{1}(X,\Sigma;\C)$). The strata of $\calH$ may be disconnected and Kontsevich-Zorich~\cite{KZ} completely classified the connected components using the notions of hyperellipticity and spin parity which we will define in the following subsections. The relevant part of the the classification for this paper is that
\begin{itemize}
\item[1.] $\calH(2)$ is connected and identified with its hyperelliptic component;
\item[2.] $\calH(4)$ has two connected components: a hyperelliptic component and a component corresponding to odd spin parity; and
\item[3.] for $g\geq 4$, $\calH(2g-2)$ has three connected components: a hyperelliptic component, a component corresponding to odd spin parity, and a component corresponding to even spin parity.
\end{itemize}

\subsection{Hyperellipticity}

A translation surface $(X,\omega)$ is said to be {\em hyperelliptic} if there exists an isometric involution $\tau:X\to X$ that induces a ramified double cover $\pi:X\to S_{0,2g+2}$ from $X$ to the $(2g+2)$-times punctured sphere. There exists a connected component of the stratum $\calH(2g-2)$, called the hyperelliptic component and denoted $\hyp(2g-2)$, consisting entirely of hyperelliptic translation surfaces.

\subsection{Spin structures and spin parity}

Here we give the complex analytic definition of a spin structure on a Riemann surface. See the work of Atiyah~\cite{A} for a discussion of its relation to the definition of a spin structure in terms of $\Spin(n)$-bundles.

Let $\Pic(X)$ be the Picard group of the Riemann surface $X$. A {\em spin structure} on $X$ is a choice of divisor class $D\in \Pic(X)$ such that
\[2D = K_{X},\]
where $K_{X}$ is the canonical class of $X$. Let $L$ be a line bundle corresponding to the divisor class $D$, and let $\Gamma(X,L)$ be the space of holomorphic sections of the line bundle $L$ on $X$. We define the {\em parity of the spin structure} $D$ to be
\[\dim\Gamma(X,L)\text{ mod } 2.\]

Given a translation surface $(X,\omega)\in\calH(2k_{1},\ldots,2k_{n})$, let $P_{i}$ be the zero of $\omega$ of order $k_{i}$. The divisor
\[Z_{\omega} = 2k_{1}P_{1}+\cdots+2k_{n}P_{n}\]
represents the canonical class $K_{X}$. As such, we have a canonical choice of spin structure on $X$ given by the divisor class
\[D_{\omega} = [k_{1}P_{1}+\cdots+k_{n}P_{n}].\]
Atiyah \cite{A} and Mumford \cite{M} proved that the parity of a spin structure is invariant under continuous deformation, from which it follows that the parity of the canonical spin structure of a translation surface is constant on each connected component of the stratum. We will say that a non-hyperelliptic connected component of $\calH(2k_{1},\ldots,2k_{n})$ has {\em even} (resp. {\em odd}) spin parity and denote it by $\even(2k_{1},\ldots,2k_{n})$ (resp. $\odd(2k_{1},\ldots,2k_{n}))$ if the parity of $D_{\omega}$ is 0 (resp. 1) for any (hence all) translation surface(s) in the component.

There is an equivalent definition of spin parity given in the language of quadratic forms on the homology of the surface that allows us to more easily compute the spin parity of a translation surface.

Indeed, an Abelian differential $\omega$ on $X$ determines a flat metric on $X$ with cone-type singularities. Moreover, this metric has trivial holonomy, and away from the zeros of $\omega$ there is a well-defined horizontal direction. We can therefore define the index, $\ind(\gamma)$, of a simple closed curve $\gamma$ on $X$ that avoids the singularities to be the degree of the Gauss map of $\gamma$. That is, $\ind(\gamma)$ is the integer such that the total change of angle between the vector tangent to $\gamma$ and the vector tangent to the horizontal direction determined by $\omega$ is $2\pi\cdot \ind(\gamma)$.

Given $(X,\omega)\in\calH(2k_{1},\ldots,2k_{n})$, we define a function $\Phi:H_{1}(X,\Z_{2})\to\Z_{2}$ by
\[\Phi([\gamma]) = \ind(\gamma)+1\text{ mod } 2,\]
where $\gamma$ is a simple closed curve and extend to a general homology class by linearity. It can be checked that this function is well-defined.

The function $\Phi$ can be shown to be a quadratic form on $H_{1}(X,\Z_{2})$, by which we mean
\[\Phi(a+b) = \Phi(a)+\Phi(b)+a\cdot b,\]
where $a\cdot b$ denotes the standard symplectic intersection form on $H_{1}(X,\Z_{2})$. Now given a choice of representatives $\{[\alpha_{i}],[\beta_{i}]\}_{i = 1}^{g}$ of a symplectic basis for $H_{1}(X,\Z_{2})$ (i.e., with $[\alpha_{i}]\cdot[\beta_{j}] = \delta_{ij}$ and $0 = [\alpha_{i}]\cdot[\alpha_{j}] = [\beta_{i}]\cdot[\beta_{j}]$), we define the {\em Arf invariant} of $\Phi$ to be
\[\sum_{i = 1}^{g}\Phi([\alpha_{i}])\cdot\Phi([\beta_{i}])\text{ mod }2 = \sum_{i = 1}^{g}(\ind(\alpha_{i})+1)(\ind(\beta_{i})+1)\text{ mod }2.\]
Arf \cite{Arf} proved that this number is independent of the choice of symplectic basis and Johnson \cite{Jo} showed that quadratic forms on $H_{1}(X,\Z_{2})$ are in one-to-one correspondence with spin structures on $X$. Moreover, Johnson proved that the value of the Arf invariant of $\Phi$ coincides with the parity of the canonical spin structure determined by $\omega$. We will make use of this formula when we calculate the parity of spin structures later in the paper.

\subsection{Origamis, monodromy groups and $\SL(2,\Z)$-orbits}\label{subsec:ori-prep}

An \emph{origami} (also known as a \emph{square-tiled surface}) is a translation surface obtained from a collection of unit squares in $\C$ by identifying by translation left-hand sides with right-hand sides and top sides with bottom sides. See Figure~\ref{fig:origami-examples} for two origamis in $\calH(4)$. Sides with the same label are identified by translation. Ignore the smaller label in the interiors of the squares for now. An Euler characteristic argument shows that an origami in $\calH(2g-2)$ must be built from at least $2g-1$ squares.

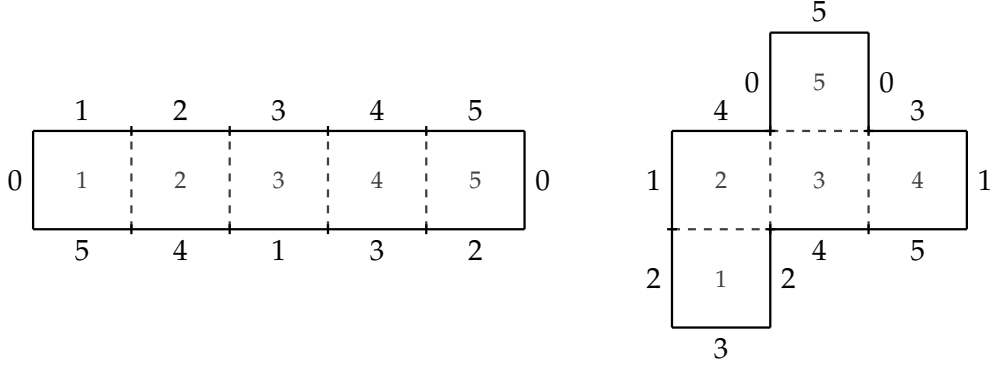
\begin{figure}[t]
    \centering
    \begin{tikzpicture}[scale = 1.3, line width = 0.3mm]
    \draw (0,0)--(0,1);
    \draw (0,1)--(5,1);
    \draw (5,1)--(5,0);
    \draw (5,0)--(0,0);
    \foreach \i in {1,2,...,4}{
        \draw [dashed, color = darkgray] (\i,0)--(\i,1);
        \draw (\i,0.05)--(\i,-0.05);
        \draw (\i,1.05)--(\i,0.95);
    }
    \draw (0,0.5) node[left] {$0$};
    \draw (5,0.5) node[right] {$0$};
    \foreach \i in {1,...,5}{
        \draw (\i-0.5,1) node[above] {$\i$};
        \draw (\i-0.5,0.5) node[color = darkgray] {\footnotesize $\i$};
    }
    \draw (0.5,0) node[below] {$5$};
    \draw (1.5,0) node[below] {$4$};
    \draw (2.5,0) node[below] {$1$};
    \draw (3.5,0) node[below] {$3$};
    \draw (4.5,0) node[below] {$2$};
    %%%%%%%%%%%%%%%%%%%%%%%%%%%%%%%%%%%%%%%%%
    \draw (6.5,-1)--(6.5,1);
    \draw (6.5,1)--(7.5,1);
    \draw (7.5,1)--(7.5,2);
    \draw (7.5,2)--(8.5,2);
    \draw (8.5,2)--(8.5,1);
    \draw (8.5,1)--(9.5,1);
    \draw (9.5,1)--(9.5,0);
    \draw (9.5,0)--(7.5,0);
    \draw (7.5,0)--(7.5,-1);
    \draw (7.5,-1)--(6.5,-1);
    \foreach \i in {0,1}{
        \draw [dashed, color = darkgray] (6.5+\i,0+\i)--(7.5+\i,0+\i);
        \draw (6.45+\i,0+\i)--(6.55+\i,0+\i);
        \draw (7.45+\i,0+\i)--(7.55+\i,0+\i);
        \draw [dashed, color = darkgray] (7.5+\i,0)--(7.5+\i,1);
        \draw (7.5+\i,0.05)--(7.5+\i,-0.05);
        \draw (7.5+\i,1.05)--(7.5+\i,0.95);
    }
    \draw (7.5,1.5) node[left] {$0$};
    \draw (8.5,1.5) node[right] {$0$};
    \draw (6.5,0.5) node[left] {$1$};
    \draw (9.5,0.5) node[right] {$1$};
    \draw (6.5,-0.5) node[left] {$2$};
    \draw (7.5,-0.5) node[right] {$2$};
    \draw (7,1) node[above] {$4$};
    \draw (7,-1) node[below] {$3$};
    \draw (8,2) node[above] {$5$};
    \draw (8,0) node[below] {$4$};
    \draw (9,1) node[above] {$3$};
    \draw (9,0) node[below] {$5$};
    \draw (7,-0.5) node[color = darkgray] {\footnotesize $1$};
    \draw (7,0.5) node[color = darkgray] {\footnotesize $2$};
    \draw (8,0.5) node[color = darkgray] {\footnotesize $3$};
    \draw (9,0.5) node[color = darkgray] {\footnotesize $4$};
    \draw (8,1.5) node[color = darkgray] {\footnotesize $5$};
    \end{tikzpicture}
    \caption{Two origamis in $\calH(4)$. The side labels indicate the identifications by translation. The smaller numbers are square numberings.}
    \label{fig:origami-examples}
\end{figure}

As mentioned in the introduction, a \emph{cylinder} in an origami is a maximally embedded flat annulus, or, equivalently, a maximal collection of freely homotopic closed geodesics in the (singular) flat metric determined by $\omega$ that avoid the zeros of $\omega$. The origami on the left of Figure~\ref{fig:origami-examples} has a single horizontal cylinder running between the sides labelled $0$. It also has a single vertical cylinder running vertically through all of the squares in the origami. We call an origami a \emph{$[1,1]$-origami} if it simultaneously has a single horizontal cylinder and a single vertical cylinder. The origami on the right of Figure~\ref{fig:origami-examples} has three horizontal cylinders running between the sides labelled 0, 1, and 2, respectively. It does, however, have a single vertical cylinder. It follows from the work of the third author~\cite{J1} that $[1,1]$-origamis in the odd or even components of $\calH(2g-2)$ can be built using the theoretical minimum of $2g-1$ squares. We call such origamis \emph{minimal $[1,1]$-origamis} since they realise the minimum number of squares for the ambient stratum. The same work of the third author also proves that a $[1,1]$-origami in the hyperelliptic component of $\calH(2g-2)$ instead requires at least $4g-4$ squares. This is why the theorems of this work require genus at least three since the minimal stratum in genus two is identified with its hyperelliptic component, so there do not exist any $[1,1]$-origamis in $\calH(2)$ built from $2g-1 = 3$ squares. It is also why the only connected components that appear in the main theorems are the odd and even components.

An origami constructed from $n$ unit squares can also be described by a pair of permutations ($\sigma$,$\tau$) with each permutation lying in the symmetric group $\Sym_{n}$ of degree $n$. The permutation $\sigma$ describes the identifications of right and left sides while the permutation $\tau$ describes the identifications of the top and bottom sides. Firstly, we number the squares from 1 to $n$. We then define $\sigma$ to be the element of $\Sym_{n}$ such that $\sigma(i) = j$ if and only if the right-hand side of the square labelled by $i$ is identified with the left-hand side of the square labelled by $j$. Symmetrically, we define $\tau$ to be the permutation satisfying $\tau(i) = j$ if and only if the top side of the square labelled by $i$ is glued to the bottom side of the square labelled by $j$. For example, the origamis in Figure~\ref{fig:origami-examples} can be realised by the pairs $O = ((1,2,3,4,5),(1,3,4,2,5))$ and $O' = ((1)(2,3,4)(5),(1,2,3,5,4))$, respectively. Since a different labelling of the squares could produce a different pair of permutations, an origami corresponds to a pair $(\sigma,\tau)$ considered up to simultaneous conjugation of $\sigma$ and $\tau$. However, we will abuse notation in this article and denote an origami simply by the pair $(\sigma,\tau)$.

\begin{remark}[Order of multiplication]
We make the convention that permutations are multiplied right to left so that $(\rho_{1}\rho_{2})(i) = \rho_{1}(\rho_{2}(i))$. For example, $(1,2)(2,3) = (1,2,3)$.
\end{remark}

An origami $O = (\sigma,\tau)$ lies in the stratum $\calH(k_{1},\ldots,k_{d})$ if when writing the commutator $[\sigma,\tau] = \sigma\tau\sigma^{-1}\tau^{-1}$ as a product of disjoint non-trivial cycles we obtain $d$ cycles of lengths $k_{1}+1,\ldots,k_{d}+1$. Observe that for $O$ and $O'$, the origamis in Figure~\ref{fig:origami-examples}, we have the commutators equal to $(1,3,2,5,4)$ and $(1,5,4,2,3)$, respectively. In particular, both origamis lie in $\calH(4)$, as claimed.

Observe that an origami $O = (\sigma,\tau)$ will be a minimal $[1,1]$-origami in the stratum $\calH(2g-2)$ if and only if $\sigma$ and $\tau$ are both $(2g-1)$-cycles whose commutator is also a $(2g-1)$-cycle. There is not a straightforward way to determine the spin parity of an origami from the permutation description.

The \emph{monodromy group} $\Mon(O)$ of an $n$-squared origami $O = (\sigma,\tau)$ is defined to be the subgroup $\langle\sigma,\tau\rangle\leqslant\Sym_{n}$ generated by $\sigma$ and $\tau$. Notice that the monodromy group of a minimal $[1,1]$-origami in $\calH(2g-2)$ will be a subgroup of the alternating group $\Alt_{2g-1}$ of degree $2g-1$.

The permutation group theoretic properties of the monodromy group are naturally related to certain topological properties of the underlying origami. For example, the monodromy group will be transitive since, by our assumption, the origami is connected. Furthermore, an origami is said to be \emph{primitive} if it is not a proper cover of another origami other than the square-torus (of which all origamis are a cover). This topological notion is related to primitivity in the permutation group sense. Indeed, a partition $B = \{\Delta_{1},\ldots,\Delta_{k}\}$ of $\{1,2,\ldots,n\}$ is said to be a \emph{block system} for a permutation group $G\leqslant \Sym_{n}$ if for all $\Delta\in B$ and all $g\in G$ either $g(\Delta) = \Delta$ or $g(\Delta)\cap\Delta=\varnothing$. A transitive permutation group $G\leqslant\Sym_{n}$ is then said to be \emph{primitive} if the only block systems that exist for $G$ are the trivial block systems $B = \{\{1,2,\ldots,n\}\}$ and $B' = \{\{1\},\{2\},\ldots,\{n\}\}$. It can be shown that an origami is primitive if and only if its monodromy group is primitive as a permutation group. We direct the reader to the thesis of Zmiaikou~\cite{Zm} for more on the permutation group descriptions of origamis.

We also remark that the definitions of the stratum and monodromy group of an origami are well-defined in the sense that they are invariant under simultaneous conjugation of $\sigma$ and $\tau$.

The group $\SL(2,\Z)$ acts on origamis by its natural action on the plane. Indeed, up to cutting and pasting, unit squares are mapped to unit squares and parallel sides are sent to parallel sides. Consider $\SL(2,\Z) = \langle T, S\rangle$, where
\[T = \begin{bmatrix}
1 & 1 \\
0 & 1
\end{bmatrix}
\,\,\,\text{and}\,\,\,
S = \begin{bmatrix}
1 & 0 \\
1 & 1
\end{bmatrix}.\]
The matrix $T$ acts by horizontally shearing the origami to the right while the matrix $S$ acts by vertically shearing the origami upwards. It can then be checked that
\[T((\sigma,\tau)) = (\sigma,\tau\sigma^{-1})\,\,\,\,\,\,\,\text{and}\,\,\,\,\,\,\,S((\sigma,\tau)) = (\sigma\tau^{-1},\tau).\]

In particular, it follows that the number of squares, the stratum, the primitivity, and the monodromy group of an origami are invariant under the action of $\SL(2,\Z)$. It can also be argued that the action of $\SL(2,\Z)$ preserves the connected components of strata. As such, it makes sense to discuss the $\SL(2,\Z)$-orbits of primitive origamis in a given connected component of a stratum.

\subsection{Kontsevich-Zorich monodromy groups}\label{subsec:KZ-prep}

Given a translation surface $(X,\omega)$, we define $\Aff(X,\omega)$ to be its group of affine diffeomorphisms; that is, $\Aff(X,\omega)$ is the group of self-diffeomorphisms of $X$ that preserve the zero set $\Sigma$ of $\omega$ and are affine in the translation charts given by $\omega$ on $X\setminus\Sigma$. The differential of any such map is a matrix in $\SL(2,\R)$ and we get a group homomorphism $D:\Aff(X,\omega)\to\SL(2,\R)$. The kernel of $D$ is the group $\Aut(X,\omega)$ of (translation) automorphisms of $(X,\omega)$, while the image of $D$ is the Veech group $\SL(X,\omega)$ (the $\SL(2,\R)$ stabiliser of $(X,\omega)$). Moreover, these groups fit into a short exact sequence
\[1\to\Aut(X,\omega)\to\Aff(X,\omega)\to\SL(X,\omega)\to 1.\]

In our case, since our origamis lie in the minimal stratum $\calH(2g-2)$, we have that $\Aut(X,\omega) = 1$ and hence $\Aff(X,\omega)\cong\SL(X,\omega)$, which will be some finite index subgroup of $\SL(2,\Z)$.

Recall that the fiber of the Hodge bundle is $H^{1}(X,\R)$. By Poincar\'e duality, we will equivalently consider $H_{1}(X,\R)$. The action of $\Aff(X,\omega)$ on the fibers then induces a representation
\[\widetilde{\alpha}:\Aff(X,\omega)\to\Sp(H_{1}(X,\R))\cong\Sp(2g,\R),\]
since the action preserves the symplectic intersection form on homology.

All origamis are branched covers of the unit torus $\mathbb{T}^{2}$. That is, every origami $(X,\omega)$ admits a map $p:(X,\omega)\to(\mathbb{T}^{2},dz)$ branched only at $0\in\mathbb{T}^{2}$ and with $\omega = p^{*}(dz)$. This map gives rise to a splitting of $H_{1}(X,\R)$:
\[H_{1}(X,\R) = H_{1}^{st}(X)\oplus H_{1}^{(0)}(X),\]
where 
\[H_{1}^{st}(X) = p_{*}^{-1}(H_{1}(\mathbb{T}^{2},\R))\,\,\,\text{and}\,\,\,H_{1}^{(0)}(X) = \ker p_{*} = \left\{\gamma\in H_{1}(X,\R)\,:\,\int_{\gamma}\omega = 0\right\}.\]
The latter is called the \emph{zero-holonomy subspace}. The splitting is preserved by the action of $\Aff(X,\omega)$. The subspace $H_{1}^{st}(X)$ is 2-dimensional and dual to the span of the real part $\Re(\omega)$ and imaginary part $\Im(\omega)$ in $H^{1}(X,\R)$. Moreover, it can be shown that any $f\in\Aff(X,\omega)$ acts on $H_{1}^{st}(X)$ by $D(f)\in\SL(X,\omega)$. For this reason, $H_{1}^{st}(X)$ is called the \emph{tautological plane}. Since the action of $\Aff(X,\omega)$ on the tautological plane is fully understood, we define the Kontsevich-Zorich monodromy group of an origami $(X,\omega)$ to be
\[\alpha(\Aff(X,\omega)):= \widetilde{\alpha}(\Aff(X,\omega))|_{H_{1}^{(0)}(X)}.\]

%%%%%%%%%%%%%%%%%%%%%%%%%%%%%%%%%%%%%%%%%%%%%%%%%%%%%%

\section{The construction of Aougab-Menasco-Nieland}\label{sec:AMN}

In this section, we describe the constructions of the $[1,1]$-origamis given by the first author with Menasco-Nieland~\cite{AMN}. In the remainder of the paper, for convenience, we will refer to these origamis as \textit{AMN origamis}.

\subsection{Odd genus}\label{subsec:AMN-odd}

Let $\sigma_{g} := (1,2,\ldots,2g-1)$. Up to simultaneous conjugation, we can assume that the $[1,1]$-origamis are of the form $(\sigma_{g},\tau)$, with $\tau$ another $(2g-1)$-cycle. For the origami to lie in $\calH(2g-2)$ we require that the commutator $[\sigma_{g},\tau] = \sigma_{g}\tau\sigma_{g}^{-1}\tau^{-1}$ is also a $(2g-1)$-cycle. It will be easier in the construction below to demonstrate that $[\sigma_{g}^{-1},\tau^{-1}]$ is a $(2g-1)$-cycle.

The permutations $\tau$ are constructed as follows. We begin with
\[\tau = \begin{pmatrix}
\bullet & \bullet & 1 & \bullet & \bullet & \cdots & \bullet & \bullet \\
1 & 2 & 3 & 4 & 5 & \cdots & 2g-2 & 2g-1
\end{pmatrix}\]
with the dots to be filled in later. We now place the pair $(3,2)$ in the top row of $\tau$ above some $(2j,2j+1)$. The next pair placed in the top row is then $(2j+1,2j)$ which is placed above some $(2k,2k+1)$. This process continues until the final pair $(2i+1,2i)$ is placed above $(1,2)$ which completes the top row of $\tau$. The construction guarantees that $\tau$ is a $(2g-1)$-cycle. Indeed, if the chosen pairs are $(2k_{i},2k_{i}+1)$ for $1\leq i \leq g-2$, then we have
\[\tau = (1,3,2k_{1},2k_{2}+1,2k_{3},...,2k_{g-3}+1,2k_{g-2},2,2k_{1}+1,2k_{2},2k_{3}+1,...,2k_{g-3},2k_{g-2}+1)\]
and
\[\sigma_{g}^{-1}\tau^{-1}\sigma_{g}\tau = (1,\rho^{2}(1),\rho^{4}(1),\ldots,\rho^{-1}(1),2g-1,2g-3,\ldots,3),\]
where $\rho = \tau^{-1}\sigma_{g}\tau$ is the $(2g-1)$-cycle obtained by considering the top row of $\tau$ as a $(2g-1)$-tuple. We obtain $(g-2)!$ minimal $[1,1]$-origamis in $\calH(2g-2)$~\cite[Theorem 4.1]{AMN}.

\begin{example}\label{ex:ori:AMNodd}
    For $g = 5$, the process could unfold as:
\[\tau = \begin{pmatrix}
\bullet & \bullet & 1 & \bullet & \bullet & \bullet & \bullet & \bullet & \bullet \\
1 & 2 & 3 & 4 & 5 & 6 & 7 & 8 & 9
\end{pmatrix}.\]
Now place $(3,2)$ above $(4,5)$ to obtain
\[\tau = \begin{pmatrix}
\bullet & \bullet & 1 & 3 & 2 & \bullet & \bullet & \bullet & \bullet \\
1 & 2 & 3 & 4 & 5 & 6 & 7 & 8 & 9
\end{pmatrix}.\]
Place $(5,4)$ above $(6,7)$ giving
\[\tau = \begin{pmatrix}
\bullet & \bullet & 1 & 3 & 2 & 5 & 4 & \bullet & \bullet \\
1 & 2 & 3 & 4 & 5 & 6 & 7 & 8 & 9
\end{pmatrix}.\]
Now place $(7,6)$ above $(8,9)$, which forces $(9,8)$ to go above $(1,2)$, and we finish with
\[\tau = \begin{pmatrix}
9 & 8 & 1 & 3 & 2 & 5 & 4 & 7 & 6 \\
1 & 2 & 3 & 4 & 5 & 6 & 7 & 8 & 9
\end{pmatrix} = (1,3,4,7,8,2,5,6,9).\]
This gives the $[1,1]$-origami $O = (\sigma_{5}, \tau)$ with $\sigma_{5}^{-1}\tau^{-1}\sigma_{5}\tau = (1,2,4,6,8,9,7,5,3)$ and so $O$ is indeed in the stratum $\calH(8)$. As a surface, $O$ can be realised as in Figure~\ref{fig:AMN-origami}. Note that the labels correspond exactly to the rows of $\tau$.
\end{example}

We discuss a topological realisation of this construction from the perspective of filling pairs in Appendix~\ref{app:alg}.

\begin{figure}[t]
    \centering
    \begin{tikzpicture}[scale = 1.3, line width = 0.3mm]
    \draw (0,0)--(0,1);
    \draw (0,1)--(9,1);
    \draw (9,1)--(9,0);
    \draw (9,0)--(0,0);
    \foreach \i in {1,2,...,8}{
        \draw [dashed, color = darkgray] (\i,0)--(\i,1);
        \draw (\i,0.05)--(\i,-0.05);
        \draw (\i,1.05)--(\i,0.95);
        \draw (\i-0.5,0.5) node[color = darkgray] {\footnotesize ${\i}$};
    }
    \draw (8.5,0.5) node[color = darkgray] {\footnotesize ${9}$};
    \draw (0,0.5) node[left] {$0$};
    \draw (9,0.5) node[right] {$0$};
    \foreach \i in {1,2,...,9}{
        \draw (\i-0.5,1) node[above] {$\i$};
    }
    \foreach \i in {9,8}{
        \draw (9-\i+0.5,0) node[below] {$\i$};
    }
    \draw (2.5,0) node[below] {$1$};
    \foreach \i in {3,5,7}{
        \draw (\i+0.5,0) node[below] {$\i$};
    }
    \foreach \i in {2,4,6}{
        \draw (\i+2.5,0) node[below] {$\i$};
    }
    \end{tikzpicture}
\caption{A realisation of the origami $O$ from Example~\ref{ex:ori:AMNodd}.}\label{fig:AMN-origami}
\end{figure}
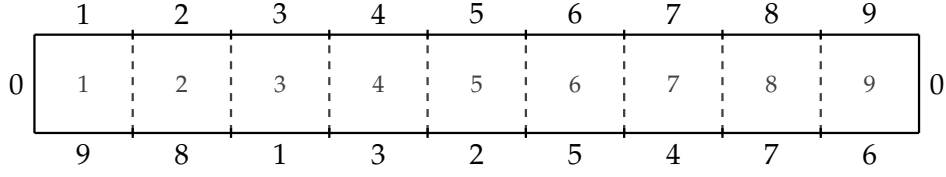

\begin{figure}[b]
    \centering
    \begin{tikzpicture}[scale = 1.3, line width = 0.3mm]
    \foreach \j in {0,-2}{
        \draw (0,\j)--(0,\j+1);
        \draw (0,\j+1)--(11,\j+1);
        \draw (11,\j+1)--(11,\j);
        \draw (11,\j)--(0,\j);
        \foreach \i in {1,2,...,10}{
            \draw [dashed, color = darkgray] (\i,\j)--(\i,\j+1);
            \draw (\i,\j+0.05)--(\i,\j-0.05);
            \draw (\i,\j+1.05)--(\i,\j+0.95);
            \draw (\i-0.5,\j+0.5) node[color = darkgray] {\footnotesize ${\i}$};
        }
        \draw (10.5,\j+0.5) node[color = darkgray] {\footnotesize ${11}$};
        \draw (0,\j+0.5) node[left] {$0$};
        \draw (11,\j+0.5) node[right] {$0$};
        \foreach \i in {1,2,...,11}{
            \draw (\i-0.5,\j+1) node[above] {$\i$};
        }
        \foreach \i in {9,8}{
            \draw (9-\i+0.5,\j) node[below] {$\i$};
        }
        \draw (2.5,\j) node[below] {$1$};
        \foreach \i in {3,5}{
            \draw (\i+0.5,\j) node[below] {$\i$};
        }
        \foreach \i in {2,4,6}{
            \draw (\i+2.5,\j) node[below] {$\i$};
        }
        \draw (9.5,\j) node[below] {$11$};
    }
    \draw (7.5,0) node[below] {$7$};
    \draw (10.5,0) node[below] {$10$};
    \draw (7.5,-2) node[below] {$10$};
    \draw (10.5,-2) node[below] {$7$};
    \end{tikzpicture}
\caption{The surgery-like construction used to build the even genus AMN origami from Example~\ref{ex:ori:AMNeven}.}\label{fig:even-AMN-example}
\end{figure}
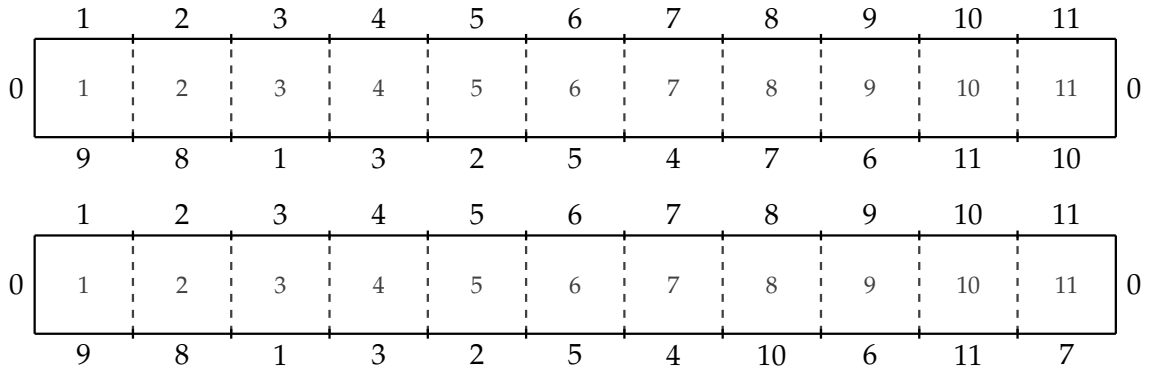

\subsection{Even genus}\label{subsec:AMN-even}

To construct $\tau$ in even genus $g\geq 4$, we begin with a permutation $\eta$ that was produced via the odd genus construction for $g-1$. The top row of $\eta$ is placed into the top row of $\tau$ above the positions they are sent to by $\eta$. Note that we have the positions above $2g-2$ and $2g-1$ in $\tau$ still blank. A choice of $k\in\{3,5,\ldots,g-3\}\setminus\{\eta^{-1}(1)\}$ is made. We then move $k$ in the top row above $2g-1$, then place $2g-1$ above $2g-2$, and $2g-2$ above $\eta(k)$. This completes $\tau$. It can then be demonstrated that $\tau$ is a $(2g-1)$-cycle with $[\sigma_{g},\tau]$ also a $(2g-1)$-cycle~\cite[Subsection 4.4]{AMN}.

\begin{example}\label{ex:ori:AMNeven}
For $g=6$, we can choose $\eta$ to be the odd-genus example and add two more numbers at the end:
\[\tau = \begin{pmatrix}
9 & 8 & 1 & 3 & 2 & 5 & 4 & 7 & 6 & \bullet & \bullet\\
1 & 2 & 3 & 4 & 5 & 6 & 7 & 8 & 9 & 10 & 11
\end{pmatrix}.\]
Next, we choose $k=7$, which is possible since $\eta(7)\ne 1$. Then we place $7$ above $11$, replace the old $7$ with $10$, and place $11$ above $10$:
\[\tau = \begin{pmatrix}
9 & 8 & 1 & 3 & 2 & 5 & 4 & 10 & 6 & 11 & 7\\
1 & 2 & 3 & 4 & 5 & 6 & 7 & 8 & 9 & 10 & 11
\end{pmatrix} = (1,3,4,7,11, 10, 8,2,5,6,9).\]
Thus, in the cycle notation, we have added $11$ and $10$ in between $7$ and $\eta(7)=8$. This gives the $[1,1]$-origami $O=(\sigma_6,\tau)$ with $\sigma_6^{-1}\tau^{-1}\sigma_6\tau=(1,2,4,9,7,8,11,6,10,5,3)$, so $O$ is in the stratum $\calH(10)$.

We can also realise this process geometrically with a surgery-like construction. See Figure~\ref{fig:even-AMN-example}. We start by taking the origami from Figure~\ref{fig:AMN-origami}. To this, we add two squares to the right with tops labelled 10 and 11, and bottoms labelled 11 and 10, respectively. On the bottom of the origami, we then swap the labels 7 and 10.
\end{example}

%\subsection{Initial computer investigations of spin parity}\label{subsec:AMN-spin}

%The above construction of the first author with Menasco-Nieland gives rise to factorially many (in the genus) minimal $[1,1]$-origamis in $\calH(2g-2)$. As discussed in the introduction, it is natural to ask in which of the connected components of $\calH(2g-2)$ they lie. By previous work of the third author~\cite{J1}, the origamis cannot lie in the hyperelliptic component of $\calH(2g-2)$ and so they lie in either the odd component $\odd(2g-2)$ or the even component $\even(2g-2)$. Recall that these components are determined by spin parity.

%Preliminary computer investigations suggested that all of the origamis arising from the odd genus construction have odd spin parity and so lie in $\odd(2g-2)$. For even genus, the investigations suggested that the origamis constructed above could have odd or even spin parity depending on the choices of $\eta$ and $k$ used in the construction. See Appendix~\ref{subapp:AMN-spin-code} for the code and output. In Section~\ref{sec:spin}, we will prove that indeed all of the odd genus origamis have odd spin parity, partially completing the proof of Theorem~\ref{thm:AMN}. We will provide generalisations of the odd genus construction that work for both odd and even genus and, moreover, will allow us to construct factorially many origamis in each of the odd and even components for all $g\geq 3$. We handle the even genus Aougab-Menasco-Nieland origamis in Subsection~\ref{subsec:AMN-even-spin}.

%%%%%%%%%%%%%%%%%%%%%%%%%%%%%%%%%%%%%%%%%%%%%%%%%%%%%%

\section{Generalised constructions}\label{sec:gen-const}

In this section, we give the generalisations of the Aougab-Menaso-Nieland constructions required for Theorem~\ref{thm:gen-const}. We prove that the constructed origamis are indeed $[1,1]$-minimal origamis in $\calH(2g-2)$. That the constructions have the claimed spin parities will be argued in Section~\ref{sec:spin}. In each case, the idea is to prescribe part of the permutation $\tau$ in order to guarantee the required spin structure, before allowing pairs $(2j,2j+1)$ to be placed as in the original construction of the odd genus AMN origamis.

\subsection{Odd genus even spin}\label{subsec:gen-odd-even}

In this subsection, we will give a construction that resembles the original construction of the first author and Menasco-Nieland. We intend the constructed origamis to lie in $\even(2g-2)$ for $g\geq 5$ odd. Note that there is no even component in $\calH(4)$ which is why $g = 3$ is not considered.

As above, fix $g\geq 5$ odd and let $\sigma_{g} = (1,2,\ldots,2g-1)$. We will construct $\tau$ by filling in the blanks in the following permutation:
\[\tau = \begin{pmatrix}
9 & 6 & 1 & 8 & 7 & 5 & 4 & \bullet & \bullet & \cdots & \bullet & \bullet \\
1 & 2 & 3 & 4 & 5 & 6 & 7 & 8 & 9 & \cdots & 2g-2 & 2g-1
\end{pmatrix}.\]
We have fixed more of the initial permutation here which (as we will see in Section~\ref{sec:spin}) forces the even spin parity we desire. We now complete the permutation by adding the pairs $(3,2),(11,10),\ldots,(2g-1,2g-2)$ to the top row of $\tau$. We place $(3,2)$ above some $(2j,2j+1)$. As long as $(2j,2j+1)\neq (8,9)$, we will not close up a cycle in $\tau$. This means that we have $(2g-1-9)/2 = g-5$ choices for the pair $(2j,2j+1)$. Now we place $(2j+1,2j)$ above some $(2k,2k+1)\neq(8,9)$. We have $g-6$ choices. We continue in this vain until we are forced to place the final pair $(2i+1,2i)$ above $(8,9)$ at which point we close up the cycle that is $\tau$. That is, if the chosen pairs are $(2k_{i},2k_{i}+1)$ for $1\leq i \leq g-5$, then we have
\[\tau = (1,3,2k_{1},2k_{2}+1,2k_{3},\ldots,2k_{g-5}+1,8,4,7,5,6,2,2k_{1}+1,2k_{2},\ldots,2k_{g-5},9).\]

Letting $\rho$ be the $(2g-1)$-cycle given by the top row of $\tau$, we find that
\[\sigma_{g}^{-1}\tau^{-1}\sigma_{g}\tau = (1,\rho^{2}(1),\rho^{4}(1),\ldots,\tau^{-1}(2g-1),8,6,2g-1,2g-3,\ldots,9,5,3)\]
is a $(2g-1)$-cycle as required.

So we have produced $(g-5)!$ many $(2g-1)$-cycles $\tau$ such that the origami $O = (\sigma_{g},\tau)$ lies in the stratum $\calH(2g-2)$. We will argue that each such $(\sigma_{g},\tau)$ pair gives rise to a distinct origami. Indeed, recall that origamis are given by pairs of permutations $(\sigma,\tau)$ up to simultaneous conjugation of $\sigma$ and $\tau$. Observe that $\eta\sigma_{g}\eta^{-1} = \sigma_{g}$ if and only if $\eta = \sigma_{g}^{i}$ for some $i$. So two pairs $(\sigma_{g},\tau_{1})$ and $(\sigma_{g},\tau_{2})$ give rise to the same origami if and only if $\tau_{1} = \sigma_{g}^{i}\tau_{2}\sigma_{g}^{-i}$ for some $i$. By the general form of $\tau$ given above, $\tau_{2}(1) = 3$ and so we must have $\tau_{1}(1+i) = 3+i$. However, again by the general form for $\tau$, we see that the only value $j$ for which $\tau_{1}(j) = j+2$ (modulo the cycle) is $j = 1$. Therefore, $i = 0$ and so $\tau_{1} = \tau_{2}$. As such, we have produced $(g-5)!$ many minimal $[1,1]$-origamis in $\calH(2g-2)$.

\begin{example}\label{ex:ori:odd-even}
Consider $g=7$. We begin with
\[\tau = \begin{pmatrix}
9 & 6 & 1 & 8 & 7 & 5 & 4 & \bullet & \bullet & \bullet & \bullet & \bullet & \bullet \\
1 & 2 & 3 & 4 & 5 & 6 & 7 & 8 & 9 & 10 & 11 & 12 & 13
\end{pmatrix}.\]
We can place $(3,2)$ above $(10,11)$ or $(12,13)$. We choose $(12,13)$ and obtain
\[\tau = \begin{pmatrix}
9 & 6 & 1 & 8 & 7 & 5 & 4 & \bullet & \bullet & \bullet & \bullet & 3 & 2 \\
1 & 2 & 3 & 4 & 5 & 6 & 7 & 8 & 9 & 10 & 11 & 12 & 13
\end{pmatrix}.\]
We are then forced to place $(13,12)$ above $(10,11)$, and $(11,10)$ above $(1,2)$ obtaining
\begin{align*}
\tau &= \begin{pmatrix}
9 & 6 & 1 & 8 & 7 & 5 & 4 & 11 & 10 & 13 & 12 & 3 & 2 \\
1 & 2 & 3 & 4 & 5 & 6 & 7 & 8 & 9 & 10 & 11 & 12 & 13
\end{pmatrix} \\
 & = (1,3,12,11,8,4,7,5,6,2,13,10,9).
\end{align*}
We have
\[\sigma_{7}^{-1}\tau^{-1}\sigma_{7}\tau = (1,7,4,10,12,2,8,6,13,11,9,5,3).\]
\end{example}

\subsection{Even genus odd spin}\label{subsec:gen-even-odd}

In this subsection, we will construct factorially many minimal $[1,1]$-origamis in $\calH(2g-2)$ for $g\geq 4$ even. We will prove in Section~\ref{sec:spin} that these origamis have odd spin parity. The construction again resembles that of Subsection~\ref{subsec:AMN-odd}.

Let $g\geq 4$ be even. Let $\sigma_{g}$ be defined as above. Here, we begin with 
\[\tau = \begin{pmatrix}
4 & 6 & 1 & 5 & 7 & \bullet & \bullet & \cdots & \bullet & \bullet \\
1 & 2 & 3 & 4 & 5 & 6 & 7 & \cdots & 2g-2 & 2g-1
\end{pmatrix}.\]
We now complete $\tau$ by adding in the pairs $(3,2),(9,8),\ldots,(2g-1,2g-2)$ into the top rows. We begin by placing $(3,2)$ above some $(2j,2j+1)\neq (6,7)$. We have $(2g-1-7)/2 = g-4$ choices. Next we place $(2j+1,2j)$ above some $(2k,2k+1)\neq (6,7)$. We have $g-3$ choices. We continue in this way until we are forced to place the final pair $(2i+1,2i)$ above $(6,7)$. If the choices of pairs are $(2k_{i},2k_{i}+1)$ for $1\leq i\leq g-4$, then we have obtained
\[\tau = (1,3,2k_{1},2k_{2}+1,2k_{3},\ldots,2_{g-4}+1,6,2,2k_{1}+1,2k_{2},\ldots,2k_{g-4},7,5,4).\]

Letting $\rho$ be the $(2g-1)$-cycle given by the top row of $\tau$, we find that
\[\sigma_{g}^{-1}\tau^{-1}\sigma_{g}\tau = (1,4,5,6,2g-1,2g-3,\ldots,7,\rho^{2}(7),\rho^{4}(7),\ldots,\tau^{-1}(2g-1),3)\]
is a $(2g-1)$-cycle as required.

By an argument similar to that given in the previous subsection, it can be shown that we have constructed $(g-4)!$ many origamis minimal $[1,1]$-origamis in $\calH(2g-2)$.

\begin{example}\label{ex:ori:even-odd}
For $g = 6$ and choosing $(2k_{1},2k_{1}+1) = (8,9)$, which forces $(2k_{2},2k_{2}+1) = (10,11)$, we obtain
\begin{align*}
\tau &= \begin{pmatrix}
4 & 6 & 1 & 5 & 7 & 11 & 10 & 3 & 2 & 9 & 8 \\
1 & 2 & 3 & 4 & 5 & 6 & 7 & 8 & 9 & 10 & 11
\end{pmatrix} \\
&= (1,3,8,11,6,2,9,10,7,5,4),
\end{align*}
with
\[\sigma_{6}^{-1}\tau^{-1}\sigma_{6}\tau = (1,4,5,6,11,9,7,10,2,8,3).\]
\end{example}

\subsection{Even genus even spin}\label{subsec:gen-even-even}

Finally, in this subsection, we will give the construction of the origamis we will, in Section~\ref{sec:spin}, prove are contained in $\even(2g-2)$.

Here, we start with
\[\tau = \begin{pmatrix}
7 & 5 & 1 & 6 & 4 & \bullet & \bullet & \cdots & \bullet & \bullet \\
1 & 2 & 3 & 4 & 5 & 6 & 7 & \cdots & 2g-2 & 2g-1
\end{pmatrix}.\]
Again, we complete $\tau$ by adding pairs $(3,2),(9,8),\ldots,(2g-1,2g-2)$ to the top row. We begin by placing $(3,2)$ above some $(2j,2j+1)\neq (6,7)$. We have $g-4$ choices for this. Next we place $(2j+1,2j)$ above some $(2k,2k+1)\neq (6,7)$. We have $g-3$ choices. Continuing in this way, we place the final pair above $(6,7)$ completing $\tau$. If the choices of pairs are $(2k_{i},2k_{i}+1)$ for $1\leq i\leq g-4$, then we have obtained
\[\tau = (1,3,2k_{1},2k_{2}+1,2k_{3},\ldots,2k_{g-4}+1,6,4,5,2,2k_{1}+1,2k_{2},\ldots,2k_{g-4},7).\]

Letting $\rho$ be the $(2g-1)$-cycle given by the top row of $\tau$, we find that
\[\sigma_{g}^{-1}\tau^{-1}\sigma_{g}\tau = (1,5,2g-1,2g-3,\ldots,7,4,\rho^{2}(4),\rho^{4}(4),\ldots,\tau^{-1}(2g-1),6,3)\]
is a $(2g-1)$-cycle as required.

As above, it can be checked that we have constructed $(g-4)!$ many minimal $[1,1]$-origamis in $\calH(2g-2)$.

\begin{example}\label{ex:ori:even-even}
Letting $g=6$ and choosing $(2k_{1},2k_{1}+1) = (8,9)$, which forces $(2k_{2},2k_{2}+1) = (10,11)$, we obtain
\begin{align*}
\tau &= \begin{pmatrix}
7 & 5 & 1 & 6 & 4 & 11 & 10 & 3 & 2 & 9 & 8 \\
1 & 2 & 3 & 4 & 5 & 6 & 7 & 8 & 9 & 10 & 11
\end{pmatrix} \\
&= (1,3,8,11,6,4,5,2,9,10,7),
\end{align*}
with
\[\sigma_{6}^{-1}\tau^{-1}\sigma_{6}\tau = (1,5,11,9,7,4,10,2,8,6,3).\]
\end{example}

\begin{remark}\label{rem:gen-const}
Observe that in all of the constructions of this section, including the original construction of the first author and Menasco-Nieland (that of Subsection~\ref{subsec:AMN-odd}), the fact that $[\sigma,\tau]$ is a $(2g-1)$-cycle only depends on the part of the construction of $\tau$ where we place pairs $(2j+1,2j)$ over pairs $(2k,2k+1)$. So the natural variations of the above constructions, where we place pairs in this way without requiring the resulting $\tau$ to be a $(2g-1)$-cycle, always give rise to origamis in $\calH(2g-2)$ with a single horizontal cylinder even if $\tau$ is not a $(2g-1)$-cycle.
\end{remark}

%%%%%%%%%%%%%%%%%%%%%%%%%%%%%%%%%%%%%%%%%%%%%%%%%%%%%%

\section{Spin parity calculations}\label{sec:spin}

In this section, we calculate the spin parities of the origamis in the families given in Subsection~\ref{subsec:AMN-odd} and Section~\ref{sec:gen-const}. That is, we calculate the spin parities of the odd geneus AMN origamis and their generalisations.

\subsection{An orthogonalisation method}\label{subsec:ortho}

Given an origami $O$ with a single horizontal cylinder and with sides labelled $0,1,\ldots,2g-1$, as in Figure~\ref{fig:origami-labelling}, we can associate an element of $H_{1}(O,\mathbb{Z}/2\mathbb{Z})$ to each pair of sides with the same label. Namely, we define the elements $c_{i}$ to be the cycles corresponding to the simple closed curves that travel from the centre of one of the sides labelled $i$ to the centre of the other side labelled $i$. Note that we will always assume that the left and top sides of the origami are labelled in the order $0,1,\ldots,2g-1$. Moreover, since we are working in homology with coefficients in $\mathbb{Z}/2\mathbb{Z}$, the orientations of the $c_{i}$ play no role, and the standard symplectic pairing on homology $\_\cdot\_$ when restricted to the $c_{i}$ becomes geometric intersection number modulo 2.

It follows from~\cite[Appendix C, Lemma 6]{Z2} that the set $\{c_{i}\mid 0\leq i\leq 2g-1\}$ is a generating set for $H_{1}(O,\mathbb{Z}/2\mathbb{Z})$. In fact, since the homology has dimension $2g$, the set is actually a basis. However, in order to calculate the Arf invariant of the quadratic form $\Phi$ coming from the spin structure, we require a symplectic basis. That is, we require a basis of the form $\{a_{i},b_{i}\mid 1\leq i\leq 2g\}$ such that
\[a_{i}\cdot a_{j} = b_{i}\cdot b_{j} = 0\,\,\, \forall i,j\]
and
\[a_{i}\cdot b_{j} = \delta_{i,j}.\]
Zorich describes~\cite[Appendix C]{Z2} a method for transforming our generating set into a symplectic basis, which here we call `orthogonalisation', in such a way that the value of the Arf invariant is known once the symplectic basis is achieved.

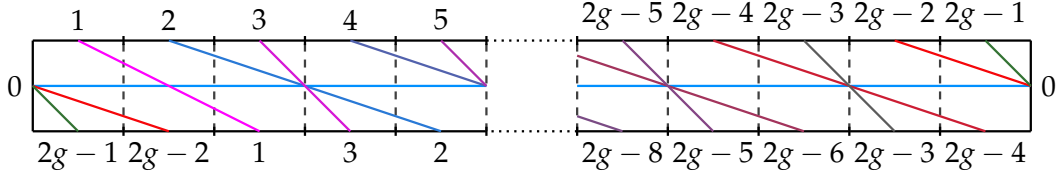
\begin{figure}[t]
    \centering
    \begin{tikzpicture}[scale = 1.2, line width = 0.3mm]
    \draw (0,0)--(0,1);
    \draw (0,1)--(5,1);
    \draw [dotted] (5,1)--(6,1);
    \draw (6,1)--(11,1);
    \draw (11,1)--(11,0);
    \draw (11,0)--(6,0);
    \draw [dotted] (6,0)--(5,0);
    \draw (5,0)--(0,0);
    
    \foreach \i in {1,2,...,10}{
        \draw [dashed, color = darkgray] (\i,0)--(\i,1);
        \draw (\i,0.05)--(\i,-0.05);
        \draw (\i,1.05)--(\i,0.95);
    }
    \draw (0,0.5) node[left] {$0$};
    \draw (11,0.5) node[right] {$0$};
    \foreach \i in {1,2,...,5}{
        \draw (\i-0.5,1) node[above] {$\i$};
        \draw (11-\i+0.5,1) node[above] {$2g-\i$};
    }
    \foreach \i in {1,2}{
        \draw (\i-0.5,0) node[below] {$2g-\i$};
    }
    \draw (2.5,0) node[below] {$1$};
    \foreach \i in {3,2}{
        \draw (6-\i+0.5,0) node[below] {$\i$};
    }
    \foreach \i in {4,6,8}{
        \draw (14-\i+0.5,0) node[below] {$2g-\i$};
    }
    \foreach \i in {3,5}{
        \draw (12-\i+0.5,0) node[below] {$2g-\i$};
    }
    \draw [color = {rgb,100:red,0;green,100;blue,100}] (0,0.5)--(5,0.5);
    \draw [color = {rgb,100:red,0;green,100;blue,100}] (6,0.5)--(11,0.5);
    \draw [color = {rgb,100:red,100;green,0;blue,100}] (0.5,1)--(2.5,0);
    \draw [color = {rgb,100:red,16;green,84;blue,84}] (1.5,1)--(4.5,0);
    \draw [color = {rgb,100:red,84;green,16;blue,84}] (2.5,1)--(3.5,0);
    \draw [color = {rgb,100:red,32;green,68;blue,68}] (3.5,1)--(5,0.5);
    \draw [color = {rgb,100:red,68;green,32;blue,68}] (4.5,1)--(5,0.5);
    \draw [color = {rgb,100:red,48;green,52;blue,52}] (6,0.166666)--(6.5,0);
    \draw [color = {rgb,100:red,64;green,36;blue,36}] (6,0.833333)--(8.5,0);
    \draw [color = {rgb,100:red,52;green,48,;blue,52}] (6.5,1)--(7.5,0);
    \draw [color = {rgb,100:red,80;green,20;blue,20}] (7.5,1)--(10.5,0);
    \draw [color = {rgb,100:red,36;green,64,;blue,36}] (8.5,1)--(9.5,0);
    \draw [color = {rgb,100:red,96;green,4;blue,4}] (9.5,1)--(11,0.5);
    \draw [color = {rgb,100:red,96;green,4;blue,4}] (0,0.5)--(1.5,0);
    \draw [color = {rgb,100:red,20;green,80,;blue,20}] (10.5,1)--(11,0.5);
    \draw [color = {rgb,100:red,20;green,80,;blue,20}] (0,0.5)--(0.5,0);
    \end{tikzpicture}
    \caption{The generating set for $H_{1}(O,\mathbb{Z}/2\mathbb{Z})$ inside the origami $O$.}\label{fig:origami-labelling}
\end{figure}

Recall that $\Phi([\gamma]) = \text{ind}(\gamma) + 1$, for $\gamma$ a simple closed curve. Since all of the cycles $c_{i}$ are represented by curves that maintain their direction with respect to the horizontal foliation in the origami, the index of the representative curve is 0 and we have $\Phi(c_{i}) = 1$ for all $i$. The key equality that we will repeatedly apply in our calculation of the Arf invariant is the following:
\[\Phi(c + c') = \Phi(c) + \Phi(c') + c\cdot c',\]
for all $c,c'\in H_{1}(O,\mathbb{Z}/2\mathbb{Z})$. This equality holds by definition from the fact that $\Phi$ is a quadratic form. Recall that, given a symplectic basis $\{a_{i},b_{i}\}$, the Arf invariant of $\Phi$ is then
\[\sum_{i = 1}^{g}\Phi(a_{i})\Phi(b_{i}).\]

The orthogonalisation procedure will repeatedly change an element of the basis in the following way. If we have just chosen $c'$ and $c''$ to form a symplectic pair $a_{i}$ and $b_{i}$ in the basis, then all of the elements that are not orthogonal to this pair will change by
\[c\mapsto c + (c\cdot c'')c' + (c\cdot c')c'' = : \tilde{c},\]
after which $\tilde{c}\cdot a_{i} = \tilde{c}\cdot b_{i} = 0$.

\begin{example}\label{ex:ARF:AMNodd}
We perform an example calculation with the origami shown in Figure~\ref{fig:origami-labelling}. We have $\Phi(c_{i}) = 1$ for all $i$ and the following intersection data
\[c_{0}\cdot c_{i} = 1,\,\forall\,i\neq 0,\]
\[c_{1}\cdot c_{i} = 1,\,\text{ for }i = 0,\]
then for $2\leq j\leq 2g-2$ even we have
\[c_{j}\cdot c_{i} = 1,\,\text{ for }i = 0,j+1,\]
and for $3\leq j\leq 2g-1$ odd we have
\[c_{j}\cdot c_{i} = 1,\,\text{ for }i = 0,j-1,\]
with all other intersections equal to 0.

We first choose $a_{1} = c_{0}$ and $b_{1} = c_{1}$. We see that these satisfy $a_{1}\cdot b_{1} = 1$. We have $\Phi(a_{1}) = 1 = \Phi(b_{1})$. Now orthogonalising we see that for $2\leq j\leq 2g-1$ we have
\[c_{j} \mapsto c_{j} + (c_{j}\cdot c_{0})c_{1} + (c_{j}\cdot c_{1})c_{0} = c_{j} + c_{1}.\]
We now have
\[\Phi(c_{j}+c_{1}) = \Phi(c_{j}) + \Phi(c_{1}) + c_{j}\cdot c_{1} = 1 + 1 + 0 \equiv 0 \mod 2.\]

Note that the intersections are now, for $2\leq j\leq 2g-2$ even,
\[(c_{j}+c_{1})\cdot(c_{i}+c_{1}) = 1,\,\text{ for }i = j+1,\]
and, for $3\leq j\leq 2g-1$ odd,
\[(c_{j}+c_{1})\cdot(c_{i}+c_{1}) = 1,\,\text{ for }i = j-1,\]
with all other intersections equal to 0. Therefore, for $2\leq i\leq g$, we can choose $a_{i} = c_{2i-2}+c_{1}$ and $b_{i} = c_{2i-1}+c_{1}$ and we obtain a symplectic basis $\{a_{i},b_{i}\}_{i = 1}^{g}$ with $\Phi(a_{i}) = \Phi(b_{i}) = \delta_{i,1}$.

Finally, the Arf invariant is calculated to be
\[\sum_{i = 1}^{g}\Phi(a_{i})\Phi(b_{i}) = 1\cdot 1 + (g-1)\cdot 0\cdot 0 = 1.\]
Hence, the origami has odd spin parity.
\end{example}

\subsection{The key lemma}\label{subsec:key-lem}

In this section, we prove the lemma that will allow us to show that, for a fixed construction of Section~\ref{sec:gen-const}, any two origamis produced by the same method have the same spin parity.

\begin{lemma}\label{lem:key}
Let $O = (\sigma_{g},\tau)$ be a $[1,1]$-origami given by one of the constructions in Subsection~\ref{subsec:AMN-odd} or Section~\ref{sec:gen-const}. Define
\[P := \left\{\begin{matrix}
\varnothing, &\text{if $O$ is given by Subsection~\ref{subsec:AMN-odd}} \\
\{4,5,6,7,8,9\}, &\text{if $O$ is given by Subsection~\ref{subsec:gen-odd-even}} \\
\{4,5,6,7\}, &\text{if $O$ is given by Subsection~\ref{subsec:gen-even-odd} or~\ref{subsec:gen-even-even}.}
\end{matrix}\right.\]
Realising $O$ with the top sides labelled as in Figure~\ref{fig:origami-labelling}, and letting $Q = \{2,3,\ldots,2g-1\}\setminus P$, let $O' = (\sigma_{g},\tau')$ be an origami obtained from $O$ by performing a permutation $\mu = (2i,2j,2k)(2i+1,2j+1,2k+1)$ of the labels of the bottom sides of $O$ with $\{2i,2i+1,2j,2j+1,2k,2k+1\}\subset Q$. Equivalently, we have $\tau' = \tau\mu$. Then $O'$ has the same spin parity as $O$.
\end{lemma}

Before we prove the lemma, we will give an illustrative example for two origamis produced by the genus 6 construction of Subsection~\ref{subsec:gen-even-odd}.

\begin{figure}[b]
    \centering
    \begin{tikzpicture}[scale = 1.2, line width = 0.3mm]
        \foreach \i in {0,-2.2}{
            \draw (0,\i)--(0,1+\i);
            \draw (0,1+\i)--(11,1+\i);
            \draw (11,1+\i)--(11,\i);
            \draw (11,\i)--(0,\i);
            \foreach \j in {1,...,10}{
                \draw [dashed, color = darkgray] (\j,\i)--(\j,1+\i);
                \draw (\j,0.95+\i)--(\j,1.05+\i);
                \draw (\j,-0.05+\i)--(\j,0.05+\i);
            }
            \draw (0,0.5+\i) node[left] {$0$};
            \draw (11,0.5+\i) node[right] {$0$};
            \draw (0.5,\i) node[below] {$4$};
            \draw (1.5,\i) node[below] {$6$};
            \draw (2.5,\i) node[below] {$1$};
            \draw (3.5,\i) node[below] {$5$};
            \draw (4.5,\i) node[below] {$7$};
            \foreach \j in {1,...,11}{
                \draw (\j-0.5,1+\i) node[above] {$\j$};
            }
        }
        \draw (5.5,0) node[below] {$11$};
        \draw (6.5,0) node[below] {$10$};
        \draw (7.5,0) node[below] {$3$};
        \draw (8.5,0) node[below] {$2$};
        \draw (9.5,0) node[below] {$9$};
        \draw (10.5,0) node[below] {$8$};
        \draw (5.5,-2.2) node[below] {$9$};
        \draw (6.5,-2.2) node[below] {$8$};
        \draw (7.5,-2.2) node[below] {$11$};
        \draw (8.5,-2.2) node[below] {$10$};
        \draw (9.5,-2.2) node[below] {$3$};
        \draw (10.5,-2.2) node[below] {$2$};
        \draw [color = {rgb,100:red,0;green,100;blue,100}] (0,0.5)--(11,0.5);
        \draw [color = {rgb,100:red,100;green,0;blue,100}] (0.5,1)--(2.5,0);
        \draw [color = {rgb,100:red,16;green,84;blue,84}] (1.5,1)--(8.5,0);
        \draw [color = {rgb,100:red,84;green,16;blue,84}] (2.5,1)--(7.5,0);
        \draw [color = {rgb,100:red,32;green,68;blue,68}] (3.5,1)--(0.5,0);
        \draw [color = {rgb,100:red,68;green,32;blue,68}] (4.5,1)--(3.5,0);
        \draw [color = {rgb,100:red,96;green,4;blue,4}] (1.5,0)--(5.5,1);
        \draw [color = {rgb,100:red,52;green,48,;blue,52}] (6.5,1)--(4.5,0);
        \draw [color = {rgb,100:red,80;green,20;blue,20}] (7.5,1)--(10.5,0);
        \draw [color = {rgb,100:red,36;green,64,;blue,36}] (8.5,1)--(9.5,0);
        \draw [color = {rgb,100:red,10;green,10;blue,80}] (9.5,1)--(6.5,0);
        \draw [color = {rgb,100:red,20;green,80,;blue,20}] (10.5,1)--(5.5,0);
        \draw [color = {rgb,100:red,0;green,100;blue,100}] (0,0.5-2.2)--(11,0.5-2.2);
        \draw [color = {rgb,100:red,100;green,0;blue,100}] (0.5,1-2.2)--(2.5,-2.2);
        \draw [color = {rgb,100:red,16;green,84;blue,84}] (1.5,-1.2)--(10.5,-2.2);
        \draw [color = {rgb,100:red,84;green,16;blue,84}] (2.5,-1.2)--(9.5,-2.2);
        \draw [color = {rgb,100:red,32;green,68;blue,68}] (3.5,-1.2)--(0.5,-2.2);
        \draw [color = {rgb,100:red,68;green,32;blue,68}] (4.5,-1.2)--(3.5,-2.2);
        \draw [color = {rgb,100:red,96;green,4;blue,4}] (1.5,-2.2)--(5.5,-1.2);
        \draw [color = {rgb,100:red,52;green,48,;blue,52}] (6.5,-1.2)--(4.5,-2.2);
        \draw [color = {rgb,100:red,80;green,20;blue,20}] (7.5,-1.2)--(6.5,-2.2);
        \draw [color = {rgb,100:red,36;green,64,;blue,36}] (8.5,-1.2)--(5.5,-2.2);
        \draw [color = {rgb,100:red,10;green,10;blue,80}] (9.5,-1.2)--(8.5,-2.2);
        \draw [color = {rgb,100:red,20;green,80,;blue,20}] (10.5,-1.2)--(7.5,-2.2);
    \end{tikzpicture}
    \caption{Realisations of the origamis $O$ and $O'$ in $\calH(10)$.}
    \label{fig:spin-example}
\end{figure}
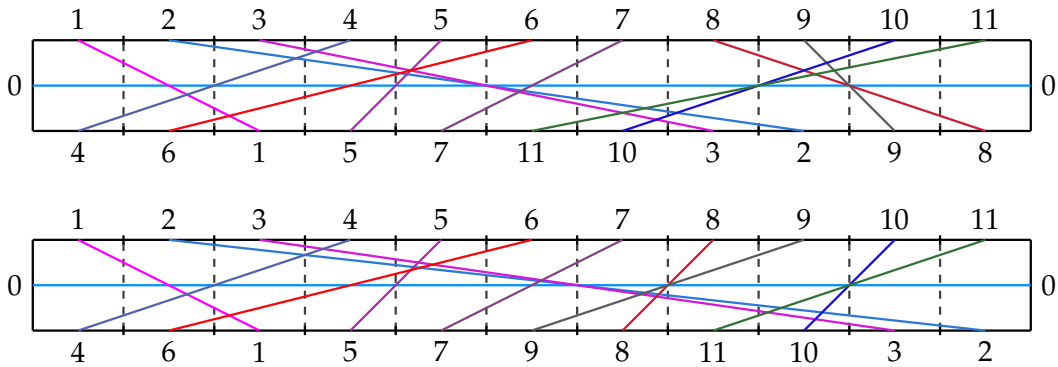

\begin{example}\label{ex:ARF:3-cycle}
Consider the origamis $$O = ((1,2,3,4,5,6,7,8,9,10,11),(1,3,8,11,6,2,9,10,7,5,4))$$ and $$O' = ((1,2,3,4,5,6,7,8,9,10,11),(1,3,10,9,6,2,11,8,7,5,4)).$$ We can realise them as in Figure~\ref{fig:spin-example}. In the language of Lemma~\ref{lem:key}, we have $P = \{4,5,6,7\}$ and $\mu = (2,8,10)(3,9,11)$.

On both $O$ and $O'$ begin with $a_{1} = c_{0}$ and $b_{1} = c_{1}$. We have $\Phi(a_{1}) = \Phi(b_{1}) = 1$. After orthogonalisation, $c_{4}$ becomes $c_{4}+c_{1}+c_{0}$, $c_{6}$ becomes $c_{6}+c_{1}+c_{0}$, and the remaining cycles $c_{i}$ become $c_{i}+c_{1}$. Next we will choose $a_{2} = c_{4}+c_{1}+c_{0}$ and $b_{2} = c_{5}+c_{1}$. It can be checked that $a_{2}\cdot b_{2} = 1$ with $\Phi(a_{2}) = \Phi(b_{2}) = 0$. Orthogonalising fixes $c_{6}+c_{1}+c_{0}$, sends $c_{2}+c_{1}$ to $c_{2}+c_{4}+c_{0}$, $c_{3}+c_{1}$ to $c_{3}+c_{4}+c_{0}$, and $c_{i}+c_{1}$ to $c_{i}+c_{5}$ for $7\leq i\leq 11$. We now choose $a_{3} = c_{6}+c_{1}+c_{0}$ and $b_{3} = c_{7}+c_{5}$. Again, it can be checked that $a_{3}\cdot b_{3} = 1$ and $\Phi(a_{3}) = \Phi(b_{3}) = 0$. Orthogonalising fixes $c_{2}+c_{4}+c_{0}$ and $c_{3}+c_{4}+c_{0}$, and sends each remaining $c_{i}+c_{5}$ to $c_{i}+c_{7}$, for $8\leq i\leq 11$.

Up to now, the orthogonalisation process has depended only on $P$ and the order in which we chose the $a_{i}$ and $b_{i}$. That is, so far, there is no difference between $O$ and $O'$.

We will first complete the calculation of the spin parity of $O$. We will choose $a_{4} = c_{2}+c_{4}+c_{0}$ and $b_{4} = c_{3}+c_{4}+c_{0}$, then $a_{4}\cdot b_{4} = 1$ and $\Phi(a_{4}) = \Phi(b_{4}) = 0$. Orthogonalising fixes $c_{10}+c_{7}$ and $c_{11}+c_{7}$, and sends $c_{8}+c_{7}$ to $c_{8}+c_{7}+c_{3}+c_{2}$ and $c_{9}+c_{7}$ to $c_{9}+c_{7}+c_{3}+c_{2}$. Next, set $a_{5} = c_{8}+c_{7}+c_{3}+c_{2}$ and $b_{5} = c_{9}+c_{7}+c_{3}+c_{2}$, giving $a_{5}\cdot b_{5} = 1$ and $\Phi(a_{5}) = \Phi(b_{5}) = 1$. The final orthogonalisation step sends $c_{10}+c_{7}$ to $c_{10}+c_{9}+c_{8}+c_{7}$ and $c_{11}+c_{7}$ to $c_{11}+c_{9}+c_{8}+c_{7}$. Finally, we set $a_{6} = c_{10}+c_{9}+c_{8}+c_{7}$ and $b_{6} = c_{11}+c_{9}+c_{8}+c_{7}$ to get $a_{6}\cdot b_{6} = 1$ and $\Phi(a_{6}) = \Phi(b_{6}) = 1$. So the spin parity of $O$ is $\sum_{i = 1}^{g}\Phi(a_{i})\Phi(b_{i}) = 1 + 0 + 0 + 0 + 1 + 1 \equiv 1$. Hence, $O$ has odd spin parity.

Now, returning to $O'$, we see that the intersection pattern of the cycles corresponding to the sides labelled by the pairs $\{3,2\},\{9,8\}$, and $\{11,10\}$ has changed from that on $O$ due to the action of $\mu$. We will again choose $a_{4} = c_{2}+c_{4}+c_{0}$ and $b_{4} = c_{3}+c_{4}+c_{0}$, giving $a_{4}\cdot b_{4} = 1$ and $\Phi(a_{4}) = \Phi(b_{4}) = 0$. However, when we perform the orthogonalisation step here we see that all cycles are fixed. That is, $c_{i}+c_{7}$ is fixed as $c_{i}+c_{7}$ for $8\leq i\leq 11$. On $O$, the differing intersection pattern meant that $c_{8}+c_{7}$ and $c_{9}+c_{7}$ were changed at this step. Next we choose $a_{5} = c_{8}+c_{7}$ and $b_{5} = c_{9}+c_{7}$, giving $a_{5}\cdot b_{5} = 1$ and $\Phi(a_{5}) = \Phi(b_{5}) = 0$. Orthogonalisation fixes $c_{10}+c_{7}$ and $c_{11}+c_{7}$. Note that these were changed in this step for $O$. We choose $a_{6} = c_{10}+c_{7}$ and $b_{6} = c_{11}+c_{7}$ to get $a_{6}\cdot b_{6} = 1$ and $\Phi(a_{6}) = \Phi(b_{6}) = 0$. So the spin parity of $O'$ is $\sum_{i = 1}^{g}\Phi(a_{i})\Phi(b_{i}) = 1 + 0 + 0 + 0 + 0 + 0 \equiv 1$. Hence $O'$ also has odd spin parity.
\end{example}

The key takeaway from this example is that, even though the values of $\Phi(a_{i})\Phi(b_{i})$ for $i \in\{5,6\}$ changed between the origamis $O$ and $O'$, the sum $\Phi(a_{5})\Phi(b_{5})+\Phi(a_{6})\Phi(b_{6})$ did not change modulo 2. This is due to the fact that the side labels were moved by a 3-cycle on the pairs. That this holds in general is the content of the following proof of Lemma~\ref{lem:key}.

\begin{proof}[Proof of Lemma~\ref{lem:key}]
On both $O$ and $O'$, let $c_{i}$ denote the homology cycle represented by a simple closed curve between the sides of the origami labelled by $i$. Choose $a_{1} = c_{0}$ and $b_{1} = c_{1}$. Perform the orthogonalisation of the remaining cycles $c_{i}$.

If $O$ has been built using the method of Subsection~\ref{subsec:AMN-odd}, then $P = \varnothing$. At this point, after the orthogonalisation, each cycle $c_{i}$ has become $c_{i}+c_{1}$ for $2\leq i\leq 2g-1$. Note that this is the same on both $O$ and $O'$. See Figure~\ref{fig:intersec-patterns} a).

If $O$ has been built using the method of Subsection~\ref{subsec:gen-odd-even}, then $P = \{4,5,6,7,8,9\}$ and after orthogonalising with respect to $a_{1}$ and $b_{1}$ the cycles $c_{6}$ and $c_{9}$ have become $c_{6}+c_{1}+c_{0}$ and $c_{9}+c_{1}+c_{0}$, and all other cycles $c_{i}$ have become $c_{i}+c_{1}$. Now choose $a_{2} = c_{6}+c_{1}+c_{0}$ and $b_{2} = c_{9}+c_{1}+c_{0}$. Orthogonalising will fix $c_{2}+c_{1}$, $c_{3}+c_{1}$, $c_{4}+c_{1}$, and $c_{5}+c_{1}$, but will send $c_{7}+c_{1}$ to $c_{7}+c_{9}+c_{0}$, $c_{8}+c_{1}$ to $c_{8}+c_{9}+c_{0}$, and send the remaining $c_{i}+c_{1}$ to $c_{i}+c_{9}+c_{6}+c_{1}$. Next we set $a_{3} = c_{4}+c_{1}$ and $b_{3} = c_{5}+c_{1}$. Orthogonalising sends $c_{2}+c_{1}$ to $c_{2}+c_{5}+c_{4}+c_{1}$, $c_{3}+c_{1}$ to $c_{3}+c_{5}+c_{4}+c_{1}$, $c_{7}+c_{9}+c_{0}$ to $c_{7}+c_{9}+c_{5}+c_{4}+c_{0}$, $c_{8}+c_{9}+c_{0}$ to $c_{8}+c_{9}+c_{5}+c_{4}+c_{0}$, and fixes the remaining cycles $c_{i}+c_{9}+c_{6}+c_{1}$, $10\leq i\leq 2g-1$. Now set $a_{4} = c_{7}+c_{9}+c_{5}+c_{4}+c_{0}$ and $b_{4} = c_{8}+c_{9}+c_{5}+c_{4}+c_{0}$. Orthogonalising sends $c_{2}+c_{5}+c_{4}+c_{1}$ to $c_{2}+c_{8}+c_{7}+c_{5}+c_{4}+c_{1}$, $c_{3}+c_{5}+c_{4}+c_{1}$ to $c_{3}+c_{8}+c_{7}+c_{5}+c_{4}+c_{1}$, and the remaining cycles $c_{i}+c_{9}+c_{6}+c_{1}$ to $c_{i}+c_{9}+c_{8}+c_{7}+c_{6}+c_{1}$. This process can be carried out on both $O$ and $O'$ since the changed side labels for $O'$ were in $Q$ which is disjoint from $P$. See Figure~\ref{fig:intersec-patterns} b).

If $O$ has been built using the method of Subsection~\ref{subsec:gen-even-odd}, then $P = \{4,5,6,7\}$. Orthogonalising with respect to $a_{1}$ and $b_{1}$ sends $c_{4}$ to $c_{4}+c_{1}+c_{0}$, $c_{6}$ to $c_{6}+c_{1}+c_{0}$ and all other cycles $c_{i}$ to $c_{i}+c_{1}$. Next we choose $a_{2} = c_{4}+c_{1}+c_{0}$ and $b_{2} = c_{5}+c_{1}$. Orthogonalising fixes $c_{6}+c_{1}+c_{0}$, sends $c_{2}+c_{1}$ to $c_{2}+c_{4}+c_{0}$, $c_{3}+c_{1}$ to $c_{3}+c_{4}+c_{0}$, and the remaining cycles $c_{i}+c_{1}$ to $c_{i}+c_{5}$. Now choose $a_{3} = c_{6}+c_{1}+c_{0}$ and $b_{3} = c_{7}+c_{5}$. Orthogonalising fixes $c_{2}+c_{4}+c_{0}$ and $c_{3}+c_{4}+c_{0}$, and sends the remaining cycles $c_{i}+c_{5}$ to $c_{i}+c_{7}$. Again, this process can be carried out on both $O$ and $O'$ since the changed side labels for $O'$ were in $Q$ which is disjoint from $P$. See Figure~\ref{fig:intersec-patterns} c).

\begin{figure}[t]
    \centering
    \begin{tikzpicture}[scale = 0.9, line width = 0.3mm]
        \draw (0,0)--(0,1);
        \draw (0,1)--(3,1);
        \draw (3,0)--(0,0);
        \foreach \j in {1,2,3}{
            \draw [dashed, color = darkgray] (\j,0)--(\j,1);
            \draw (\j,0.95)--(\j,1.05);
            \draw (\j,-0.05)--(\j,0.05);
            \draw (\j-0.5,1) node[above] {$\j$};
        }
        \draw (0,0.5) node[left] {$0$};
        \draw (0.5,0) node[below] {$2j+1$};
        \draw (1.5,0) node[below] {$2j$};
        \draw (2.5,0) node[below] {$1$};
        \draw [color = {rgb,100:red,0;green,100;blue,100}] (0,0.5)--(3,0.5);
        \draw [color = {rgb,100:red,100;green,0;blue,100}] (0.5,1)--(2.5,0);
        \draw [color = {rgb,100:red,16;green,84;blue,84}] (1.5,1)--(3,0.5);
        \draw [color = {rgb,100:red,84;green,16;blue,84}] (2.5,1)--(3,0.5);
        \draw [color = {rgb,100:red,5;green,80,;blue,15}] (0.5,0)--(0,0.25);
        \draw [color = {rgb,100:red,10;green,80,;blue,10}] (1.5,0)--(0,0.25);
        \draw (1.5,-1) node {a)};
        %%%%%%%%%%%%%%%%%%%%%%%%%%%%%%%%%%%
        \foreach \i in {5}{
            \draw (\i,0)--(\i,1);
            \draw (\i,1)--(\i+9,1);
            \draw (\i+9,0)--(\i,0);
            \foreach \j in {1,...,9}{
                \draw [dashed, color = darkgray] (\i+\j,0)--(\i+\j,1);
                \draw (\i+\j,0.95)--(\i+\j,1.05);
                \draw (\i+\j,-0.05)--(\i+\j,0.05);
                \draw (\i+\j-0.5,1) node[above] {$\j$};
            }
            \draw (\i,0.5) node[left] {$0$};
            \draw (\i+0.5,0) node[below] {$9$};
            \draw (\i+1.5,0) node[below] {$6$};
            \draw (\i+2.5,0) node[below] {$1$};
            \draw (\i+3.5,0) node[below] {$8$};
            \draw (\i+4.5,0) node[below] {$7$};
            \draw (\i+5.5,0) node[below] {$5$};
            \draw (\i+6.5,0) node[below] {$4$};
            \draw (\i+7.5,0) node[below] {$2j+1$};
            \draw (\i+8.5,0) node[below] {$2j$};
            \draw [color = {rgb,100:red,0;green,100;blue,100}] (\i,0.5)--(\i+9,0.5);
            \draw [color = {rgb,100:red,100;green,0;blue,100}] (\i+0.5,1)--(\i+2.5,0);
            \draw [color = {rgb,100:red,16;green,84;blue,84}] (\i+1.5,1)--(\i+9,0.5);
            \draw [color = {rgb,100:red,84;green,16;blue,84}] (\i+2.5,1)--(\i+9,0.5);
            \draw [color = {rgb,100:red,32;green,68;blue,68}] (\i+3.5,1)--(\i+6.5,0);
            \draw [color = {rgb,100:red,68;green,32;blue,68}] (\i+4.5,1)--(\i+5.5,0);
            \draw [color = {rgb,100:red,96;green,4;blue,4}] (\i+1.5,0)--(\i+5.5,1);
            \draw [color = {rgb,100:red,52;green,48,;blue,52}] (\i+6.5,1)--(\i+4.5,0);
            \draw [color = {rgb,100:red,80;green,20;blue,20}] (\i+7.5,1)--(\i+3.5,0);
            \draw [color = {rgb,100:red,36;green,64,;blue,36}] (\i+8.5,1)--(\i+0.5,0);
            \draw [color = {rgb,100:red,5;green,80,;blue,15}] (\i+7.5,0)--(\i+9,0.25);
            \draw [color = {rgb,100:red,10;green,80,;blue,10}] (\i+8.5,0)--(\i+9,0.25);
            \draw (\i+4.5,-1) node {b)};
        }
        %%%%%%%%%%%%%%%%%%%%%%%%%%%%%%%%%%%%%
        \foreach \i in {-0.5}{
            \draw (\i,-4)--(\i,-3);
            \draw (\i,-3)--(\i+7,-3);
            \draw (\i+7,-4)--(\i,-4);
            \foreach \j in {1,...,7}{
                \draw [dashed, color = darkgray] (\i+\j,-4)--(\i+\j,-3);
                \draw (\i+\j,-4+0.95)--(\i+\j,-4+1.05);
                \draw (\i+\j,-4-0.05)--(\i+\j,-4+0.05);
                \draw (\i+\j-0.5,-3) node[above] {$\j$};
            }
            \draw (\i,-4+0.5) node[left] {$0$};
            \draw (\i+0.5,-4) node[below] {$4$};
            \draw (\i+1.5,-4) node[below] {$6$};
            \draw (\i+2.5,-4) node[below] {$1$};
            \draw (\i+3.5,-4) node[below] {$5$};
            \draw (\i+4.5,-4) node[below] {$7$};
            \draw (\i+5.5,-4) node[below] {$2j+1$};
            \draw (\i+6.5,-4) node[below] {$2j$};
            \draw [color = {rgb,100:red,0;green,100;blue,100}] (\i,-4+0.5)--(\i+7,-4+0.5);
            \draw [color = {rgb,100:red,100;green,0;blue,100}] (\i+0.5,-3)--(\i+2.5,-4);
            \draw [color = {rgb,100:red,16;green,84;blue,84}] (\i+1.5,-3)--(\i+7,-4+0.5);
            \draw [color = {rgb,100:red,84;green,16;blue,84}] (\i+2.5,-3)--(\i+7,-4+0.5);
            \draw [color = {rgb,100:red,32;green,68;blue,68}] (\i+3.5,-3)--(\i+0.5,-4);
            \draw [color = {rgb,100:red,68;green,32;blue,68}] (\i+4.5,-3)--(\i+3.5,-4);
            \draw [color = {rgb,100:red,96;green,4;blue,4}] (\i+1.5,-4)--(\i+5.5,-3);
            \draw [color = {rgb,100:red,52;green,48,;blue,52}] (\i+6.5,-3)--(\i+4.5,-4);
            \draw [color = {rgb,100:red,5;green,80,;blue,15}] (\i+6.5,-4)--(\i+7,-4+0.25);
            \draw [color = {rgb,100:red,10;green,80,;blue,10}] (\i+5.5,-4)--(\i+7,-4+0.25);
            \draw (\i+3.5,-2) node {c)};
        }
        %%%%%%%%%%%%%%%%%%%%%%%%%%%%%%%%%%%%%%
        \foreach \i in {7.5}{
            \draw (\i,-4)--(\i,-3);
            \draw (\i,-3)--(\i+7,-3);
            \draw (\i+7,-4)--(\i,-4);
            \foreach \j in {1,...,7}{
                \draw [dashed, color = darkgray] (\i+\j,-4)--(\i+\j,-3);
                \draw (\i+\j,-4+0.95)--(\i+\j,-4+1.05);
                \draw (\i+\j,-4-0.05)--(\i+\j,-4+0.05);
                \draw (\i+\j-0.5,-3) node[above] {$\j$};
            }
            \draw (\i,-4+0.5) node[left] {$0$};
            \draw (\i+0.5,-4) node[below] {$7$};
            \draw (\i+1.5,-4) node[below] {$5$};
            \draw (\i+2.5,-4) node[below] {$1$};
            \draw (\i+3.5,-4) node[below] {$6$};
            \draw (\i+4.5,-4) node[below] {$4$};
            \draw (\i+5.5,-4) node[below] {$2j+1$};
            \draw (\i+6.5,-4) node[below] {$2j$};
            \draw [color = {rgb,100:red,0;green,100;blue,100}] (\i,-4+0.5)--(\i+7,-4+0.5);
            \draw [color = {rgb,100:red,100;green,0;blue,100}] (\i+0.5,-3)--(\i+2.5,-4);
            \draw [color = {rgb,100:red,16;green,84;blue,84}] (\i+1.5,-3)--(\i+7,-4+0.5);
            \draw [color = {rgb,100:red,84;green,16;blue,84}] (\i+2.5,-3)--(\i+7,-4+0.5);
            \draw [color = {rgb,100:red,32;green,68;blue,68}] (\i+3.5,-3)--(\i+4.5,-4);
            \draw [color = {rgb,100:red,68;green,32;blue,68}] (\i+4.5,-3)--(\i+1.5,-4);
            \draw [color = {rgb,100:red,96;green,4;blue,4}] (\i+3.5,-4)--(\i+5.5,-3);
            \draw [color = {rgb,100:red,52;green,48,;blue,52}] (\i+6.5,-3)--(\i+0.5,-4);
            \draw [color = {rgb,100:red,5;green,80,;blue,15}] (\i+6.5,-4)--(\i+7,-4+0.25);
            \draw [color = {rgb,100:red,10;green,80,;blue,10}] (\i+5.5,-4)--(\i+7,-4+0.25);
            \draw (\i+3.5,-2) node {d)};
        }
    \end{tikzpicture}
    \caption{Intersection patterns in origamis constructed using the methods of a) Subsection~\ref{subsec:AMN-odd}, b) Subsection~\ref{subsec:gen-odd-even}, c) Subsection~\ref{subsec:gen-even-odd}, and d) Subsection~\ref{subsec:gen-even-even}.}
    \label{fig:intersec-patterns}
\end{figure}
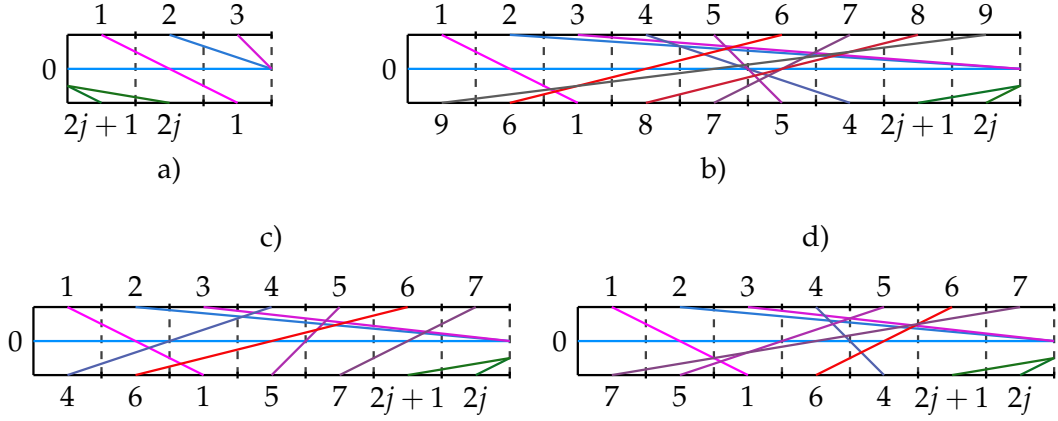

Finally, if $O$ has been built using the method of Subsection~\ref{subsec:gen-even-even}, then $P = \{4,5,6,7\}$, and orthogonalising with respect to $a_{1}$ and $b_{1}$ sends $c_{5}$ to $c_{5}+c_{1}+c_{0}$, $c_{7}$ to $c_{7}+c_{1}+c_{0}$, and the remaining cycles $c_{i}$ to $c_{i}+c_{1}$. Next we choose $a_{2} = c_{5}+c_{1}+c_{0}$ and $b_{2} = c_{7}+c_{1}+c_{0}$. Orthogonalising fixes $c_{2}+c_{1}, c_{3}+c_{1}$ and $c_{4}+c_{1}$, and sends $c_{6}+c_{1}$ to $c_{6}+c_{7}+c_{0}$ and the remaining cycles $c_{i}+c_{1}$ to $c_{i}+c_{7}+c_{5}+c_{1}$. Now choose $a_{3} = c_{4}+c_{1}$ and $b_{3} = c_{6}+c_{7}+c_{0}$. Orthogonalising sends $c_{2}+c_{1}$ to $c_{2}+c_{7}+c_{6}+c_{4}+c_{0}$, $c_{3}+c_{1}$ to $c_{3}+c_{7}+c_{6}+c_{4}+c_{0}$, and each remaining cycle $c_{i}+c_{7}+c_{5}+c_{1}$ to $c_{i}+c_{7}+c_{5}+c_{4}$. Again, this process carries out in the same way on both $O$ and $O'$. See Figure~\ref{fig:intersec-patterns} d).

In all cases, we can begin the orthogonalisation process on $O$ and $O'$ in the same way, handling the cycles $c_{0}$, $c_{1}$, and those corresponding to labels in $P$. The remaining cycles correspond to labels in $Q$ and we claim that they can be paired up to form the remaining choices of $a_{i}$ and $b_{i}$. 

\begin{lemma} \label{lem:pair-up-q-cycles}
    In every iteration $i$ of the orthogonalisation process, any pair of cycles $c$ and $c'$ that began as $c_{2j}$ and $c_{2j+1}$ for $2j,2j+1\in Q$ are of the form $c = c_{2j} + d_i$, $c' = c_{2j+1} + d_i$ where $d_i$ is some cycle forced by the orthogonalisation process. Thus, $c \cdot c' = 1$ and $c'' \cdot c = c'' \cdot c'$ for all $c'' \not \in \{c,c'\}$.
\end{lemma}
\begin{proof}
    The statement trivially holds at the beginning of the orthogonalisation process. Proceeding inductively, if $c_{2j},c_{2j+1}$ are transformed to $c_{2j} + d_i, c_{2j+1} + d_i$ at the $i$-th iteration, and neither has been chosen as a symplectic basis element, choose a pair of symplectic basis elements $a_n,b_n$ and assume $a_n,b_n \not = c_{2j},c_{2j+1}$. Then, $c_{2j} + d_i \mapsto c_{2j} + (c_{2j} \cdot a_{i+1} + d_i \cdot a_{i+1})b_{i+1} + (c_{2j} \cdot b_{i+1} + d_i \cdot b_{i+1})a_{i+1}, c_{2j+1} \mapsto c_{2j+1} + (c_{2j+1} \cdot a_{i+1} + d_i \cdot a_{i+1})b_{i+1} + (c_{2j+1} \cdot b_{i+1} + d_i \cdot b_{i+1})a_{i+1}$ with $c_{2j} \cdot a_{i+1} = c_{2j+1} \cdot a_{i+1}$, $c_{2j} \cdot b_{i+1} = c_{2j+1} \cdot b_{i+1}$. This completes the first part of the proof. It follows that $c'' \cdot c = c'' \cdot (c_{2j}+ d_i) = c'' \cdot (c_{2j+1} + d_i) = c'' \cdot c'$ for all $c'' \not \in \{c,c'\}$ and
    \begin{multline*}
        c \cdot c' = (c_{2j}+ d_i) \cdot (c_{2j+1} + d_i) 
        = c_{2j} \cdot c_{2j+1} + c_{2j} \cdot d_i + d_i \cdot c_{2j+1} + d_i\cdot d_i \\
        = 1 + 2c_{2j} \cdot d_i + d_i \cdot d_i
        = 1 \mod 2.
    \end{multline*}
\end{proof}

We must therefore check that the orthogonalisation steps required to complete this process result in bases on $O$ and $O'$ having the same Arf invariant. Note that this does need to be checked as the orthogonalisation steps will depend on the intersection patterns of cycles with labels in $Q$ which, by assumption, are different on $O$ and $O'$. The key observation is summarized in the following lemma.

\begin{lemma} \label{lem:arf-track}
    If $c$ and $c'$ are cycles that began as $c_{2j}$ and $c_{2j+1}$ for $2j,2j+1\in Q$, $\Phi(c) = \Phi(c')$. Once $c,c'$ are chosen as symplectic basis elements, any pair of cycles $d, d'$ coming from some $\{2l,2l+1\}\subset Q$ will either remain unchanged and hence their values under $\Phi$ are fixed, or they are both changed by adding $c+c'$ and their values under $\Phi$ change by 1 modulo 2. Moreover, relative orthogonality is preserved among the remaining cycles. 
\end{lemma}
\begin{proof}
    The lemma is mostly a corollary of Lemma \ref{lem:pair-up-q-cycles}. The first statement follows from the computation $\Phi(c)$ = $\Phi(c_{2j} + d_i) = \Phi(c_{2j}) + \Phi(d_i) + c_{2j} \cdot d_{i} = \Phi(c_{2j+1}) + \Phi(d_i) + c_{2j+1} \cdot d_i = \Phi(c_{2j+1} + d_i) = \Phi(c')$. The second statement is a consequence of $d \cdot c = d \cdot c' = d' \cdot c = d' \cdot c'$ and if $d \cdot c = 1$, $\Phi(d + c + c') = \Phi(d) + \Phi(c) + \Phi(c') + d \cdot c + d \cdot c' + c \cdot c' = \Phi(d) + 2\Phi(c) + 2 d \cdot c + 1 = \Phi(d) + 1 \mod 2$. Finally, if $d, e$ are two cycles distinct from $c, c'$, we have
    \begin{align*}
    &(d + (d\cdot c)(c+c'))\cdot (e + (e\cdot c)(c+c')) \\[-15pt]
    &= d\cdot e + (e\cdot c)(d\cdot(c+c')) + (d\cdot c)(e\cdot (c+c')) + (d\cdot c)(e\cdot c)\overbrace{(c+c')\cdot(c+c')}^{=0} \\
    &= d\cdot e + (d\cdot c)(e\cdot c)+\underbrace{(d\cdot c')}_{=(d\cdot c)}(e\cdot c)+(d\cdot c)(e\cdot c)+(d\cdot c)\underbrace{(e\cdot c')}_{=(e\cdot c)} \\[-15pt]
    &\equiv d\cdot e.
\end{align*}
\end{proof}

\begin{figure}[b]
    \centering
    \includegraphics{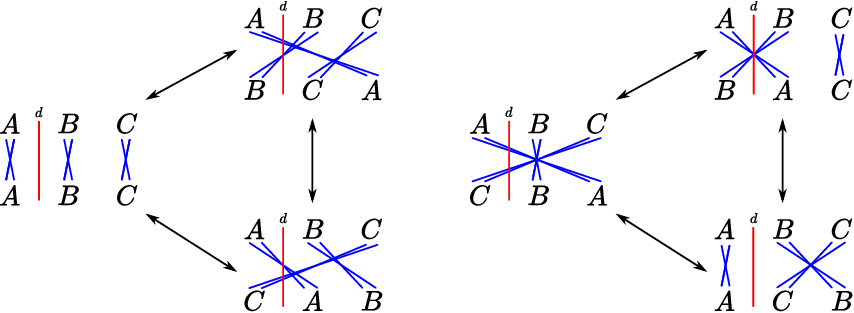}
    \caption{Intersection patterns between the pairs of sides involved in $\mu$.}
    \label{fig:AMN-3-cycles}
\end{figure}

Now consider the three pairs $A = \{x,x+1\}, B=\{y,y+1\}$ and $C=\{z,z+1\}$ with $A\cup B\cup C = \{2i,2i+1,2j,2j+1,2k,2k+1\}$, where $\mu = (2i+1,2j+1,2k+1)(2i,2j,2k)$ is the permutation changing the labels of $O$ to the labels of $O'$. We assume that $x<y<z$. The situation then lies on one of the 3-cycles in Figure~\ref{fig:AMN-3-cycles}. The diagram should be read as follows. The letters $A, B$ and $C$ represent their corresponding side labels on the top and bottom of the origami. The blue lines between the same letters represent the pair of cycles between the side labels. The red curve represents a distinct homology cycle $d$.

We will argue explicitly in the situation where $O$ corresponds to the left-most vertex of the left 3-cycle in Figure~\ref{fig:AMN-3-cycles} and $O'$ corresponds to the top vertex. The other cases are argued similarly. For clarity, we will let $\Phi_{O}$ and $\Phi_{O'}$ denote the quadratic form on $O$ and $O'$ respectively.

Let $|P| = 2p$. So far, we have set $a_{1},b_{1},\ldots,a_{p+1},b_{p+1}$. Next, we choose $a_{p+2}$ and $b_{p+2}$ to be the cycles that were originally $c_{x}$ and $c_{x+1}$, where $A = \{x,x+1\}$. Note that $\Phi_{O}(a_{p+2})=\Phi_{O'}(a_{p+2})$ and $\Phi_{O}(b_{p+2})=\Phi_{O'}(b_{p+2})$. If on $O$ the cycles $a_{p+2}$ and $b_{p+2}$ are orthogonal (resp. non-orthogonal) to the cycles that started as $c_{y}$, $c_{y+1}$, $c_{z}$ and $c_{z+1}$, where $B = \{y,y+1\}$ and $C = \{z,z+1\}$, then they will be non-orthogonal (resp. orthogonal) on $O'$. Therefore, after orthogonalising with respect to $a_{p+2}$ and $b_{p+2}$, the values under $\Phi_{O}$ of the cycles that started as $c_{y}$, $c_{y+1}$, $c_{z}$ and $c_{z+1}$ will differ by 1 modulo 2 from the values under $\Phi_{O'}$. Next, we set $a_{p+3}$ and $b_{p+3}$ to be the cycles that started as $c_{y}$ and $c_{y+1}$. The previous discussion then means $\Phi_{O'}(a_{p+3})\equiv \Phi_{O}(a_{p+3}) + 1$ and $\Phi_{O'}(b_{p+3})\equiv \Phi_{O}(b_{p+3}) + 1$. The cycles $a_{p+3}$ and $b_{p+3}$ are orthogonal (resp. non-orthogonal) to the cycles that started as $c_{z}$ and $c_{z+1}$ on $O$ if and only if they are orthogonal (resp. non-orthogonal) on $O'$ and so the orthogonalisation process affects the values of $\Phi_{O}$ and $\Phi_{O'}$ in the same way. Next, we set $a_{p+4}$ and $b_{p+4}$ to be the cycles that started as $c_{z}$ and $c_{z+1}$. We have $\Phi_{O'}(a_{p+4})\equiv\Phi_{O}(a_{p+4})+1$ and $\Phi_{O'}(b_{p+4})\equiv\Phi_{O}(b_{p+4})+1$. Therefore, we have
\begin{align*}
&\Phi_{O'}(a_{p+3})\Phi_{O'}(b_{p+3}) + \Phi_{O'}(a_{p+4})\Phi_{O'}(b_{p+4}) \\
&\equiv (\Phi_{O}(a_{p+3})+1)(\Phi_{O}(b_{p+3})+1) + (\Phi_{O}(a_{p+4})+1)(\Phi_{O}(b_{p+4})+1) \\
&= \Phi_{O}(a_{p+3})\Phi_{O}(b_{p+3}) + \Phi_{O}(a_{p+3}) + \Phi_{O}(b_{p+3}) \\
&\hspace*{2cm}+ \Phi_{O}(a_{p+4}) + \Phi_{O}(b_{p+4}) + \Phi_{O}(a_{p+4})\Phi_{O}(b_{p+4}) \\
&\equiv \Phi_{O}(a_{p+3})\Phi_{O}(b_{p+3}) + \Phi_{O}(a_{p+4})\Phi_{O}(b_{p+4}).
\end{align*}
So the contribution of the cycles involving the sides labelled by $A, B$ and $C$ to the Arf invariant on $O$ and $O'$ are the same.

Now consider a cycle $d = c_{m}$ for $m\in Q\setminus\{2i,2i+1,2j,2j+1,2k,2k+1\}$. See Figure~\ref{fig:AMN-3-cycles}, for an example. We observe that the number of orthogonalisation steps involving $A, B$ and $C$ that affect $d$ are the same modulo 2 on $O$ and $O'$. Therefore, since each orthogonalisation step changes the value of $\Phi$ by 1 modulo 2, $\Phi_{O}(d) \equiv \Phi_{O'}(d)$ after all of the above orthogonalisation steps.

We can now continue the orthogonalisation process by choosing the remaining $a_{n}$ and $b_{n}$ to be the cycles that started as cycles $c_{2l},c_{2l+1}\in Q\setminus\{2i,2i+1,2j,2j+1,2k,2k+1\}$. The calculations will proceed in the same way for both $O$ and $O'$. It therefore follows that the Arf invariants of $\Phi_{O}$ and $\Phi_{O'}$ are equal. That is, $O$ and $O'$ have the same spin parity, as claimed.
\end{proof}

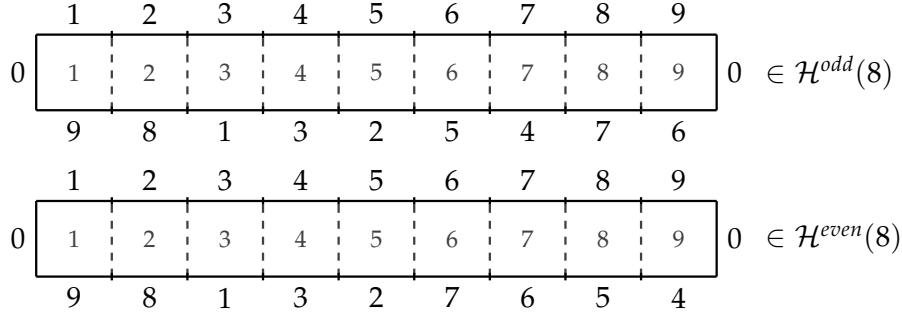
\begin{figure}[t]
    \centering
    \begin{tikzpicture}[scale = 1, line width = 0.3mm]
        \foreach \i in {0,-2.2}{
            \draw (0,\i+0.5) node[left] {$0$};
            \draw (9,\i+0.5) node[right] {$0$};
            \draw (0,\i)--(0,\i+1);
            \draw (0,\i+1)--(9,\i+1);
            \draw (9,\i+1)--(9,\i);
            \draw (9,\i)--(0,\i);
            \foreach \j in {1,...,8}{
                \draw [dashed, color = darkgray] (\j,\i)--(\j,\i+1);
                \draw (\j,\i+0.95)--(\j,\i+1.05);
                \draw (\j,\i-0.05)--(\j,\i+0.05);
                \draw (\j-0.5,\i+1) node[above] {$\j$};
                \draw (\j-0.5,\i+0.5) node[color = darkgray] {\footnotesize ${\j}$};
            }
            \draw (8.5,\i+0.5) node[color = darkgray] {\footnotesize ${9}$};
            \draw (9-0.5,\i+1) node[above] {$9$};
            \draw (0.5,\i) node[below] {$9$};
            \draw (1.5,\i) node[below] {$8$};
            \draw (2.5,\i) node[below] {$1$};
            \draw (3.5,\i) node[below] {$3$};
            \draw (4.5,\i) node[below] {$2$};
        }
        \draw (5.5,0) node[below] {$5$};
        \draw (6.5,0) node[below] {$4$};
        \draw (7.5,0) node[below] {$7$};
        \draw (8.5,0) node[below] {$6$};
        \draw (9.5,0.5) node[right] {$\in\odd(8)$};
        \draw (5.5,-2.2) node[below] {$7$};
        \draw (6.5,-2.2) node[below] {$6$};
        \draw (7.5,-2.2) node[below] {$5$};
        \draw (8.5,-2.2) node[below] {$4$};
        \draw (9.5,-2.2+0.5) node[right] {$\in\even(8)$};
    \end{tikzpicture}
    \caption{Origamis demonstrating the need of the 3-cycles in $\mu$ of Lemma~\ref{lem:key}. Here the label-changing permutation is $(4,6)(5,7)$.}
    \label{fig:non-example}
\end{figure}

The origamis in Figure~\ref{fig:non-example} demonstrate the necessity of the 3-cycles in $\mu$. Indeed, the labels here have been changed by the permutation $(4,6)(5,7)$ and we see that the spin parities are distinct. In fact, an alternate proof of Lemma~\ref{lem:key} would show that multiplying $\tau$ by such a $\mu = (2i,2j)(2i+1,2j+1)$ changes the spin parity by 1 modulo 2.

The choice of the structure of the permutation $\mu$ in Lemma~\ref{lem:key} is also motivated by the fact that all of the origamis constructed by a single choice of method from Section~\ref{sec:gen-const} are related by multiplying by a number of such permutations. More concretely, we have the following proposition.

\begin{proposition}\label{prop:same-spin}
For a fixed construction of Subsection~\ref{subsec:AMN-odd} or Section~\ref{sec:gen-const}, any two origamis $O = (\sigma_{g},\tau)$ and $O' = (\sigma_{g},\tau')$ are related by a sequence of moves $\tau = \tau_{0}\mapsto\tau_{1}\mapsto\cdots\mapsto\tau_{n}=\tau'$ of the form $\tau_{i+1} = \tau_{i}\mu_{i}$ with $\mu_{i} = (2j+1,2k+1,2l+1)(2j,2k,2l)$ such that $\{2j,2j+1,2k,2k+1,2l,2l+1\}\subset Q$, for $Q$ as defined in Lemma~\ref{lem:key}.
\end{proposition}

\begin{proof}
We begin by considering the construction of Subsection~\ref{subsec:AMN-odd}. In this case, for the choice of pairs $(2k_{i},2k_{i}+1)$, $1\leq i\leq g-2$ in the construction, we produce the $[1,1]$-origami $O = (\sigma_{g},\tau)$ with 
\[\tau = (1,3,2k_{1},2k_{2}+1,2k_{3},\ldots,2k_{g-2},2,2k_{1}+1,2k_{2},2k_{3}+1,\ldots,2k_{g-2}+1).\]
If instead we had chosen the order $(2k_{2},2k_{2}+1)$, then $(2k_{1},2k_{1}+1)$, and then $(2k_{i},2k_{i}+1)$ for $3\leq i\leq g-2$, we would have produced the origami $O' = (\sigma_{g},\tau')$ with
\[\tau' = (1,3,2k_{2},2k_{1}+1,2k_{3},\ldots,2k_{g-2},2,2k_{2}+1,2k_{1},2k_{3}+1,\ldots,2k_{g-2}+1).\]
Here we have
\[\tau' = \tau(3,2k_{1}+1,2k_{2}+1)(2,2k_{1},2k_{2}).\]
Similarly, if we had instead chosen the pair $(2k_{g-2},2k_{g-2}+1)$ first, and then $(2k_{i},2k_{i}+1)$ for $1\leq i\leq g-3$, we would obtain the origami $O'' = (\sigma_{g},\tau'')$ with
\[\tau'' = (1,3,2k_{g-2},2k_{1}+1,2k_{2},\ldots,2k_{g-3},2,2k_{g-2}+1,2k_{1},2k_{2}+1,\ldots,2k_{g-3}+1).\]
Here we have
\[\tau'' = \tau(3,2k_{g-3}+1,2k_{g-2}+1)(2,2k_{g-3},2k_{g-2}).\]
Since all permutations of the chosen pairs $(2k_{i},2k_{i}+1)$ can be realised by applying a sequence of moves of the above two forms (this follows from the fact that the transposition $(1,2)$ and the standard $n$-cycle together generate $\Sym_{n}$), we see that any two origamis produced by the construction of Subsection~\ref{subsec:AMN-odd} are connected by a sequence of moves of the form claimed in the statement of the proposition.

The arguments for the constructions of Subsections~\ref{subsec:gen-odd-even},~\ref{subsec:gen-even-odd} and ~\ref{subsec:gen-even-even} follow analogously.
\end{proof}

The following then follows easily from Lemma~\ref{lem:key} and Proposition~\ref{prop:same-spin}.

\begin{corollary}\label{cor:same-spin}
For a fixed construction of Subsection~\ref{subsec:AMN-odd} or Section~\ref{sec:gen-const}, any two origamis produced by this method have the same spin parity.
\end{corollary}

\begin{remark}\label{rem:spin}
It follows from Lemma~\ref{lem:key} that the modified constructions discussed in Remark~\ref{rem:gen-const} produce factorially many origamis in $\calH(2g-2)$ with a single horizontal cylinder half of which have odd spin parity and half of which have even spin parity. Corollary~\ref{cor:same-spin} then says that all of the origamis with simultaneously a single vertical cylinder have the same spin parity.

More concretely, each of the constructions of Subsection~\ref{subsec:AMN-odd} or Section~\ref{sec:gen-const} produces $(|Q|/2)!$ many origamis in $\calH(2g-2)$ with a single horizontal cylinder. Here $Q$ is as defined in Lemma~\ref{lem:key}. It is possible for one of these to lie in the hyperelliptic component if it is not a $[1,1]$-origami. In such a case, we must have $2g+2$ Weierstraß points fixed by the hyperelliptic involution. Exactly two of these lie in the interior of the horizontal cylinder and the single zero of order $2g-2$ (the image of the vertices of the squares) must also be a fixed point. Therefore, the remaining $2g-1$ Weierstraß points must lie at the mid-points of every horizontal edge. Since we have $\tau(1) = 3$ for all $\tau$, one of the Weierstraß points lying in the cylinder is forced to lie in the centre of the second square. This then forces
\[\tau = \begin{pmatrix}
3 & 2 & 1 & 2g-1 & 2g-2 & \cdots & 5 & 4 \\
1 & 2 & 3 & 4 & 5 & \cdots & 2g-2 & 2g-1
\end{pmatrix}\]
which is only achievable by the modification of the construction of Subsection~\ref{subsec:AMN-odd}. By~\cite[Corollary 5]{KZ}, this origami has spin parity $\lfloor\frac{g+1}{2}\rfloor\mod 2$.

Therefore, the modification of the construction of Subsection~\ref{subsec:AMN-odd} produces $\left(\frac{|Q|}{2}\right)! = (g-1)!$ many origamis in $\calH(2g-2)$ with a single horizontal cylinder. One of these lies in the hyperelliptic component, $\frac{(g-1)!}{2}-1$ lie in the component corresponding to spin parity $\lfloor\frac{g+1}{2}\rfloor\mod 2$, and the other $\frac{(g-1)!}{2}$ lie in the component corresponding to spin parity $\lfloor\frac{g+1}{2}\rfloor + 1\mod 2$. The modifications of the constructions of Section~\ref{sec:gen-const}, produce $\left(\frac{|Q|}{2}\right)!$ many origamis in $\calH(2g-2)$ with a single horizontal cylinder divided equally between the odd and even components.
\end{remark}

\subsection{Odd genus odd spin}

Here, we will prove that all of the origamis constructed using the method of Subsection~\ref{subsec:AMN-odd} have odd spin parity. Recall that this is the odd genus case of the original construction of the first author and Menasco-Nieland.

\begin{proposition}\label{prop:odd-odd-2}
For odd $g\geq 3$, the origami $O = (\sigma_{g},\tau)$ with
\[\tau = \begin{pmatrix}
2g-1 & 2g-2 & 1 & 3 & 2 & 5 & 4 & 7 & 6 & \cdots & 2g-3 & 2g-4 \\
1 & 2 & 3 & 4 & 5 & 6 & 7 & 8 & 9 & \cdots & 2g-2 & 2g-1
\end{pmatrix}
\]
given by the construction of Subsection~\ref{subsec:AMN-odd} has odd spin parity.
\end{proposition}

\begin{proof}
The origami $O$ is exactly the origami presented in Figure~\ref{fig:origami-labelling}. In Example~\ref{ex:ARF:AMNodd}, we demonstrated that this origami has odd spin parity.
\end{proof}

By Corollary~\ref{cor:same-spin}, we are able to claim the following.

\begin{corollary}\label{cor:odd-odd}
All of the origamis produced by the constructions of Subsection~\ref{subsec:AMN-odd} have odd spin parity.
\end{corollary}

This is the first half of Theorem~\ref{thm:AMN} of the introduction.

\subsection{Odd genus even spin}

Here we will prove that all of the origamis constructed in Subsection~\ref{subsec:gen-odd-even} have even spin parity. As above, it suffices to prove this for a single origami in this construction.

\begin{proposition}
For odd $g\geq 5$, the origami $O = (\sigma_{g},\tau)$ with
\[\tau = \begin{pmatrix}
9 & 6 & 1 & 8 & 7 & 5 & 4 & 11 & 10 & 13 & 12 & \cdots & 2g-1 & 2g-2 & 3 & 2 \\
1 & 2 & 3 & 4 & 5 & 6 & 7 & 8 & 9 & 10 & 11 & \cdots & 2g-4 & 2g-3 & 2g-2 & 2g-1
\end{pmatrix}
\]
given by the construction of Subsection~\ref{subsec:gen-odd-even} has even spin parity.
\end{proposition}

\begin{proof}
Let $c_{i}$ denote the cycle corresponding to a simple closed between the sides labelled $i$. We have $\Phi(c_{i}) = 1$, for all $0\leq i\leq 2g-1$, and
\begin{align*}
c_{0}\cdot c_{i}&=1 \,\,\,\,\forall\,i\neq 0,\\
c_{1}\cdot c_{i}&=1 \,\,\,\,\forall\,i\in\{0,6,9\},\\
c_{2}\cdot c_{i}&=1 \,\,\,\,\forall\,i\neq 1,\\
c_{3}\cdot c_{i}&=1 \,\,\,\,\forall\,i\neq 1,\\
c_{4}\cdot c_{i}&=1 \,\,\,\,\forall\,i\in\{0,2,3,5,6,7,8,9\},\\
c_{5}\cdot c_{i}&=1 \,\,\,\,\forall\,i\in\{0,2,3,4,6,7,8,9\},\\
c_{6}\cdot c_{i}&=1 \,\,\,\,\forall\,i\in\{0,1,2,3,4,5,9\},\\
c_{7}\cdot c_{i}&=1 \,\,\,\,\forall\,i\in\{0,2,3,4,5,8,9\},\\
c_{8}\cdot c_{i}&=1 \,\,\,\,\forall\,i\in\{0,2,3,4,5,7,9\},\\
c_{9}\cdot c_{i}&=1 \,\,\,\,\forall\,i\in\{0,1,2,3,4,5,6,7,8\},\\
c_{2j}\cdot c_{i}&=1 \,\,\,\,\forall\,10\leq 2j\leq 2g-2,\,i\in\{0,2,3,2j+1\},\,\text{and}\\
c_{2j+1}\cdot c_{i}&=1 \,\,\,\,\forall\,11\leq 2j+1\leq 2g-1,\,i\in\{0,2,3,2j\},
\end{align*}
with all other intersections equal to 0.

If follows from the proof of Lemma~\ref{lem:key} that we can choose $a_{1} = c_{0}$, $b_{1} = c_{1}$, $a_{2} = c_{6}+c_{1}+c_{0}$, $b_{2} = c_{9}+c_{1}+c_{0}$, $a_{3} = c_{4}+c_{1}$, $b_{3} = c_{5}+c_{1}$, $a_{4} = c_{7}+c_{9}+c_{5}+c_{4}+c_{0}$ and $b_{4} = c_{8}+c_{9}+c_{5}+c_{4}+c_{0}$ with $\Phi(a_{1})=\Phi(b_{1})=\Phi(a_{4})=\Phi(b_{4}) = 1$ and $\Phi(a_{2})=\Phi(b_{2})=\Phi(a_{3})=\Phi(b_{3}) = 0$, and such that after all of the orthogonalisation steps have been performed $c_{2}$ and $c_{3}$ have been sent to $c_{2}+c_{8}+c_{7}+c_{5}+c_{4}+c_{1}$ and $c_{3}+c_{8}+c_{7}+c_{5}+c_{4}+c_{1}$, and the remaining cycles $c_{i}$ have been sent to $c_{i}+c_{9}+c_{8}+c_{7}+c_{6}+c_{1}$.

From here we can choose $a_{5} = c_{2}+c_{8}+c_{7}+c_{5}+c_{4}+c_{1}$ and $b_{5} = c_{3}+c_{8}+c_{7}+c_{5}+c_{4}+c_{1}$ with $\Phi(a_{5}) = \Phi(b_{5}) = 0$. Orthogonalising sends $c_{i}+c_{9}+c_{8}+c_{7}+c_{6}+c_{1}$ to $c_{i}+c_{9}+c_{8}+c_{7}+c_{6}+c_{3}+c_{2}+c_{1}$. Finally, for $6\leq j\leq g$, we can set $a_{j} = c_{2j-2}+c_{9}+c_{8}+c_{7}+c_{6}+c_{3}+c_{2}+c_{1}$ and $b_{j} = c_{2j-1}+c_{9}+c_{8}+c_{7}+c_{6}+c_{3}+c_{2}+c_{1}$, giving $\Phi(a_{j}) = \Phi(b_{j}) = 1$. Indeed, these cycles already satisfy the correct symplectic orthogonality conditions.

Hence, we have
\[\sum_{i = 1}^{g}\Phi(a_{i})\Phi(b_{i}) = 1\cdot 1 + 0\cdot 0 + 0\cdot 0 + 1\cdot 1 + 0\cdot 0 + (g-5)(1\cdot 1)\equiv 0,\]
since $g-5$ is even. Therefore, $O$ has even spin parity, as claimed.
\end{proof}

\begin{corollary}
All of the origamis produced by the constructions of Subsection~\ref{subsec:gen-odd-even} have even spin parity.
\end{corollary}

\subsection{Even genus odd spin}

Here, we will prove that all of the origamis constructed in Subsection~\ref{subsec:gen-even-odd} have odd spin parity.

\begin{proposition}
For even $g\geq 4$, the origami $O = (\sigma_{g},\tau)$ with
\[\tau = \begin{pmatrix}
4 & 6 & 1 & 5 & 7 & 9 & 8 & 11 & 10 & \cdots & 2g-1 & 2g-2 & 3 & 2 \\
1 & 2 & 3 & 4 & 5 & 6 & 7 & 8 & 9 & \cdots & 2g-4 & 2g-3 & 2g-2 & 2g-1
\end{pmatrix}
\]
has odd spin parity.
\end{proposition}

\begin{proof}
Let $c_{i}$ denote the cycle corresponding to a simple closed between the sides labelled $i$. We have $\Phi(c_{i}) = 1$, for all $0\leq i\leq 2g-1$, and
\begin{align*}
c_{0}\cdot c_{i}&=1 \,\,\,\,\forall\,i\neq 0,\\
c_{1}\cdot c_{i}&=1 \,\,\,\,\forall\,i\in\{0,4,6\},\\
c_{2}\cdot c_{i}&=1 \,\,\,\,\forall\,i\neq 1,\\
c_{3}\cdot c_{i}&=1 \,\,\,\,\forall\,i\neq 1,\\
c_{4}\cdot c_{i}&=1 \,\,\,\,\forall\,i\in\{0,1,2,3\},\\
c_{5}\cdot c_{i}&=1 \,\,\,\,\forall\,i\in\{0,2,3,6\},\\
c_{6}\cdot c_{i}&=1 \,\,\,\,\forall\,i\in\{0,1,2,3,5\},\\
c_{7}\cdot c_{i}&=1 \,\,\,\,\forall\,i\in\{0,2,3\},\\
c_{2j}\cdot c_{i}&=1 \,\,\,\,\forall\,8\leq 2j\leq 2g-2,\,i\in\{0,2,3,2j+1\},\,\text{and}\\
c_{2j+1}\cdot c_{i}&=1 \,\,\,\,\forall\,9\leq 2j+1\leq 2g-1,\,i\in\{0,2,3,2j\},
\end{align*}
with all other intersections equal to 0.

It follows from the proof of Lemma~\ref{lem:key} that we can choose $a_{1} = c_{0}$, $b_{1} = c_{1}$, $a_{2} = c_{4}+c_{1}+c_{0}$, $b_{2} = c_{5}+c_{1}$, $a_{3} = c_{6}+c_{1}+c_{0}$ and $b_{3} = c_{7}+c_{5}$ with $\Phi(a_{1})=\Phi(b_{1})=1$ and $\Phi(a_{2})=\Phi(b_{2})=\Phi(a_{3})=\Phi(b_{3})=0$, and such that after orthogonalisation $c_{2}$ and $c_{3}$ have been sent to $c_{2}+c_{4}+c_{0}$ and $c_{3}+c_{4}+c_{0}$, and the remaining cycles $c_{i}$ have been sent to $c_{i}+c_{7}$.

We can now choose $a_{4} = c_{2}+c_{4}+c_{0}$ and $b_{4} = c_{3}+c_{4}+c_{0}$ with $\Phi(a_{4}) = \Phi(b_{4}) = 0$. Orthogonalising fixes all other cycles $c_{i}+c_{7}$. Finally, for $5\leq j\leq g$, we can set $a_{j} = c_{2j-2}+c_{7}$ and $b_{j} = c_{2j-1}+c_{7}$ with $\Phi(a_{j}) = \Phi(b_{j}) = 0$.

Hence, we have
\[\sum_{i = 1}^{g}\Phi(a_{i})\Phi(b_{i}) = 1\cdot 1 + (g-1)(0\cdot 0) = 1,\]
and so $O$ has odd spin parity, as claimed.
\end{proof}

\begin{corollary}
All of the origamis produced by the constructions of Subsection~\ref{subsec:gen-even-odd} have odd spin parity.
\end{corollary}

\subsection{Even genus even spin}

Finally, we will prove that all of the origamis constructed in Subsection~\ref{subsec:gen-even-even} have even spin parity.

\begin{proposition}
For even $g\geq 4$, the origami $O = (\sigma_{g},\tau)$ with
\[\tau = \begin{pmatrix}
7 & 5 & 1 & 6 & 4 & 9 & 8 & 11 & 10 & \cdots & 2g-1 & 2g-2 & 3 & 2 \\
1 & 2 & 3 & 4 & 5 & 6 & 7 & 8 & 9 & \cdots & 2g-4 & 2g-3 & 2g-2 & 2g-1
\end{pmatrix}
\]
has odd spin parity.
\end{proposition}

\begin{proof}
Let $c_{i}$ denote the cycle corresponding to a simple closed between the sides labelled $i$. We have $\Phi(c_{i}) = 1$, for all $0\leq i\leq 2g-1$, and
\begin{align*}
c_{0}\cdot c_{i}&=1 \,\,\,\,\forall\,i\neq 0,\\
c_{1}\cdot c_{i}&=1 \,\,\,\,\forall\,i\in\{0,5,7\},\\
c_{2}\cdot c_{i}&=1 \,\,\,\,\forall\,i\neq 1,\\
c_{3}\cdot c_{i}&=1 \,\,\,\,\forall\,i\neq 1,\\
c_{4}\cdot c_{i}&=1 \,\,\,\,\forall\,i\in\{0,2,3,5,6,7\},\\
c_{5}\cdot c_{i}&=1 \,\,\,\,\forall\,i\in\{0,1,2,3,4,7\},\\
c_{6}\cdot c_{i}&=1 \,\,\,\,\forall\,i\in\{0,2,3,4,7\},\\
c_{7}\cdot c_{i}&=1 \,\,\,\,\forall\,i\in\{0,1,2,3,4,5,6\},\\
c_{2j}\cdot c_{i}&=1 \,\,\,\,\forall\,8\leq 2j\leq 2g-2,\,i\in\{0,2,3,2j+1\},\,\text{and}\\
c_{2j+1}\cdot c_{i}&=1 \,\,\,\,\forall\,9\leq 2j+1\leq 2g-1,\,i\in\{0,2,3,2j\},
\end{align*}
with all other intersections equal to 0.

It follows from the proof of Lemma~\ref{lem:key} that we can choose $a_{1} = c_{0}$, $b_{1} = c_{1}$, $a_{2} = c_{5}+c_{1}+c_{0}$, $b_{2} = c_{7}+c_{1}+c_{0}$, $a_{3} = c_{4}+c_{1}$ and $b_{3} = c_{6}+c_{7}+c_{0}$ with $\Phi(a_{1})=\Phi(b_{1})=1$ and $\Phi(a_{2})=\Phi(b_{2})=\Phi(a_{3})=\Phi(b_{3})=0$, and such that orthogonalisation has sent $c_{2}$ and $c_{3}$ to $c_{2}+c_{7}+c_{6}+c_{4}+c_{0}$ and $c_{3}+c_{7}+c_{6}+c_{4}+c_{0}$, and the remaining cycles $c_{i}$ to $c_{i}+c_{7}+c_{5}+c_{4}$.

We can now choose $a_{4} = c_{2}+c_{7}+c_{6}+c_{4}+c_{0}$ and $b_{4} = c_{3}+c_{7}+c_{6}+c_{4}+c_{0}$ with $\Phi(a_{4})=\Phi(b_{4})=1$. Orthogonalising fixes the remaining cycles $c_{i}+c_{7}+c_{5}+c_{4}$. Finally, for $5\leq j\leq g$, we can choose $a_{j} = c_{2j-2}+c_{7}+c_{5}+c_{4}$ and $b_{j} = c_{2j-1}+c_{7}+c_{5}+c_{4}$ with $\Phi(a_{j}) = \Phi(b_{j}) = 1$.

Hence, we have
\[\sum_{i = 1}^{g}\Phi(a_{i})\Phi(b_{i}) = 1\cdot 1 + 0\cdot 0 + 0\cdot 0 + 1\cdot 1 + (g-4)(1\cdot 1)\equiv 0,\]
since $g-4$ is even. Therefore, $O$ has even spin parity, as claimed.
\end{proof}

\begin{corollary}
All of the origamis produced by the constructions of Subsection~\ref{subsec:gen-even-even} have even spin parity.
\end{corollary}

This completes the proof of Theorem~\ref{thm:gen-const}.

\subsection{Even genus AMN origamis}\label{subsec:AMN-even-spin}

Following the notation of Section \ref{subsec:AMN-even}, we now prove the theorem below.

\begin{theorem}\label{thm:AMN-even-g}
    The origamis constructed by Aougab-Menasco-Nieland with even genus $g$ has even spin if and only if $k \equiv 1 \mod 4$ and $\eta(k) \equiv 2 \mod 4$.
\end{theorem}

The proof breaks into three steps, corresponding to the steps of the surgery-like construction outlined at the end of Example \ref{ex:ori:AMNeven}. Consider an origami $O = (\sigma_g,\tau)$ of even genus $g$ constructed by modifying an origami $O' = (\sigma_{g-1},\eta)$ of odd genus $g-1$ according to Aougab-Menasco-Nieland. In particular, let the choice $k\in\{3,5,\ldots,g-3\}\setminus\{\eta^{-1}(1)\}$ be made. First, we find an explicit symplectic basis of $H_1(O',\zz/2\zz)$ via orthogonalisation that can be used to compute the Arf invariant of $O$ in the way discussed in Section \ref{subsec:ortho}. Knowing that $O'$ has odd spin parity, we show that the origami $\widehat{O} = (\sigma_g, \eta(2g-1, 2g-2))$ of even genus $g$ has even spin by extending and modifying the symplectic basis of $H_1(O',\zz/2\zz)$ to one of $H_1(\widehat{O},\zz/2\zz)$. Geometrically, $\widehat{O}$ is obtained by extending two squares to the right of $O'$ where the top and the bottom side of the first square are labeled $2g-2, 2g-1$ respectively and vice versa for the second square. See Figure~\ref{fig:even-AMN-example} above or Figure~\ref{fig:extending-O} below. Observe that $O$ is obtained by swapping the labels $k$ and $2g-2$ on bottom of $\widehat{O}$. Thus, it suffices to show that the swapping operation preserves the spin parity if and only if $k \equiv 1 \mod 4$ and $\eta(k) \equiv 2 \mod 4$. This is done, again, by examining how the symplectic basis of $H_1(\widehat{O},\zz/2\zz)$ transforms to one of $H_1(O,\zz/2\zz)$.

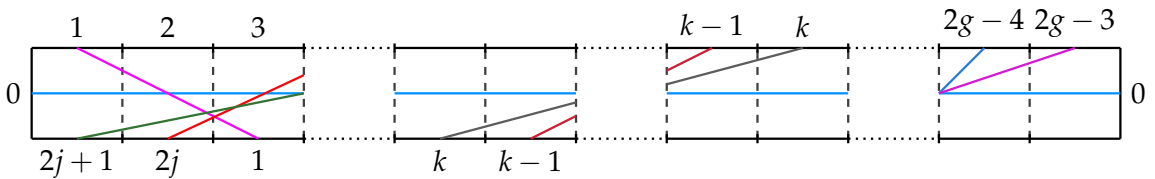
\begin{figure}[b]
    \centering
    \begin{tikzpicture}[scale = 1.2, line width = 0.3mm]
    \draw (0,0)--(0,1);
    \draw (0,1)--(3,1);
    \draw [dotted] (3,1)--(4,1);
    \draw (12,1)--(12,0);
    \draw [dotted] (3,0)--(4,0);
    \draw (0,0)--(3,0);
    \draw (4,0)--(6,0);
    \draw (4,1)--(6,1);
    \draw [dotted] (6,0)--(7,0);
    \draw [dotted] (6,1)--(7,1);
    \draw (7,0)--(9,0);
    \draw (7,1)--(9,1);
    \draw [dotted] (9,1)--(10,1);
    \draw [dotted] (9,0)--(10,0);
    \draw (10,0)--(12,0);
    \draw (10,1)--(12,1);
    
    \foreach \i in {1,2,...,11}{
        \draw [dashed, color = darkgray] (\i,0)--(\i,1);
        \draw (\i,0.05)--(\i,-0.05);
        \draw (\i,1.05)--(\i,0.95);
    }
    \draw (0,0.5) node[left] {$0$};
    \draw (12,0.5) node[right] {$0$};
    \foreach \i in {3,4}{
        \draw (14-\i+0.5,1) node[above] {$2g-\i$};
    }
    \draw (0.5,0) node[below] {$2j+1$};
    \draw (1.5,0) node[below] {$2j$};
    \draw (2.5,0) node[below] {$1$};
    \foreach \i in {1,2,3}{
        \draw (\i-0.5,1) node[above] {$\i$};
    }
    \draw (4.5,0) node[below] {$k$};
    \draw (5.5,0) node[below] {$k-1$};
    \draw (7.5,1) node[above] {$k-1$};
    \draw (8.5,1) node[above] {$k$};
    \draw [color = {rgb,100:red,0;green,100;blue,100}] (0,0.5)--(3,0.5);
    \draw [color = {rgb,100:red,0;green,100;blue,100}] (7,0.5)--(9,0.5);
    \draw [color = {rgb,100:red,0;green,100;blue,100}] (4,0.5)--(6,0.5);
    \draw [color = {rgb,100:red,0;green,100;blue,100}] (10,0.5)--(12,0.5);
    \draw [color = {rgb,100:red,100;green,0;blue,100}] (0.5,1)--(2.5,0);
    %\draw [color = {rgb,100:red,16;green,84;blue,84}] (1.5,1)--(3,0.5);
    %\draw [color = {rgb,100:red,84;green,16;blue,84}] (2.5,1)--(3,0.5);
    %\draw [color = {rgb,100:red,32;green,68;blue,68}] (5.5,1)--(5,0.5);
    %\draw [color = {rgb,100:red,68;green,32;blue,68}] (4.5,1)--(5,0.5);
    \draw [color = {rgb,100:red,80;green,20;blue,20}] (7.5,1)--(7,0.75);
    \draw [color = {rgb,100:red,36;green,64,;blue,36}] (8.5,1)--(7,0.6);
    \draw [color = {rgb,100:red,80;green,20;blue,20}] (6,0.25)--(5.5,0);
    \draw [color = {rgb,100:red,36;green,64,;blue,36}] (6,0.4)--(4.5,0);
    \draw [color = {rgb,100:red,96;green,4;blue,4}] (1.5,0)--(3,0.7);
    \draw [color = {rgb,100:red,20;green,80,;blue,20}] (0.5,0)--(3,0.5);
    \draw [color = {rgb,100:red,16;green,84,;blue,84}] (10.5,1)--(10,0.5);
    \draw [color = {rgb,100:red,84;green,16,;blue,84}] (11.5,1)--(10,0.5);
    \end{tikzpicture}
    \caption{Intersection pattern for an AMN origami $O'$ of odd genus.}\label{fig:intersec-O'}
\end{figure}

\emph{Step 1.} We start by specifying an explicit symplectic basis of $H_1(O',\zz/2\zz)$ via the orthogonalisation method, where $O' = (\sigma_{g-1},\eta)$ is an origami of odd genus $g-1$ constructed according to Aougab-Menasco-Nieland. Following the notations of Section \ref{subsec:ortho}, let $\{c_{i}\ |\ 0 \leq i \leq 2g-3\}$ be the generating set of $H_1(O',\zz/2\zz)$ corresponding to the simple closed curves that travel from the centre of one of the sides labelled $i$ to the centre of the other side labelled $i$. Denote $\eta^{-1}(2) = 2j$ and thus $\eta^{-1}(1) = 2j+1$. Recall that we have $\Phi(c_{i}) = 1$ for all $i$ in the beginning. A representative intersection data in $O'$ is shown in Figure \ref{fig:intersec-O'}. It shall prove useful to partition the labels $2 \leq i \leq 2g-3$ not equal to $k-1,k,2j,2j+1$ into the following subsets.
\begin{align*}
    I_1 &= \{2 \leq i \leq 2g-3\ |\ i > k, \eta(i) > \eta(k), i \not = 2j, 2j+1\}, \\
    I_2 &= \{2 \leq i \leq 2g-3\ |\ i > k, \eta(i) < \eta(k), i \not = 2j, 2j+1\}, \\
    I_3 &= \{2 \leq i \leq 2g-3\ |\ i < k-1, \eta(i) > \eta(k), i \not = 2j, 2j+1\}, \\
    I_4 &= \{2 \leq i \leq 2g-3\ |\ i < k-1, \eta(i) < \eta(k), i \not = 2j, 2j+1\}.
\end{align*}
Observe that all four subsets have even cardinality and let $n_i = |I_i|/2$ for $i = 1,2,3,4$. 

We begin the orthogonalisation method on $O'$. As in the sample calculation of Section \ref{subsec:ortho}, we first choose $a_{1} = c_{0}$ and $b_{1} = c_{1}$ with $a_{1}\cdot b_{1} = 1$ and $\Phi(a_{1}) = 1 = \Phi(b_{1})$. Orthogonalising we see that for $2\leq i\leq 2g-3$,
\begin{equation*}
c_{i} \mapsto \begin{cases}
     c_{i} + (c_{i}\cdot c_{0})c_{1} + (c_{i}\cdot c_{1})c_{0} = c_{i} + c_{1}, & i \not = 2j,2j+1 \\
    c_{i} + (c_{i}\cdot c_{0})c_{1} + (c_{i}\cdot c_{1})c_{0} = c_i + c_1 + c_0, & i = 2j,2j+1
\end{cases}
\end{equation*}

Next, choose $a_2 = c_{k-1} + c_1$ and $b_2 = c_k + c_1$ for the fixed $k \in \{3,5,\ldots,g-3\}\setminus\{\eta^{-1}(1)\}$ above. Again, $a_2 \cdot b_2 = 1$ and $\Phi(a_{2}) = \Phi(c_{k-1}) + \Phi(c_1) + c_{k-1} \cdot c_1 = 1 + 1  + 0 = 0 \mod 2 = \Phi(b_2)$. Orthogonalising we see that for $2\leq i\leq 2g-3$ not equal to $k-1,k,2j,2j+1$,
\begin{equation*}
    c_i + c_1 \mapsto \begin{cases}
         c_i + c_1, & i \in I_1 \cup I_4, \\
         c_i + c_1 + c_{k-1} + c_1 + c_k + c_1 = c_i + c_1 + c_{k-1} + c_k, & i \in I_2 \cup I_3.
    \end{cases}
\end{equation*}
For $i = 2j,2j+1$, there are two cases: 
\begin{equation*}
    c_{i} + c_1 + c_0 \mapsto \begin{cases}
        c_{i} + c_1 + c_0, & \text{if}\ 2j,2j+1 > k, \\
        c_i + c_1 + c_0 + c_{k-1} + c_1 + c_k + c_1 = c_i + c_1 + c_0 + c_{k-1} + c_k, & \text{if}\ 2j,2j+1 < k.
    \end{cases}
\end{equation*}
See Figure \ref{fig:intersec-O'} for an illustration of the relevant intersection patterns in $O'$ that demonstrate the computations above. 

If $2j,2j+1 > k$, choose $a_3 = c_{2j} + c_1 + c_0$ and $b_3 = c_{2j+1} + c_1 + c_0$ with $\Phi(a_3) = \Phi(c_{2j}) + \Phi(c_1) + \Phi(c_0) + c_{2j} \cdot c_1 + c_1 \cdot c_0 + c_{2j} \cdot c_0 = 1 + 1 + 1 + 1 + 1 + 1 = 0 \mod 2 = \Phi(b_3)$. If $2j,2j+1 < k$, similarly, choose $a_3 = c_{2j} + c_1 + c_0 + c_{k-1} + c_k$ and $b_3 = c_{2j+1} + c_1 + c_0 + c_{k-1} + c_k$ with $\Phi(a_3) = 1 = \Phi(b_3)$. In both cases, we note for future use that orthogonalisation yields
\begin{equation} \label{eqn:0}
    \begin{cases}
        c_i + c_1 \mapsto c_i+ c_1 + (c_i\cdot c_{2j} + 1)(a_3+b_3), & i \in I_1 \cup I_4 \\
        c_i + c_1 + c_{k-1} + c_k \mapsto c_i + c_1 + c_{k-1} + c_k + (c_i\cdot c_{2j} + 1)(a_3+b_3), & i \in I_2 \cup I_3.
    \end{cases}
\end{equation}

At this point, by Lemma \ref{lem:pair-up-q-cycles}, we may choose the remaining symplectic basis elements $a_n,b_n$ for $4 \leq n \leq g-1$ to be some remaining cycles coming from $2i, 2i+1$ for $2 \leq 2i \leq 2g-3$ not equal to $k-1, 2j$ in each iteration. As it is known that $O'$ has odd spin by Corollary \ref{cor:odd-odd}, note that if $2j,2j+1 > k$,
\begin{equation*}
    \sum_{n = 4}^{g-1} \Phi(a_n)\Phi(b_n) = 1 - \Phi(a_1)\Phi(b_1) - \Phi(a_2)\Phi(b_2) - \Phi(a_3)\Phi(b_3) = 1 - 1 - 0 - 0 = 0 \mod 2,
\end{equation*}
and if $2j,2j+1 < k$,
\begin{equation*}
     \sum_{n = 4}^{g-1} \Phi(a_n)\Phi(b_n) = 1 - \Phi(a_1)\Phi(b_1) - \Phi(a_2)\Phi(b_2) - \Phi(a_3)\Phi(b_3) = 1 - 1 - 0 - 1 = 1 \mod 2,
\end{equation*}
In summary, we have proven the following lemma.

\begin{lemma}\label{lem:odd-sym-basis}
    The set
    \begin{align*}
        a_1 &= c_0, \qquad\qquad\,\,\,\,\,\, b_1 = c_1, \\
        a_2 &= c_{k-1} + c_1, \qquad b_2 = c_k + c_1, \\
        a_3 &= \begin{cases}
            c_{2j}+c_1+c_0, & \text{if}\ 2j,2j+1 > k \\
            c_{2j} + c_1 + c_0 + c_{k-1} + c_k, & \text{if}\ 2j,2j+1 < k
        \end{cases},\\
        \ b_3 &= \begin{cases}
            c_{2j+1}+c_1+c_0, & \text{if}\ 2j,2j+1 > k \\
            c_{2j+1} + c_1 + c_0 + c_{k-1} + c_k, & \text{if}\ 2j,2j+1 < k
        \end{cases},
    \end{align*}
    and for $4 \leq n \leq g-1$, $a_n,b_n$ the cycles coming from $2i,2i+1$ for $2i \in I_1 \cup I_2 \cup I_3 \cup I_4$ is a symplectic basis of $H_1(O',\zz/2\zz)$. Moreover, the Arf invariants are given by
    \begin{align*}
        \Phi(a_1) = \Phi(b_1) &= 1, \\
        \Phi(a_2) = \Phi(b_2) &= 0, \\
        \Phi(a_3) = \Phi(b_3) &= \begin{cases}
            0,\ & \text{if}\ 2j,2j+1 > k \\
            1,\ & \text{if}\ 2j,2j+1 < k
        \end{cases}, \\
        \sum_{n = 4}^{g-1} \Phi(a_n)\Phi(b_n) &= \begin{cases}
            0,\ & \text{if}\ 2j,2j+1 > k \\
            1,\ & \text{if}\ 2j,2j+1 < k
        \end{cases}.
    \end{align*}
\end{lemma}

\begin{figure}[b]
    \centering
    \begin{tikzpicture}[scale = 1.2, line width = 0.3mm]
    \draw (0,0)--(0,1);
    \draw (0,1)--(3,1);
    \draw [dotted] (3,1)--(4,1);
    \draw (12,1)--(12,0);
    \draw [dotted] (3,0)--(4,0);
    \draw (0,0)--(3,0);
    \draw (4,0)--(6,0);
    \draw (4,1)--(6,1);
    \draw [dotted] (6,0)--(7,0);
    \draw [dotted] (6,1)--(7,1);
    \draw (7,0)--(9,0);
    \draw (7,1)--(9,1);
    \draw [dotted] (9,1)--(10,1);
    \draw [dotted] (9,0)--(10,0);
    \draw (10,0)--(12,0);
    \draw (10,1)--(12,1);
    
    \foreach \i in {1,2,...,11}{
        \draw [dashed, color = darkgray] (\i,0)--(\i,1);
        \draw (\i,0.05)--(\i,-0.05);
        \draw (\i,1.05)--(\i,0.95);
    }
    \draw (0,0.5) node[left] {$0$};
    \draw (12,0.5) node[right] {$0$};
    \foreach \i in {1,2}{
        \draw (12-\i+0.5,1) node[above] {$2g-\i$};
    }
    \draw (10.5,0) node[below] {$2g-1$};
    \draw (11.5,0) node[below] {$2g-2$};
    \draw (0.5,0) node[below] {$2j+1$};
    \draw (1.5,0) node[below] {$2j$};
    \draw (2.5,0) node[below] {$1$};
    \draw (2.5,0) node[below] {$1$};
    \draw (2.5,0) node[below] {$1$};
    \foreach \i in {1,2,3}{
        \draw (\i-0.5,1) node[above] {$\i$};
    }
    \draw (4.5,0) node[below] {$k$};
    \draw (5.5,0) node[below] {$k-1$};
    \draw (7.5,1) node[above] {$k-1$};
    \draw (8.5,1) node[above] {$k$};
    \draw [color = {rgb,100:red,0;green,100;blue,100}] (0,0.5)--(3,0.5);
    \draw [color = {rgb,100:red,0;green,100;blue,100}] (7,0.5)--(9,0.5);
    \draw [color = {rgb,100:red,0;green,100;blue,100}] (4,0.5)--(6,0.5);
    \draw [color = {rgb,100:red,0;green,100;blue,100}] (10,0.5)--(12,0.5);
    \draw [color = {rgb,100:red,100;green,0;blue,100}] (0.5,1)--(2.5,0);
    %\draw [color = {rgb,100:red,16;green,84;blue,84}] (1.5,1)--(3,0.5);
    %\draw [color = {rgb,100:red,84;green,16;blue,84}] (2.5,1)--(3,0.5);
    \draw [color = {rgb,100:red,80;green,20;blue,20}] (7.5,1)--(7,0.75);
    \draw [color = {rgb,100:red,36;green,64,;blue,36}] (8.5,1)--(7,0.6);
    \draw [color = {rgb,100:red,80;green,20;blue,20}] (6,0.25)--(5.5,0);
    \draw [color = {rgb,100:red,36;green,64,;blue,36}] (6,0.4)--(4.5,0);
    \draw [color = {rgb,100:red,96;green,4;blue,4}] (1.5,0)--(3,0.7);
    \draw [color = {rgb,100:red,20;green,80,;blue,20}] (0.5,0)--(3,0.5);
    \draw [color = {rgb,100:red,32;green,68;blue,68}] (10.5,1)--(11.5,0);
    \draw [color = {rgb,100:red,68;green,32;blue,68}] (11.5,1)--(10.5,0);
    \end{tikzpicture}
    \caption{Intersection pattern for $\widehat{O}$.}\label{fig:extending-O}
\end{figure}
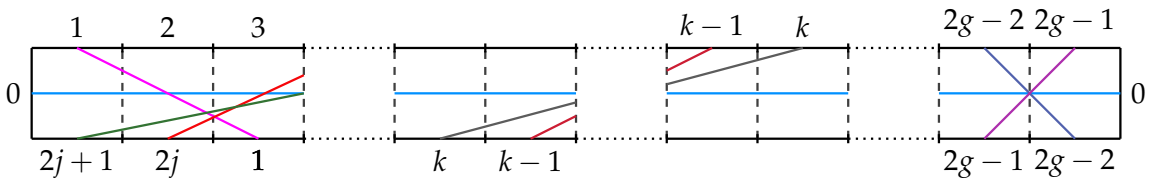

\emph{Step 2.} Next, consider the origami $\widehat{O}$ of even genus $g$ obtained by extending two squares to the right of $O'$ where the top and the bottom side of the first square are labeled $2g-2, 2g-1$ respectively and vice versa for the second square. See Figure~\ref{fig:extending-O}. We specify a symplectic basis of $H_1(\widehat{O}, \zz/2\zz)$ which proves that $\widehat{O}$ has even spin. Again, let $\{c_i\ |\ 0 \leq i \leq 2g-1\}$ be a generating set of $H_1(\widehat{O}, \zz/2\zz)$ of $H_1(\widehat{O},\zz/2\zz)$ corresponding to the simple closed curves that travel from the centre of one of the sides labeled $i$ to the centre of the other side labelled $i$. Observe that the subset $\{c_i\ |\ 0 \leq i \leq 2g-3\}$ is precisely the generating set of $H_1(O', \zz/2\zz)$ we started with before. Again, we use the same subsets $I_1,I_2,I_3,I_4$ to partition the indices $2 \leq i \leq 2g-3$ not equal to $k-1,k,2j,2j+1$. We shall omit the computational details below for efficiency and we refer readers to Figure~\ref{fig:extending-O} for verification.

By the similarity between $\widehat{O}$ and $O'$, choose $a_{1} = c_{0}$, $b_{1} = c_{1}$, $a_2 = c_{k-1} + c_1$, $b_2 = c_k+c_1$ and $c_i$'s are transformed as before for $3 \leq i \leq 2g-3$ not equal to $k-1,k$. For $i = 2g-2,2g-1$, we have
\begin{equation*}
    c_i \mapsto c_i + c_1.
\end{equation*}

Next, choose $a_3 = c_{2g-2} + c_1$ and $b_3 = c_{2g-1}+c_1$ with $\Phi(a_3) = 0 = \Phi(b_3)$. Since $c_{2g-1},c_{2g-2},c_1$ only intersect with $c_0$ among other curves, orthogonalising we see that $c_i + c_1, c_i+c_{k-1}+c_k+c_1$ stay unchanged for $i \in I_1 \cup I_4$ and $i \in I_2 \cup I_3$ respectively. For $i = 2j,2j+1$, we have
\begin{equation*}
    \begin{cases}
        c_i + c_1 + c_0 \mapsto c_i + c_1 + c_0 + c_{2g-1} + c_{2g-2}, & \text{if}\ 2j,2j+1 > k \\
        c_i+c_1+c_0+c_k+c_{k-1} \mapsto c_i+c_1+c_0+c_k+c_{k-1}+c_{2g-1}+c_{2g-2}, & \text{if}\ 2j,2j+1 < k.
    \end{cases}
\end{equation*}

For the next pair of symplectic basis elements, we break into two cases again. If $2j,2j+1 > k$, choose $a_4 = c_{2j} + c_1 + c_0 + c_{2g-1} + c_{2g-2}$, $b_4 = c_{2j+1} + c_1 + c_0 + c_{2g-1} + c_{2g-2}$ with $\Phi(a_4) = 1 = \Phi(b_4)$. If $2j,2j+1 < k$, choose $a_4 = c_{2j} + c_1 + c_0 + c_k + c_{k-1} + c_{2g-1} + c_{2g-2}$ and $b_4 = c_{2j+1} + c_1 + c_0 + c_k + c_{k-1} + c_{2g-1} + c_{2g-2}$ with $\Phi(a_4) = 0 = \Phi(b_4)$. In both cases, orthogonalisation yields
\begin{equation} \label{eqn:1}
    \begin{cases}
        c_i + c_1 \mapsto c_i+ c_1 + (c_i\cdot c_{2j} + 1)(a_3+b_3), & i \in I_1 \cup I_4 \\
        c_i + c_1 + c_{k-1} + c_k \mapsto c_i + c_1 + c_{k-1} + c_k + (c_i\cdot c_{2j} + 1)(a_3+b_3), & i \in I_2 \cup I_3.
    \end{cases}
\end{equation}

By Lemma \ref{lem:pair-up-q-cycles}, we may choose again the remaining symplectic basis elements $a_n,b_n$ for $5 \leq n \leq g$ to be the cycles corresponding to $2i,2i+1$ for $2 \leq 2i \leq 2g-3$ not equal to $k-1,2j$, which are setwise exactly the same as $a_n,b_n$ chosen for $O'$ with $4 \leq n \leq g-1$. In particular, it follows from Lemma \ref{lem:odd-sym-basis} that $\sum_{i = 5}^g \Phi(a_n)\Phi(b_n) = 1$ if $2j,2j+1 > k$ and $0$ if $2j,2j+1 < k$. We then compute that if $2j,2j+1 > k$,
\begin{align*}
    \Phi(a_1)\Phi(b_1) + \Phi(a_2)\Phi(b_2) + \Phi(a_3)\Phi(b_3) + \Phi(a_4)\Phi(b_4) + \sum_{i = 5}^g \Phi(a_n)\Phi(b_n) \\ 
    = 1 + 0 + 0 + 1 + 0 = 0 \mod 2,
\end{align*}
and if $2j, 2j+1 < k$,
\begin{align*}
    \Phi(a_1)\Phi(b_1) + \Phi(a_2)\Phi(b_2) + \Phi(a_3)\Phi(b_3) + \Phi(a_4)\Phi(b_4) + \sum_{i = 5}^g \Phi(a_n)\Phi(b_n) \\
    = 1 + 0 + 0 + 0 + 1 = 0 \mod 2.
\end{align*}
In summary, we have proven the following lemma.

\begin{lemma}
    The origami $\widehat{O}$ defined above has even spin parity. 
\end{lemma}

It will prove useful to be slightly more precise about our choices of $a_n,b_n$ for $5 \leq n \leq g$ and track their Arf invariants. First, we note the following key observation as a direct consequence of Lemma \ref{lem:arf-track}.

\begin{corollary} \label{cor:arf-track}
    For all $2 \leq 2i_n,2i_n+1 \leq 2g-3$ not equal to $k-1,k,2j,2j+1$, if $a_n,b_n$ ($5 \leq n \leq g$) are chosen to be the cycles coming from $2i_n,2i_n+1$ and $c,c'$ are the cycles from $2i_n,2i_n+1$ when $a_3,b_3$ are chosen, then $\Phi(a_n) = \Phi(b_n)$ differ from $\Phi(c) = \Phi(c')$ by $(c_{2i_n} + d_{2i_n}) \cdot a_4 + \sum_{l = 5}^{n-1} (c_{2i_n} + d_{2i_n}) \cdot (c_{2i_l} + d_{2i_l})$.
\end{corollary}

Now, let $a_n,b_n$ be the cycles coming from indices in $I_4$ for $5 \leq n \leq 4+n_4$, in $I_3$ for $5 + n_4 \leq n \leq 4+n_3+n_4$, in $I_2$ for $4+n_3+n_4 \leq n \leq 5+n_2+n_3+n_4$, and in $I_1$ for $4+n_2+n_3+n_4 \leq n \leq 5+n_1+n_2+n_3+n_4 = g$. By Lemma \ref{cor:arf-track}, the Arf invariants of these basis elements are determined by the Arf invariants of their corresponding cycles after choosing $a_4,b_4$, namely,
\begin{align*}
    \Phi(c_i+c_1) = 0,\quad & i \in I_1 \cup I_4, \\
    \Phi(c_i+c_1+c_{k-1}+c_k) = 1,\quad & i \in I_2 \cup I_3,
\end{align*}
and the following intersection data: if $2j,2j+1 > k$,
\begin{align*}
    (c_i+c_1) \cdot (c_{2j}+c_1+c_0+c_{2g-1}+c_{2g-2}) = c_i \cdot c_{2j} + 1,\quad & i \in I_1 \cup I_4, \\
    (c_i+c_1+c_{k-1}+c_k) \cdot (c_{2j}+c_1+c_0+c_{2g-1}+c_{2g-2}) = c_i \cdot c_{2j} + 1, \quad & i \in I_2 \cup I_3,
\end{align*}
and if $2j,2j+1 < k$,
\begin{align*}
    (c_i+c_1) \cdot (c_{2j}+c_1+c_0+c_{k-1}+c_k+c_{2g-2}+c_{2g-1}) = c_i \cdot c_{2j} + 1,\quad & i \in I_1 \cup I_4, \\
    (c_i+c_1+c_{k-1}+c_k) \cdot (c_{2j}+c_1+c_0+c_{k-1}+c_k+c_{2g-2}+c_{2g-1}) = c_i \cdot c_{2j} + 1, \quad & i \in I_2 \cup I_3.
\end{align*}
For $i \in I_4$,
\begin{align*}
    (c_i + c_1) \cdot (c_l+c_1+c_{k-1}+c_k) = c_i \cdot c_l,\quad &l \in I_1 \cup I_2 \\
    (c_i + c_1) \cdot (c_l+c_1) = c_i \cdot c_l,\quad &l \in I_3 \cup I_4;
\end{align*}
for $i \in I_3$,
\begin{align*}
    (c_i+c_1+c_{k-1}+c_k) \cdot (c_l+c_1+c_{k-1}+c_k) = c_i \cdot c_l,\quad &l \in I_1 \cup I_2 \\
    (c_i+c_1+c_{k-1}+c_k) \cdot (c_l+c_1) = c_i \cdot c_l,\quad &l \in I_3;
\end{align*}
for $i \in I_2$.
\begin{align*}
    (c_i+c_1+c_{k-1}+c_k) \cdot (c_l+c_1+c_{k-1}+c_k) = c_i \cdot c_l,\quad &l \in I_1 \cup I_2;
\end{align*}
and for $i \in I_1$,
\begin{align*}
    (c_i + c_1) \cdot (c_l+c_1+c_{k-1}+c_k) = c_i \cdot c_l,\quad &l \in I_1.
\end{align*}

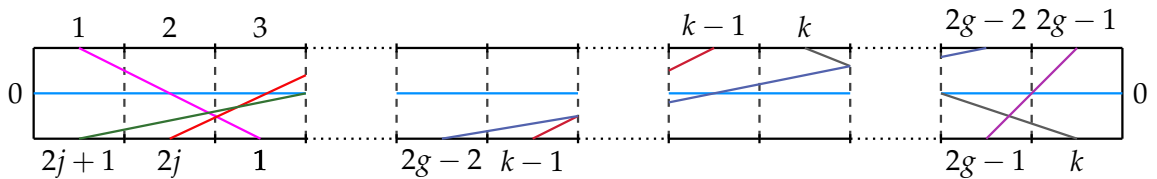
\begin{figure}[b]
    \centering
    \begin{tikzpicture}[scale = 1.2, line width = 0.3mm]
    \draw (0,0)--(0,1);
    \draw (0,1)--(3,1);
    \draw [dotted] (3,1)--(4,1);
    \draw (12,1)--(12,0);
    \draw [dotted] (3,0)--(4,0);
    \draw (0,0)--(3,0);
    \draw (4,0)--(6,0);
    \draw (4,1)--(6,1);
    \draw [dotted] (6,0)--(7,0);
    \draw [dotted] (6,1)--(7,1);
    \draw (7,0)--(9,0);
    \draw (7,1)--(9,1);
    \draw [dotted] (9,1)--(10,1);
    \draw [dotted] (9,0)--(10,0);
    \draw (10,0)--(12,0);
    \draw (10,1)--(12,1);
    
    \foreach \i in {1,2,...,11}{
        \draw [dashed, color = darkgray] (\i,0)--(\i,1);
        \draw (\i,0.05)--(\i,-0.05);
        \draw (\i,1.05)--(\i,0.95);
    }
    \draw (0,0.5) node[left] {$0$};
    \draw (12,0.5) node[right] {$0$};
    \foreach \i in {1,2}{
        \draw (12-\i+0.5,1) node[above] {$2g-\i$};
    }
    \draw (10.5,0) node[below] {$2g-1$};
    \draw (11.5,0) node[below] {$k$};
    \draw (0.5,0) node[below] {$2j+1$};
    \draw (1.5,0) node[below] {$2j$};
    \draw (2.5,0) node[below] {$1$};
    \draw (2.5,0) node[below] {$1$};
    \draw (2.5,0) node[below] {$1$};
    \foreach \i in {1,2,3}{
        \draw (\i-0.5,1) node[above] {$\i$};
    }
    \draw (4.5,0) node[below] {$2g-2$};
    \draw (5.5,0) node[below] {$k-1$};
    \draw (7.5,1) node[above] {$k-1$};
    \draw (8.5,1) node[above] {$k$};
    \draw [color = {rgb,100:red,0;green,100;blue,100}] (0,0.5)--(3,0.5);
    \draw [color = {rgb,100:red,0;green,100;blue,100}] (7,0.5)--(9,0.5);
    \draw [color = {rgb,100:red,0;green,100;blue,100}] (4,0.5)--(6,0.5);
    \draw [color = {rgb,100:red,0;green,100;blue,100}] (10,0.5)--(12,0.5);
    \draw [color = {rgb,100:red,100;green,0;blue,100}] (0.5,1)--(2.5,0);
    %\draw [color = {rgb,100:red,16;green,84;blue,84}] (1.5,1)--(3,0.5);
    %\draw [color = {rgb,100:red,84;green,16;blue,84}] (2.5,1)--(3,0.5);
    \draw [color = {rgb,100:red,80;green,20;blue,20}] (7.5,1)--(7,0.75);
    \draw [color = {rgb,100:red,36;green,64,;blue,36}] (8.5,1)--(9,0.8);
    \draw [color = {rgb,100:red,36;green,64,;blue,36}] (10,0.5)--(11.5,0);
    \draw [color = {rgb,100:red,80;green,20;blue,20}] (6,0.25)--(5.5,0);
    \draw [color = {rgb,100:red,32;green,68;blue,68}] (6,0.25)--(4.5,0);
    \draw [color = {rgb,100:red,96;green,4;blue,4}] (1.5,0)--(3,0.7);
    \draw [color = {rgb,100:red,20;green,80,;blue,20}] (0.5,0)--(3,0.5);
    \draw [color = {rgb,100:red,32;green,68;blue,68}] (10.5,1)--(10,0.9);
    \draw [color = {rgb,100:red,32;green,68;blue,68}] (7,0.4)--(9,0.8);
    \draw [color = {rgb,100:red,68;green,32;blue,68}] (11.5,1)--(10.5,0);
    \end{tikzpicture}
    \caption{Intersection pattern for an AMN origami $O$ of even genus $g$.}\label{fig:AMN-even-origami}
\end{figure}

\emph{Step 3.} Finally, consider the origami $O$ obtained by swapping the labels $k$ and $2g-2$ on bottom of $\widehat{O}$, which represents a generic origami of even genus $g$ constructed by Aougab-Menasco-Nieland. See Figure~\ref{fig:AMN-even-origami}. To complete the proof of Theorem \ref{thm:AMN-even-g}, it suffices to prove the following proposition.

\begin{proposition}
    Let origamis $O, \widehat{O}$ be defined above. $O$ has the same spin parity as $\widehat{O}$ if and only if $k \equiv 1 \mod 4$ and $\eta(k) \equiv 2 \mod 4$.
\end{proposition}
\begin{proof}
    The idea is to find a symplectic basis of $O$ and compare it with that of $\widehat{O}$ found above. Let $\{c_i\ |\ 0 \leq i \leq 2g-1\}$ be a generating set of $H_1(O,\zz/2\zz)$ corresponding to the simple closed curves that travel from the centre of one of the sides labeled $i$ to the centre of the other side labelled $i$, where for $i \not = k,2g-2$, $c_i$'s coincide with those used for $\widehat{O}$. Again, we use the same subsets $I_1,I_2,I_3,I_4$ to partition the indices $2 \leq i \leq 2g-3$ not equal to $k-1,k,2j,2j+1$. We refer readers to Figure \ref{fig:AMN-even-origami} for verification of the computation below.

    First, choose $a_1 = c_0$, $b_1 = c_1$ with $\Phi(a_1) = 1 = \Phi(b_1)$. Orthogonalising we see that for $2\leq i\leq 2g-3$, 
    \begin{equation*}
    c_{i} \mapsto \begin{cases}
        c_{i} + (c_{i}\cdot c_{0})c_{1} + (c_{i}\cdot c_{1})c_{0} = c_{i} + c_{1}, & i \not = 2j,2j+1, \\
        c_{i} + (c_{i}\cdot c_{0})c_{1} + (c_{i}\cdot c_{1})c_{0} = c_i + c_1 + c_0, & i = 2j,2j+1.
    \end{cases}
    \end{equation*}

    Next, choose $a_2 = c_k + c_1$ and $b_2 = c_{2g-1} + c_1$ with $\Phi(a_2) = \Phi(c_k) + \Phi(c_1) + c_k \cdot c_1 = 1 + 1 + 0 = 0 \mod 2 = \Phi(b_2)$. Orthogonalisation yields that
    \begin{equation*}
        c_i + c_1 \mapsto \begin{cases}
            c_i + c_{2g-1}, & i \in I_1 \cup I_2, \\
            c_i + c_1, & i \in I_3 \cup I_4.
        \end{cases}
    \end{equation*}
    For $i = 2j, 2j+1$, we have
    \begin{equation*}
        c_i + c_1 + c_0 \mapsto \begin{cases}
            c_i + c_0 + c_k, & \text{if}\ 2j,2j+1 > k \\
            c_i + c_1 + c_0 + c_k + c_{2g-1}, & \text{if}\ 2j,2j+1 < k.
        \end{cases}
    \end{equation*}

    Then, choose $a_3 = c_{k-1} + c_1$ and $b_3 = c_{2g-2} + c_{2g-1}$ with $\Phi(a_3) = \Phi(c_{k-1}) + \Phi(c_1) + c_{k-1} \cdot c_1 = 1 + 1 + 0 = 0 \mod 2 = \Phi(b_3)$. Orthogonalising we see that 
    \begin{equation*}
        \begin{cases}
            c_i + c_{2g-1} \mapsto c_i + c_{2g-1} + c_{k-1} + c_1, & i \in I_1, \\
            c_i + c_{2g-1} \mapsto c_i + c_{2g-2}, & i \in I_2 \\
            c_i + c_1 \mapsto c_i + c_{k-1} + c_{2g-2} + c_{2g-1}, & i \in I_3 \\
            c_i + c_1 \mapsto c_i + c_1, & i \in I_4.
        \end{cases}
    \end{equation*}
    For $i = 2j, 2j+1$, we have
    \begin{equation*}
        \begin{cases}
            c_i + c_0 + c_k \mapsto c_i + c_0 + c_k, & \text{if}\ 2j,2j+1 > k, \\
            c_i + c_1 + c_0 + c_k + c_{2g-1} \mapsto c_{i} + c_1 + c_0 + c_k + c_{2g-2}, & \text{if}\ 2j,2j+1 < k.
        \end{cases}
    \end{equation*}

    For the next pair of symplectic basis elements, we break into two cases. If $2j,2j+1 > k$, choose $a_4 = c_{2j} + c_0 + c_k$ and $b_4 = c_{2j+1} + c_0 + c_k$ with $\Phi(a_4) = 0 = \Phi(b_4)$. If $2j,2j+1 < k$, choose $a_4 = c_{2j} + c_1 + c_0 + c_k + c_{2g-2}$ and $b_4 = c_{2j+1} + c_1 + c_0 + c_k + c_{2g-2}$ with $\Phi(a_4) = 1 = \Phi(b_4)$.

    Finally, by Lemma \ref{lem:pair-up-q-cycles}, we choose in either case the remaining symplectic basis elements $a_n,b_n$ for $5 \leq n \leq g$ to be the cycles coming from $2i,2i+1$ for $2 \leq 2i \leq 2g-3$ not equal to $k-1,2j$, in the same order specified at the end of step 2. To complete the proof, it remains to compare the Arf invariants of $a_n,b_n$ for $O,\widehat{O}$.

    First, observe that $\Phi(a_n)\Phi(b_n)$ are the same for $n = 1,2,3$ whereas $\Phi(a_4)\Phi(b_4)$ differs by $1$ for $O$ and $O'$. By Lemma \ref{cor:arf-track}, the Arf invariants of $a_n,b_n$ for $O$ and $5 \leq n \leq g$ are determined by
    \begin{align*}
        \Phi(c_i+c_{2g-1}+c_{k-1}+c_1) = 0,\quad & i \in I_1, \\
        \Phi(c_i+c_{2g-2}) = 0,\quad & i \in I_2, \\
        \Phi(c_i+c_{k-1}+c_{2g-2}+c_{2g-1}) = 1,\quad & i \in I_3, \\
        \Phi(c_i+c_1) = 0,\quad & i \in I_4,
        \end{align*}
    and the following intersection data: if $2j,2j+1 > k$,
    \begin{align*}
        (c_i+c_{2g-1}+c_{k-1}+c_1) \cdot (c_{2j}+c_0+c_k) = c_i\cdot c_{2j},\quad  & i \in I_1,\\
        (c_i+c_{2g-2}) \cdot (c_{2j}+c_0+c_k) = c_i\cdot c_{2j},\quad & i \in I_2,\\
        (c_i+c_{k-1}+c_{2g-2}+c_{2g-1}) \cdot (c_{2j}+c_0+c_k) = c_i\cdot c_{2j}+1,\quad & i \in I_3, \\
        (c_i+c_1) \cdot (c_{2j}+c_0+c_k) = c_i\cdot c_{2j}+1,\quad & i \in I_4, 
    \end{align*}
    and if $2j,2j+1 < k$,
    \begin{align*}
        (c_i+c_{2g-1}+c_{k-1}+c_1) \cdot (c_{2j}+c_1+c_0+c_k+c_{2g-2}) = c_i\cdot c_{2j}+1,\quad  & i \in I_1,\\
        (c_i+c_{2g-2}) \cdot (c_{2j}+c_1+c_0+c_k+c_{2g-2}) = c_i\cdot c_{2j},\quad & i \in I_2,\\
        (c_i+c_{k-1}+c_{2g-2}+c_{2g-1}) \cdot (c_{2j}+c_1+c_0+c_k+c_{2g-2}) = c_i\cdot c_{2j},\quad & i \in I_3, \\
        (c_i+c_1) \cdot (c_{2j}+c_1+c_0+c_k+c_{2g-2}) = c_i\cdot c_{2j}+1,\quad & i \in I_4.
    \end{align*}
    For $i \in I_4$,
    \begin{align*}
        (c_i + c_1) \cdot (c_l+c_{2g-1}+c_{k-1}+c_1) = c_i \cdot c_l,\quad &l \in I_1,\\
        (c_i + c_1) \cdot (c_l+c_{k-1}+c_{2g-2}+c_{2g-1}) = c_i \cdot c_l,\quad &l \in I_2,\\
        (c_i + c_1) \cdot (c_l+c_{2g-2}) = c_i \cdot c_l,\quad &l \in I_3,\\
        (c_i + c_1) \cdot (c_l+c_1) = c_i \cdot c_l,\quad &l \in I_4;
    \end{align*}
    for $i \in I_3$,
    \begin{align*}
        (c_i+c_{k-1}+c_{2g-2}+c_{2g-1}) \cdot (c_l+c_{2g-1}+c_{k-1}+c_1) = c_i \cdot c_l + 1,\quad &l \in I_1,\\
        (c_i+c_{k-1}+c_{2g-2}+c_{2g-1}) \cdot (c_l+c_{2g-2}) = c_i \cdot c_l+1,\quad &l \in I_2, \\
        (c_i+c_{k-1}+c_{2g-2}+c_{2g-1}) \cdot (c_l+c_{k-1}+c_{2g-2}+c_{2g-1}) = c_i \cdot c_l,\quad &l \in I_3;
    \end{align*}
    for $i \in I_2$,
    \begin{align*}
        (c_i+c_{2g-2}) \cdot (c_l+c_{2g-1}+c_{k-1}+c_1) = c_i \cdot c_l+1,\quad &l \in I_1, \\
        (c_i+c_{2g-2}) \cdot (c_l+c_{2g-2}) = c_i \cdot c_l,\quad &l \in I_2;
    \end{align*}
    and for $i \in I_1$,
    \begin{align*}
        (c_i+c_{2g-1}+c_{k-1}+c_1) \cdot (c_l+c_{2g-1}+c_{k-1}+c_1) = c_i \cdot c_l,\quad & l \in I_1.
    \end{align*}
    Comparing with the computation at the end of step 2, it follows from Lemma \ref{cor:arf-track} that if $2j,2j+1 > k$, $\sum_{n = 5}^{n_4} \Phi(a_n)\Phi(b_n)$, $\sum_{n = 5+n_4}^{4+n_3+n_4} \Phi(a_n)\Phi(b_n)$ are the same, whereas $\sum_{n = 5+n_3+n_4}^{4+n_2+n_3+n_4} \Phi(a_n)\Phi(b_n)$, $\sum_{n = 5+n_2+n_3+n_4}^{g}\Phi(a_n)\Phi(b_n)$ differ by $n_2n_3$ and $n_1(1+n_2+n_3)$ respectively for $O,\widehat{O}$. The same holds for the case when $2j,2j+1 < k-1$ except that now $\sum_{n = 5+n_4}^{4+n_3+n_4} \Phi(a_n)\Phi(b_n)$ and $\sum_{n = 5+n_2+n_3+n_4}^{g}\Phi(a_n)\Phi(b_n)$ differ by $n_3$ and $n_1(n_2+n_3)$ respectively instead. We break the remaining proof into two cases.\\

    \paragraph{\textbf{Case 1: $\boldsymbol{2j,2j+1 > k}$}}
        The Arf invariants of $O,\widehat{O}$ differ by $1+n_2n_3+n_1(1+n_2+n_3) = 1+n_3(n_1+n_2)+n_1(1+n_2)$ and their spin parity are the same if and only if $n_3(n_1+n_2)+n_1(1+n_2)$ is odd.

        Suppose $k \equiv 1 \mod 4$ and $\eta(k) \equiv 2 \mod 4$. The first assumption implies that the number of labels from $2$ to $k-2$ and from $k+1$ to $2g-3$ excluding $2j,2j+1$ are both $2 \mod 4$. It follows that $n_1+n_2$ and $n_3+n_4$ are both odd.
        
        The second assumption implies that the number of labels from $1$ to $k$ (both exclusive) and from $k-1$ to $2g-1$ (both exclusive) at bottom of $O$ are both $2 \mod 4$. It follows that $n_1+n_3$ and $n_2+n_4$ are both odd.
        
        Consider now $n_3(n_1+n_2)+n_1(1+n_2)$. If $n_1$ is even, then $n_1(1+n_2)$ is even and $n_3$ is odd. Thus, $n_3(n_1+n_2)$ is odd and the sum $n_3(n_1+n_2)+n_1(1+n_2)$ is odd as desired. If $n_1$ is odd, then both $n_2$ and $n_3$ are even. This implies that $n_1(1+n_2)$ is odd and $n_3(n_1+n_2)$ is even. Again, $n_3(n_1+n_2)+n_1(1+n_2)$ is odd as desired.

        Conversely, suppose $n_3(n_1+n_2)+n_1(1+n_2)$ is odd. If $n_3(n_1+n_2)$ is odd and $n_1(1+n_2)$ is even, the first assumption implies that $n_3, n_1+n_2$ are odd and the second assumption implies that $n_1$ is even or $n_1,n_2$ are both odd. However, $n_1+n_2$ being odd rules out the second possibility and thus, $n_1$ is even. It follows that $n_1 + n_2, n_1 + n_3$ are both odd, which implies that $k \equiv 1 \mod 4$ and $p \equiv 2 \mod 4$. If $n_3(n_1+n_2)$ is even and $n_1(1+n_2)$ is odd, the first assumption implies that $n_1$ is odd and $n_2$ is even so that $n_1+n_2$ is odd; the second assumption implies that $n_3$ is even so that $n_1 + n_3$ is odd. Again, we have $k \equiv 1 \mod 4$ and $p \equiv 2 \mod 4$ as desired.\\

    \paragraph{\textbf{Case 2: $\boldsymbol{2j,2j+1 < k}$}}
        Similarly, the Arf invariants of $O,\widehat{O}$ differ by $1+n_3+n_2n_3+n_1(n_2+n_3) = 1+n_3(1+n_1+n_2)+n_1n_2$ and their spin parity are the same if and only if $n_3(1+n_1+n_2)+n_1n_2$ is odd.

        Suppose $k \equiv 1 \mod 4$ and $\eta(k) \equiv 2 \mod 4$. The first assumption implies that the number of labels from $2$ to $k-2$ excluding $2j,2j+1$ and from $k+1$ to $2g-3$ are both $0 \mod 4$. It follows that $n_1+n_2$ and $n_3+n_4$ are both even.
        
        The second assumption implies that the number of labels from $1$ to $k$ (both exclusive) and from $k-1$ to $2g-1$ (both exclusive) at bottom of $O$ are both $0 \mod 4$. It follows that $n_1+n_3$ and $n_2+n_4$ are both odd.
        
        Consider now $n_3(1+n_1+n_2)+n_1n_2$. If $n_1$ is even, then $n_1n_2$ is even and $n_3$ is odd. Thus, $n_3(1+n_1+n_2)$ is odd and the sum $n_3(1+n_1+n_2)+n_1n_2$ is odd as desired. If $n_1$ is odd, then $n_2$ is odd and $n_3$ are even. This implies that $n_1n_2$ is odd and $n_3(1+n_1+n_2)$ is even. Again, $n_3(1+n_1+n_2)+n_1n_2$ is odd as desired.

        Conversely, suppose $n_3(1+n_1+n_2)+n_1n_2$ is odd. If $n_3(1+n_1+n_2)$ is odd and $n_1n_2$ is even, the first assumption implies that $n_3$ is odd and $n_1+n_2$ is even; the second assumption implies that $n_1$ is even or $n_2$ is even. Thus, $n_1+n_2$ being even yields that $n_1,n_2$ are both even. It follows that $n_1 + n_2$ is even and $n_1 + n_3$ is odd, which implies that $k \equiv 1 \mod 4$ and $p \equiv 2 \mod 4$ in this case. If $n_3(1+n_1+n_2)$ is even and $n_1n_2$ is odd, the first assumption implies that $n_3$ is even or $n_1+n_2$ is odd; the second assumption implies that $n_1,n_2$ are both odd. Thus, $n_1+n_2$ cannot be odd so that $n_3$ is even. Again, we have $n_1+n_2$ even and $n_1+n_3$ odd so that $k \equiv 1 \mod 4$ and $p \equiv 2 \mod 4$ as desired.
\end{proof}

To complete the proof of Theorem~\ref{thm:AMN}. We need to count how many of the $(g-3)(g-3)!$ even genus AMN origamis have even spin.

Let $g = 2m$ and consider an AMN origami of genus $g-1$ that can be used to build a genus $g$ AMN origami. There are $m-1$ choices of $k$ with $5\leq k\leq 2g-3$ and $k\equiv 1\!\!\mod 4$. There are also $m-2$ positions $\eta(k)$ with $6\leq \eta(k) \leq 2g-3$ and $\eta(k)\equiv 2\!\!\mod 4$. Given a genus $g-1$ AMN origami, we can record the $(m-2)$-tuple of values $k\equiv 1\!\!\mod 4$ with $\eta(k)\equiv 2\!\!\mod 4$. Note that some of the entries in this $(m-2)$-tuple may be empty. For example, the origami of genus 7 in Figure~\ref{fig:tuple-example} gives the 2-tuple $(\underline{\,\,\,\,}\,,13)$ since there is no $k\equiv 1\!\!\mod 4$ with $\eta(k) = 6$ but $13$ has $\eta(13) = 10$. 

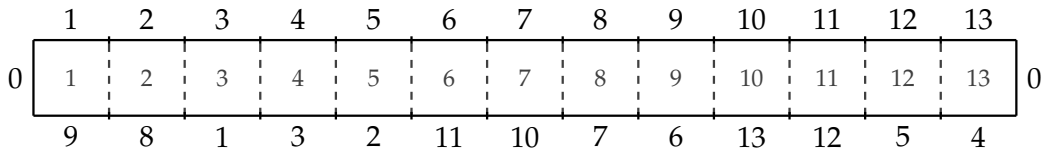
\begin{figure}[b]
    \centering
    \begin{tikzpicture}[scale = 1, line width = 0.3mm]
        \foreach \i in {0}{
            \draw (0,\i)--(0,1+\i);
            \draw (0,1+\i)--(13,1+\i);
            \draw (13,1+\i)--(13,\i);
            \draw (13,\i)--(0,\i);
            \draw (0,0.5) node[left] {$0$};
            \draw (13,0.5) node[right] {$0$};
            \foreach \j in {1,...,12}{
                \draw [dashed, color = darkgray] (\j,\i)--(\j,1+\i);
                \draw (\j,0.95+\i)--(\j,1.05+\i);
                \draw (\j,-0.05+\i)--(\j,0.05+\i);
            }
            \foreach \j in {1,...,13}{
                \draw (\j-0.5,1+\i) node[above] {$\j$};
                \draw (\j-0.5,0.5) node[color = darkgray] {\footnotesize $\j$};
            }
        }
        \draw (0.5,0) node[below] {$9$};
        \draw (1.5,0) node[below] {$8$};
        \draw (2.5,0) node[below] {$1$};
        \draw (3.5,0) node[below] {$3$};
        \draw (4.5,0) node[below] {$2$};
        \draw (5.5,0) node[below] {$11$};
        \draw (6.5,0) node[below] {$10$};
        \draw (7.5,0) node[below] {$7$};
        \draw (8.5,0) node[below] {$6$};
        \draw (9.5,0) node[below] {$13$};
        \draw (10.5,0) node[below] {$12$};
        \draw (11.5,0) node[below] {$5$};
        \draw (12.5,0) node[below] {$4$};
    \end{tikzpicture}
    \caption{The origami above corresponds to the tuple $(\underline{\,\,\,\,}\,,13)$.}
    \label{fig:tuple-example}
\end{figure}

For every non-empty entry in the $(m-2)$-tuple we can produce an even genus AMN origami with even spin, so the count should proceed as follows. First, determine the number $t_{i}$ of $(m-2)$-tuples that have $i$ non-empty entries. Second, determine the number $o_{i}$ of odd genus AMN origamis that can give rise to such an $(m-2)$-tuple. Finally, we calculate the sum
\[\sum_{i = 1}^{m-2} i\cdot t_{i} \cdot o_{i}.\]
This will give us the number of even genus AMN origamis that have even spin.

The value of $t_{i}$ is $\binom{m-2}{i}\binom{m-1}{i}i!$ since we have $\binom{m-2}{i}$ ways to choose the non-empty positions, and $\binom{m-1}{i}i!$ ways to choose and order the values of $k \equiv 1\!\!\mod 4$ appearing in those positions.

Now we calculate the value of $o_{i}$. Suppose we have an $(m-2)$-tuple with $i$ non-empty entries. When building an associated odd genus AMN origami, we see that $i$ non-empty entries in the tuple force us to place $i$ pairs of the form $(k,k-1)$ with $k\equiv 1\!\!\mod 4$ in the positions corresponding to $\eta(k)\equiv 2\!\!\mod 4$. This process can never close up any cycles. Next, we must fill the empty entries with values $j\equiv 3\!\!\mod 4$. There are $m-1$ such values $3\leq j\leq 2g-3$. Here we must be careful about closing cycles. Indeed, we must not place $(j,j-1)$ under $(j-1,j)$. However, this is the only restriction. Hence, we have $(m-2)$ ways to fill the first empty entry, $(m-3)$ ways for the second and so on until we have $(i+1)$ ways to fill the final empty entry. That is, we have $\frac{(m-2)!}{i!}$ ways to satisfy the conditions forced by the $(m-2)$-tuple. Finally, we still have $m$ pairs to place. We can proceed along the lines of the normal odd genus AMN construction. That is, place a pair so that it does not complete a cycle, and continue. So we have $(m-1)!$ ways to complete the origami. This gives
\[o_{i} = \frac{(m-1)!(m-2)!}{i!}.\]

For example, the tuple $(\underline{\,\,\,\,}\,,13)$ in Figure~\ref{fig:tuple-example} forces us to place (13,12) below (10,11). Next, we must place one of the pairs (3,2), (7,6), or (11,10) under (6,7). We cannot choose (7,6), but we have two other choices. In the figure, we choose (11,10). Next, we place the remaining pairs (3,2), (5,4), (7,6) and (9,8). First, we place (3,2) under (4,5). We could have placed it under (8,9) or (12,13) instead. Next we place (5,4) under (12,13). We could have placed it under (8,9) instead. Now, (13,12) is already placed so we place the next pair along the cycle [$(13,12)\rightarrow(11,10)\rightarrow(7,6)$] that has not been placed which is (7,6). We must place this pair under (8,9) which then forces (9,8) under (1,2), and we are done.

As a sanity check, we can see that
\[\sum_{i = 0}^{m-2}t_{i}\cdot o_{i} = (2(m-2)+1)! = (g-3)!\]
which, recalling that $g-1$ is odd, is the correct number of genus $g-1$ AMN origamis.

Finally, we calculate the number of even spin even genus AMN origamis to be
\begin{align*}
    \sum_{i = 1}^{m-2} i\cdot t_{i} \cdot o_{i} &= \sum_{i = 1}^{m-2}i\cdot \binom{m-2}{i}\binom{m-1}{i}i! \cdot \frac{(m-1)!(m-2)!}{i!} \\
    &= (m-1)!(m-2)!\sum_{i = 1}^{m-2}i\binom{m-2}{i}\binom{m-1}{i} \\
    &= (m-1)!(m-2)!(m-2)\binom{2(m-2)}{m-2} \\
    &= (m-1)(m-2)(2(m-2))! \\
    &= \left(\frac{g}{2}-1\right)\left(\frac{g}{2}-2\right)(g-4)!.
\end{align*}
Hence, there are
\[(g-3)(g-3)! - \left(\frac{g}{2}-1\right)\left(\frac{g}{2}-2\right)(g-4)! = \frac{1}{4}(3(g-3)^2+1)(g-4)!\]
even genus AMN origamis with odd spin parity. This completes the proof of Theorem~\ref{thm:AMN}.

%%%%%%%%%%%%%%%%%%%%%%%%%%%%%%%%%%%%%%%%%%%%%%%%%%%%%%

\section{Monodromy calculations}\label{sec:mono}

In this section, we investigate the monodromy groups of the origamis considered in this paper (i.e., those of Sections~\ref{sec:AMN} and~\ref{sec:gen-const}). In doing so, we will prove Theorem~\ref{thm:monodromy}.

First, we will show that all of the monodromy groups are primitive. Following this, we will show that the monodromy groups of the odd genus AMN origamis of Subsection~\ref{subsec:AMN-odd} contain 3-cycles. It then follows from the result of Jordan below that the monodromy groups of the odd genus AMN origamis are the alternating group. Indeed, since all of the monodromy groups are generated by $n$-cycles for $n$ odd, they are already subgroups of $\Alt_{n}$. The following theorem gives the other inclusion.

\begin{theorem}[\cite{Jord}]\label{thm:Jordan}
Let $G\leqslant \Sym_{n}$ be a primitive permutation group.
\begin{itemize}
    \item[1.] If $G$ contains a transposition, then $G=\Sym_{n}$.
    \item[2.] If $G$ contains a 3-cycle, then $\Alt_{n}\leqslant G$.
    \item[3.] If $G$ contains a $p$-cycle for prime $p\leq n-3$, then $\Alt_{n}\leqslant G$.
\end{itemize}
\end{theorem}

We should note that there is also the more recent result of Jones~\cite[Corollary 1.3]{Jones2} that strengthens Jordan's theorem by removing the primality condition in part 3. Unfortunately, we have not been able to take advantage of this stronger result.

We remark that it is necessary to consider a case by case proof below. This is because there is no hope for a general result of the form ``All minimal $[1,1]$-origamis in the minimal stratum are primitive." Indeed, it can be checked that if a minimal $[1,1]$-origami in the minimal stratum is not primitive, then it must cover another minimal $[1,1]$-origami in a minimal stratum of lower genus ($\sigma,\tau,$ and $[\sigma,\tau]$, being $(2g-1)$-cycles, must all act as $k$-cycles on the set of blocks in a block system of size $k$). The lowest genus surface that can be covered is then genus $3$ with $5$ squares. Since minimal $[1,1]$-origamis have an odd number of squares, the smallest odd degree cover will have $15$ squares and so genus $8$ (recall that $2g-1 = 15$, here) and, indeed, the genus $8$ $[1,1]$-origamis
\[O_{1} = ((1,2,3,4,5,6,7,8,9,10,11,12,13,14,15),(1,3,15,4,12,6,8,10,9,7,11,13,5,14,2))\]
in $\odd(14)$ and
\[O_{2} = ((1,2,3,4,5,6,7,8,9,10,11,12,13,14,15),(1,3,5,14,7,6,8,10,9,12,11,13,15,4,2))\]
in $\even(14)$ are both covers of the origami
\[O' = ((1,2,3,4,5),(1,3,5,4,2))\]
in $\odd(4)$.

For genus less than 8, we have the following.

\begin{proposition}
Let $3\leq g\leq 7$, then every minimal $[1,1]$-origami in $\calH(2g-2)$ is primitive.
\end{proposition}

\begin{proof}
Let $O$ be a minimal $[1,1]$-origami in $\calH(2g-2)$ for some $3\leq g\leq 7$. Then $\Mon(O)\leq\Sym_{2g-1}$, with $2g-1\in\{5,7,9,11,13\}$. If $2g-1$ is prime, then any block system of $\{1,\ldots,2g-1\}$ must be trivial (by transitivity all blocks have the same size, so this size must be a divisor of $2g-1$), and so $\Mon(O)$ is primitive. We are left to handle the case of $2g-1 = 9$. If such an origami is not primitive then, from the discussion above, such an origami must be a cover of a $[1,1]$-origami with 3 squares and a single singularity of order $2$. However, no such origami exists. Indeed, there do not exist two 3-cycles in $\Sym_{3}$ whose commutator is also a 3-cycle. (Topological proofs of this fact can be found in~\cite{J1}.)
\end{proof}

This can also just be checked computationally by iterating through the $(2g-1)$-cycles $\tau\in\Sym_{2g-1}$, $3\leq g\leq 7$, whose commutator with the standard $(2g-1)$-cycle $\sigma_{g}$ is also a $(2g-1)$-cycle, and verifying that the subgroup $\langle \sigma_{g}, \tau\rangle$ that they generate is primitive.

\subsection{Primitivity}\label{subsec:primitivity}

Here we prove that all of the origamis constructed in Sections~\ref{sec:AMN} and~\ref{sec:gen-const} have primitive monodromy groups.

\subsubsection{Odd genus AMN origamis}

\begin{proposition}\label{prop:prim-odd-odd}
All of the origamis constructed in Subsection~\ref{subsec:AMN-odd} are primitive.
\end{proposition}

\begin{proof}
Let $O = (\sigma_{g},\tau)$ be an origami constructed using the method of Subsection~\ref{subsec:AMN-odd}. Let $\Mon(O)$ be its monodromy group and suppose that $\Mon(O)$ admits a block system of size $k$ for some $k\neq 1,2g-1$; that is, a non-trivial block system of size $k$.

Since $\sigma_{g}$, being a $(2g-1)$-cycle, must act as a $k$-cycle on these blocks we can enumerate them as $\{\Delta_{i} | 1\leq i\leq k\}$ with $\Delta_{i} = \{j\,|\,j\equiv i\mod k\}$.

Recall that 
\[\nu:=\tau^{-1}\sigma_{g}^{-1}\tau\sigma_{g} = (1,3,5,\ldots,2g-3,2g-1,\rho^{-1}(1),\rho^{-3}(1),\ldots,\rho^{2}(1)),\]
where $\rho := \tau^{-1}\sigma_{g}\tau$ was the $(2g-1)$-cycle given by the top row in the construction of $\tau$. Since $2g-1$ is odd, $k|2g-1$, and $k<2g-1$, we have $2g-1\geq 3k$. Therefore, we see that $\nu$ acts on the blocks as the $k$-cycle
\[ (\Delta_{1},\Delta_{3},\ldots,\Delta_{k},\Delta_{2},\Delta_{4},\ldots,\Delta_{k-1}).\]
Indeed, $1\in\Delta_{1}$ is sent to $3\in\Delta_{3}$, $3$ is sent to $5\in\Delta_{5}$, and so on until the $k$-cycle closes up when $2k-1\in\Delta_{k-1}$ is sent to $2k+1\in\Delta_{1}$. This now forces $\rho^{-1}(1)$ to be in $\Delta_{2}$, since $2g-1\in\Delta_{k}$ is mapped to $\rho^{-1}(1)$.

However, $\rho$ is the $(2g-1)$-cycle corresponding to the top row of
\[\tau = \begin{pmatrix}
\bullet & \bullet & 1 & \bullet & \bullet & \cdots & \bullet & \bullet \\
1 & 2 & 3 & 4 & 5 & \cdots & 2g-2 & 2g-1
\end{pmatrix}.\]
In particular, $\tau(\rho^{-1}(1)) = 2\in\Delta_{2}$ which forces $\tau(\Delta_{2}) = \Delta_{2}$. So $\tau$ does not act as a $k$-cycle on the blocks and so cannot be a $(2g-1)$-cycle. This is a contradiction, and so no block system can exist. Therefore, $O$ is indeed a primitive origami.
\end{proof}

\begin{remark}\label{rem:prim-odd-odd}
Continuing the analysis of the proof of Proposition~\ref{prop:prim-odd-odd}, we see that $\tau$ is forced to act on the block system as the permutation $(\Delta_{1},\Delta_{3})(\Delta_{4},\Delta_{k})\cdots(\Delta_{i},\Delta_{k+4-i})\cdots(\Delta_{\frac{k+3}{2}},\Delta_{\frac{k+5}{2}})$
and so the action of $\sigma_{g}$ and $\tau$ on the block system generates the dihedral group $D_{k}$ of degree $k$. Therefore, the monodromy group of an origami $O$ obtained from the modified construction of Subsection~\ref{subsec:AMN-odd} considered in Remark~\ref{rem:gen-const} must be a subgroup of $\Sym_{\frac{2g-1}{k}} \wr D_{k}$.

It can be shown that the hyperelliptic origami considered in Remark~\ref{rem:spin} has monodromy group isomorphic to the dihedral group $D_{2g-1}$ of degree $2g-1$ and so is only primitive when $2g-1$ is prime. In the other cases, computer investigations suggest that either \[ [(\Sym_{\frac{2g-1}{k}} \wr D_{k}):\Mon(O)] = 2\]
and we have 
\[ \Mon(O)\cong(\Sym_{\frac{2g-1}{k}} \wr D_{k})\cap\Alt_{2g-1},\]
where we view $\Sym_{\frac{2g-1}{k}} \wr D_{k}$ as a subgroup of $\Sym_{2g-1}$; or 
\[ [(\Sym_{\frac{2g-1}{k}} \wr D_{k}):\Mon(O)] = 2^{k},\] 
and 
\[ \Mon(O)\cong(\Alt_{\frac{2g-1}{k}})^{k}\rtimes\Sym_{k}\leq(\Sym_{\frac{2g-1}{k}} \wr D_{k})\cap\Alt_{2g-1},\]
where the semidirect product action is not necessarily that coming from the standard wreath product.
\end{remark}

\subsubsection{Even genus AMN origamis}

Here we will argue that all of the even genus AMN origamis are also primitive. We will use the notation of Subsection~\ref{subsec:AMN-even}.

Let $g$ be even, and let $(\sigma_{g-1},\eta)$ be an AMN origami of genus $g-1$. Set $\rho = \eta^{-1}\sigma_{g-1}\eta$. Recall then that $\rho$ is the cycle given by the top row of $\eta$ when written as a 2-row permutation. To construct an even genus AMN origami we choose 
\[ k\in\{3,5,\ldots,2g-3\}\setminus\{\eta^{-1}(1)\},\]
and set
\[\tau = \begin{pmatrix}
    \eta^{-1}(1) & \eta^{-1}(2) & 1 & \eta^{-1}(4) & \cdots & 2g-2 & \cdots & \eta^{-1}(2g-3) & 2g-1 & k \\
    1 & 2 & 3 & 4 & \cdots & \eta(k) & \cdots & 2g-3 & 2g-2 & 2g-1
\end{pmatrix}\]
and define the origami to be $(\sigma_{g},\tau)$.

\begin{proposition}
    All even genus AMN origamis $(\sigma_{g},\tau)$ are primitive.
\end{proposition}

\begin{proof}
    It can be checked that
    \begin{align*}\nu &:= \tau^{-1}\sigma_{g}^{-1}\tau\sigma_{g} \\
    &= (1,3,5,\ldots,k-2,2g-2,\rho^{-3}(1),\rho^{-5}(1),\ldots,\underbrace{\rho^{-2m-1}(1)}_{=k-1},2g-1,\\
    & \hspace{2cm}\underbrace{\rho^{-1}(1)}_{=\tau^{-1}(2)},k,k+2,k+4,\ldots,2g-3,\rho^{-2m-3}(1),\rho^{-2m-5}(1),\ldots,\rho^{2}(1)).
    \end{align*}
    Now suppose that the monodromy group admits a non-trivial block system $\{\Delta_{i}\}_{i = 1}^{t}$. The permutation $\sigma_{g}$ must act on these blocks as a $t$-cycle and so we can relabel them as $\Delta_{i} = \{j| j\equiv i\mod 2g-1\}$. By the discussion at the beginning of this section, we have $2g-1\geq 3t$ and $t\geq 5$. Therefore, there are at least 2 odd numbers in $\Delta_{t}$. Indeed, we will have $2g-1,2g-1-2t\in\Delta_{t}$.
    
    Let $l = 2g-1-2t$. We will consider where $l$ lies inside the permutation $\nu$.

    \textbf{Case 1: $\boldsymbol{l\in\{3,5,7,\ldots,k-4\}}$.} Since $l\in\Delta_{t}$, we have $\nu(l) = l+2\in \Delta_{2}$. Hence, on the block system, $\nu(\Delta_{t}) = \Delta_{2}$. We also see that $\nu(2g-1) = \tau^{-1}(2)$. Since $2g-1\in\Delta_{t}$, it must be the case that $\tau^{-1}(2)\in\Delta_{2}$. Now $\tau(\tau^{-1}(2)) = 2\in\Delta_{2}$ so we must have $\tau(\Delta_{2}) = \Delta_{2}$. However, this contradicts the fact that $\tau$ is a $(2g-1)$-cycle and so must act as a $t$-cycle on the blocks.

    \textbf{Case 2: $\boldsymbol{l\in\{k,k+2,\ldots,2g-3\}}$.} Since $2g-3\in\Delta_{t-2}\neq\Delta_{t}$, we must have $l\in\{k,\ldots,2g-5\}$. So we again have $\nu(l) = l+2$ giving $\nu(\Delta_{t}) = \Delta_{2}$, and the same contradiction as above follows.

    \textbf{Case 3: $\boldsymbol{l = k-2}$.} We have $k-2\in\Delta_{t}$ so that $k-1\in\Delta_{1}$ and $k-4\in\Delta_{t-2}$. Now $\nu^{-1}(k-2) = k-4$ implies that $\nu^{-1}(\Delta_{t}) = \Delta_{t-2}$ and $\nu^{-1}(2g-1) = k-1$ implies that $\nu^{-1}(\Delta_{t}) = \Delta_{1}$. However, this leads to a contradiction since $t\geq 5$ implies that $\Delta_{t-2}\neq\Delta_{1}$.

    In all cases we arrive at a contradiction and hence no non-trivial block system can exist. Thus, the origami is primitive.
\end{proof}

\subsubsection{Odd genus even spin}

\begin{proposition}\label{prop:prim-odd-even}
All of the origamis constructed using the method of Subsection~\ref{subsec:gen-odd-even} are primitive.
\end{proposition}

\begin{proof}
Let $O = (\sigma_{g},\tau)$ be an origami constructed using the method of Subsection~\ref{subsec:gen-odd-even}. From the discussion above, we observe that such an origami must cover a minimal $[1,1]$-origami in the minimal stratum of genus at least 3 (that is, with at least 5 squares). In terms of the monodromy group, any block system of size $k$ for $\Mon(O)$ must have $k = 1$ or $k = 2g-1$ corresponding to the trivial block systems, or $5\leq k < 2g-1$.

In fact, $k = 5$ is not possible. Indeed, since $\sigma_{g}$ must act by a 5-cycle on these blocks we can enumerate them as $\{\Delta_{i} | 1\leq i\leq 5\}$ with $\Delta_{i} = \{j\,|\,j\equiv i\mod 5\}$. Recall that $\tau$ has the structure
\[\tau = \begin{pmatrix}
9 & 6 & 1 & 8 & 7 & 5 & 4 & \bullet & \bullet & \cdots & \bullet & \bullet \\
1 & 2 & 3 & 4 & 5 & 6 & 7 & 8 & 9 & \cdots & 2g-2 & 2g-1
\end{pmatrix}.\]
Now, $\tau(1) = 3$ enforces $\tau(\Delta_{1}) = \Delta_{3}$ but this contradicts $\tau(6) = 2$ since $6\in\Delta_{1}$ and $2\in\Delta_{2}\neq\Delta_{3}$. In fact, $k = 3$ can be ruled out by an argument similar to this without having to use the $[1,1]$-origami covering argument discussed above. See Remark~\ref{rem:prim} for why this is note-worthy.

A non-trivial block system must therefore have $7\leq k < 2g-1$. So assume that we have a such block system of size $k$ and again enumerate it as $\{\Delta_{i} | 1\leq i\leq k\}$ with $\Delta_{i} = \{j\,|\,j\equiv i\mod k\}$. Now, since $2g-1$ is odd, $k|2g-1$, and $7\leq k < 2g-1$, we have $2g-1\geq 3k>k+6>9$. Now recall that the top row of $\tau$ which we previously called $\rho = \tau^{-1}\sigma_{g}\tau$ has the property that $\rho(k+6) = k+5$. In particular, this forces $\rho(\Delta_{6}) = \Delta_{5}$. However, we have $\rho(6) = 1$ which also forces $\rho(\Delta_{6}) = \Delta_{1}\neq \Delta_{5}$, a contradiction.

Therefore, the only possible block systems are the trivial block systems. That is, the origami $O$ is primitive.
\end{proof}

\subsubsection{Even genus odd spin}

\begin{proposition}\label{prop:prim-even-odd}
All of the origamis constructed using the method of Subsection~\ref{subsec:gen-even-odd} are primitive.
\end{proposition}

\begin{proof}
Let $O = (\sigma_{g},\tau)$ be an origami constructed using the method of Subsection~\ref{subsec:gen-even-odd}. As above, any block system of size $k$ for $\Mon(O)$ must have $k = 1$ or $k = 2g-1$ corresponding to the trivial block systems, or $5\leq k < 2g-1$.

Again, $k = 5$ is not possible. Indeed, enumerating the blocks of such a block system as $\{\Delta_{i} | 1\leq i\leq 5\}$ with $\Delta_{i} = \{j\,|\,j\equiv i\mod 5\}$. Since $\tau$ has the structure
\[\tau = \begin{pmatrix}
4 & 6 & 1 & 5 & 7 & \bullet & \bullet & \cdots & \bullet & \bullet \\
1 & 2 & 3 & 4 & 5 & 6 & 7 & \cdots & 2g-2 & 2g-1
\end{pmatrix},\]
$\tau(1) = 3$ enforces $\tau(\Delta_{1}) = \Delta_{3}$ but this contradicts $\tau(6) = 2$ since $6\in\Delta_{1}$ and $2\in\Delta_{2}\neq\Delta_{3}$. In fact, $k = 3$ can be ruled by a similar analysis without relying on the $[1,1]$-origami covering argument.

A non-trivial block system must therefore have $7\leq k < 2g-1$. So assume that we have a such block system of size $k$ and again enumerate it as $\{\Delta_{i} | 1\leq i\leq k\}$ with $\Delta_{i} = \{j\,|\,j\equiv i\mod k\}$. Here, we have $2g-1\geq 3k>k+6>7$. Again, $\rho(k+6) = k+5$ forces $\rho(\Delta_{6}) = \Delta_{5}$. However, we have $\rho(6) = 1$ which also forces $\rho(\Delta_{6}) = \Delta_{1}\neq \Delta_{5}$, a contradiction.

Therefore, the only possible block systems are the trivial block systems. That is, the origami $O$ is primitive.
\end{proof}

\subsubsection{Even genus even spin}

\begin{proposition}\label{prop:prim-even-even}
All of the origamis constructed using the method of Subsection~\ref{subsec:gen-even-even} are primitive.
\end{proposition}

\begin{proof}
Let $O = (\sigma_{g},\tau)$ be an origami constructed using the method of Subsection~\ref{subsec:gen-even-even}. As above, any block system of size $k$ for $\Mon(O)$ must have $k = 1$ or $k = 2g-1$ corresponding to the trivial block systems, or $5\leq k < 2g-1$.

Again, $k = 5$ is not possible. Indeed, enumerating the blocks of such a block system as $\{\Delta_{i} | 1\leq i\leq 5\}$ with $\Delta_{i} = \{j\,|\,j\equiv i\mod 5\}$. Since $\tau$ has the structure
\[\tau = \begin{pmatrix}
7 & 5 & 1 & 6 & 4 & \bullet & \bullet & \cdots & \bullet & \bullet \\
1 & 2 & 3 & 4 & 5 & 6 & 7 & \cdots & 2g-2 & 2g-1
\end{pmatrix},\]
$\tau(1) = 3$ enforces $\tau(\Delta_{1}) = \Delta_{3}$ but this contradicts $\tau(6) = 4$ since $6\in\Delta_{1}$ and $4\in\Delta_{4}\neq\Delta_{3}$. Again, as in the last two propositions, $k=3$ can in fact be ruled out by a similar analysis without using the $[1,1]$-origami covering argument.

A non-trivial block system must therefore have $7\leq k < 2g-1$. So assume that we have a such block system of size $k$ and again enumerate it as $\{\Delta_{i} | 1\leq i\leq k\}$ with $\Delta_{i} = \{j\,|\,j\equiv i\mod k\}$. As above, we have $2g-1\geq 3k>k+6>7$. Again, $\rho = \tau^{-1}\sigma_{g}\tau$ has the property that $\rho(k+6) = k+5$. In particular, this forces $\rho(\Delta_{6}) = \Delta_{5}$. However, we have $\rho(6) = 4$ which also forces $\rho(\Delta_{6}) = \Delta_{4}\neq \Delta_{5}$, a contradiction.

Therefore, the only possible block systems are the trivial block systems. That is, the origami $O$ is primitive.
\end{proof}

\begin{remark}\label{rem:prim}
We observe that the proofs of Propositions~\ref{prop:prim-odd-even},~\ref{prop:prim-even-odd} and~\ref{prop:prim-even-even} only make use of the parts of the structure of $\tau$ that remain in the modified constructions of Remarks~\ref{rem:gen-const} and~\ref{rem:spin}. That is, the proofs do not require that $\tau$ is a $(2g-1)$-cycle. Therefore, all of the single cylinder origamis produced by the modifications of the constructions of Subsections~\ref{subsec:gen-odd-even},~\ref{subsec:gen-even-odd} and~\ref{subsec:gen-even-even} discussed in the above remarks are all primitive.
\end{remark}

\subsection{Constructing 3-cycles}\label{subsec:3-cycles}

Here we will construct a 3-cycle in the monodromy groups of the odd genus AMN origamis constructed in Subsection~\ref{subsec:AMN-odd}. We start by illustrating two motivating examples for the general construction.

\begin{example}\label{ex:3-cycle} Consider $g = 7$ and
\[\tau = \begin{pmatrix}
1 & 3 & 2 & 13 & 12 & 11 & 10 & 7 & 6 & 5 & 4 & 9 & 8 \\
1 & 2 & 3 & 4 & 5 & 6 & 7 & 8 & 9 & 10 & 11 & 12 & 13
\end{pmatrix}.\]
Then, $\rho = (1,3,2,13,12,11,10,7,6,5,4,9,8)$ and
\begin{equation*}
    \rho\sigma_g = (1,13,3,9,7),\ \sigma_g\rho = (1,4,10,8,2).
\end{equation*}
Conjugating $\sigma_g\rho$ by $\rho\sigma_g$, we obtain
\begin{equation*}
    \beta = (\rho\sigma_g)\sigma_g\rho(\rho\sigma_g)^{-1} = (13,4,10,8,2).
\end{equation*}
Note that $\beta$ is the same permutation as $\sigma_g\rho$ except $1$ is replaced by $13$. We compute that
\begin{equation*}
    (\sigma_g\rho)^{-1}\beta = (1,2,13),
\end{equation*}
as desired.

\vskip 5pt

Next, let 
\[\tau = \begin{pmatrix}
1 & 13 & 12 & 9 & 8 & 11 & 10 & 5 & 4 & 3 & 2 & 7 & 6 \\
1 & 2 & 3 & 4 & 5 & 6 & 7 & 8 & 9 & 10 & 11 & 12 & 13
\end{pmatrix}.\]
Similarly, we have
\begin{equation*}
    \rho\sigma_g = (1,7,11,9,5),\ \sigma_g\rho = (2,8,12,10,6).
\end{equation*}
Note in particular that $1$ is only in the cycle representation for $\rho\sigma_g$ but not $\sigma_g\rho$. Hence, the previous strategy of relabeling $1$ with another odd number fails. Instead, we compute that
\begin{equation*}
    \sigma_g^{-1}\rho = (1,12,8,10,4,2,6,13,11,9,7,5,3).
\end{equation*}
Conjugating $\rho\sigma_g$ by $(\sigma_g^{-1}\rho)^{-2}$, we obtain
\begin{equation*}
    \alpha = (\sigma_g^{-1}\rho)^{-2}\rho\sigma_g(\sigma_g^{-1}\rho)^2 = (5,11,6,13,9).
\end{equation*}
Now, $6$ is in both $\alpha$ and $\sigma_g\rho$ where the rest entries of $\pi$ are odd and those of $\sigma_g\rho$ are even. Hence, we may relabel $6$ in $\sigma_g\rho$ with an odd number while keeping other indices fixed by conjugating $\sigma_g\rho$ with $\alpha$:
\begin{equation*}
    \beta = \alpha\sigma_g\rho\alpha^{-1} = (2,8,12,10,13).
\end{equation*}
Observe that $\beta$ is the same permutation as $\sigma_g\rho$ except $6$ is replaced by $13$. Finally, we compute that
\begin{equation*}
    (\sigma_g\rho)^{-1}\beta = (6,10,13),
\end{equation*}
as desired.
\end{example}

\begin{proposition} 
\label{prop:3-cycle}
The monodromy groups of the odd genus AMN origamis constructed in Subsection~\ref{subsec:AMN-odd} contain a $3$-cycle.
\end{proposition}

\begin{proof} Following the example above, we break the proof into two cases, i.e., when $\rho(1) \not = 2g-1$ and when $\rho(1) = 2g-1$. The reader is encouraged to refer back to the example throughout the argument.

Assume first $\rho(1) \not = 2g-1$. We start by giving an explicit description for $\rho\sigma_g$ and $\sigma_g\rho$. For an even index $i$, recall that $\rho(i) = i-1$ and $\sigma_g(i) = i+1$ is an odd index not equal to $1$. Thus, $\rho\sigma_g(i) = i$ and $\sigma_g\rho$ fixes all even indices in $\{1,2,\ldots,2g-1\}$. Now, consider the index $1$, for which $\sigma_g(1) = 2$. In construction of $\tau$, $(3,2)$ is the first pair to be placed in the algorithm so that it cannot be at the end of the cycle representation of $\rho$; thus, $2$ is followed by an odd index $i$ not equal to $1$. Similarly, since $\sigma_g(2g-1) = 1$ and by assumption $\rho(1)$ is an odd number not equal to $2g-1$, we know that $2g-1$ is in a non-trivial cycle of $\rho\sigma_g$. In summary, we conclude that every non-trivial cycle of $\rho\sigma_g$ contains only odd indices, some of which contains $1$ and $2g-1$.

Next, observe that
\begin{equation*}
    \sigma_g\rho = \sigma_g(\rho\sigma_g)\sigma^{-1}.
\end{equation*}
Hence, $\sigma_g\rho$ is obtained by adding $1$ (mod $2g-1$) to each entry of $\rho\sigma_g$. It follows from the characterization of $\sigma_g\rho$ above that every non-trivial cycle of $\rho\sigma_g$ contains only $1$ or even indices.

Now, since every even index is fixed by $\rho\sigma_g$ and every odd index is fixed by $\sigma_g\rho$ except $1$, conjugating $\sigma_g\rho$ by $\rho\sigma_g$ yields a permutation $\beta$ that is the same as $\sigma_g\rho$ except $1$ is replaced by an odd index $\rho\sigma_g(1) \not = 1$. We claim that $(\sigma_g\rho)^{-1}\beta$ is a $3$-cycle. 

To verify, for $i \not = 1, \rho\sigma_g(1), (\sigma_g\rho)^{-1}(1)$, the characterizations of $\beta$ and $\sigma_g\rho$ imply that $\beta(i) = \sigma_g\rho(i)$ and thus,
\begin{equation*}
    (\sigma_g\rho)^{-1}\beta(i) = i.
\end{equation*}
Next, considering the index $1$, we have $\beta(1) = 1$ so that $(\sigma_g\rho)^{-1}\beta(1) = (\sigma_g\rho)^{-1}(1)$. Then, since $\beta(\sigma_g\rho)^{-1}(1) = \rho\sigma_g(1)$ and $\sigma_g\rho$ fixes $\rho\sigma_g(1)$, we have $((\sigma_g\rho)^{-1}\beta)(\sigma_g\rho)^{-1}(1) = \rho\sigma_g(1)$. Finally, since $(\sigma_g\rho)^{-1}\beta$ fixes all other indices, it has to send $\rho\sigma_g(1)$ back to $1$. Therefore, we obtain a $3$-cycle
\begin{equation*}
    (\sigma_g\rho)^{-1}\beta = (1,(\sigma_g\rho)^{-1}(1),\rho\sigma_g(1)),
\end{equation*}
in the monodromy group as desired.

Next, suppose $\rho(1) = 2g-1$. By the same argument above, we know that every non-trivial cycle of $\rho\sigma_g$ contains only odd indices and necessarily $1$. In particular, note that $(\rho\sigma_g)^{-1}(1) = \rho^{-1}(1)-1$. Further, since $\sigma_g(2g-1) = 1$ and $\rho(1) = 2g-1$, $\rho\sigma_g$ fixes $2g-1$ and this is different from the first case. In summary, we conclude that every non-trivial cycle of $\rho\sigma_g$ contains only odd indices except $2g-1$ and $\rho\sigma_g(\rho^{-1}(1)-1) = 1$.

Since $\sigma_g\rho$ is obtained by adding $1$ (mod $2g-1$) to each entry of $\rho\sigma_g$, again, it follows from the characterization of $\rho\sigma_g$ that every non-trivial cycle of $\sigma_g\rho$ contains only even indices and $\rho\sigma_g(\rho^{-1}(1)) = 2$.

Consider the permutation $\sigma_g^{-1}\rho$. It has been shown in Section~\ref{subsec:AMN-odd} that
\begin{equation*}
    \sigma_g^{-1}\rho = ((1,\rho^{2}(1),\rho^{4}(1),\ldots,\rho^{-1}(1),2g-1,2g-3,\ldots,3).
\end{equation*}
Since non-trivial cycles of $\rho\sigma_g$ contain odd indices only, conjugating $\rho\sigma_g$ by a power of $(\sigma_g^{-1}\rho)^{-1}$ is equivalent to moving each index of $\rho\sigma_g$ along $\sigma_g^{-1}\rho$ from right to left. In particular, there exists a power of $(\sigma_g^{-1}\rho)^{-1}$ such that every non-trivial cycle of the resulting conjugation $\alpha$ contains only $\rho^{-1}(1)$ or odd indices not equal to $1$.

Now, we conjugate $\sigma_g\rho$ by $\alpha$ and obtain a permutation $\beta$ that is the same as $\sigma_g\rho$ except $\rho^{-1}(1)$ is replaced by an odd index $\alpha\rho^{-1}(1) \not = 1$. Thus, by the same argument as the first case, we obtain a $3$-cycle
\begin{equation*}
    (\sigma_g\rho)^{-1}\beta = (\rho^{-1}(1), (\sigma_g\rho)^{-1}\rho^{-1}(1), \alpha\rho^{-1}(1)),
\end{equation*}
in the monodromy group as desired.
\end{proof}

\begin{corollary}
    By Theorem~\ref{thm:Jordan} and Propositions~\ref{prop:prim-odd-odd} and~\ref{prop:3-cycle}, the odd genus AMN origamis have monodromy group equal to the alternating group.
\end{corollary}

This, combined with Theorem~\ref{thm:simple} proved in Section~\ref{sec:simple} and some calculations to rule out the sporadic exmaples, completes the proof of Theorem~\ref{thm:monodromy}.

\section{Orientation double covers}\label{sec:double-covers}

Here we will show that $(g-3)!!$ of the $(g-2)!$ odd genus AMN origamis are orientation double covers. This implies that there are at least two $\SL(2,\Z)$-orbits of odd genus AMN origamis. We will also prove that none of the other origamis considered in this paper are orientation double covers. This is the content of Theorem~\ref{thm:double-covers}.

A translation surface $(X,\omega)$ is said to be an \emph{orientation double cover} or a \emph{translation double cover} if $X$ is a (possibly ramified) double cover of a surface $Y$ equipped with a quadratic differential $q$ so that the pullback of $q$ to $X$ is equal to $\omega^{2}$. See~\cite[pp. 15-16]{inftrans} for more details. If $(X,\omega)$ is an orientation double cover then it is fixed by $-I\in\SL(2,\R)$ and double covers $X/\langle -I\rangle$.

\begin{figure}[b]
    \centering
    \begin{tikzpicture}[scale = 1.2, line width = 0.3mm]
    \draw (0,0)--(0,1);
    \draw (0,1)--(5,1);
    \draw (5,1)--(5,0);
    \draw (5,0)--(0,0);
    \foreach \i in {1,2,...,4}{
        \draw [dashed, color = darkgray] (\i,0)--(\i,1);
        \draw (\i,0.05)--(\i,-0.05);
        \draw (\i,1.05)--(\i,0.95);
    }
    \draw (0,0.5) node[left] {$0$};
    \draw (5,0.5) node[right] {$0$};
    \foreach \i in {1,...,5}{
        \draw (\i-0.5,1) node[above] {$\i$};
        \draw (\i-0.5,0.5) node[color = darkgray] {\footnotesize $\i$};
    }
    \draw (0.5,0) node[below] {$5$};
    \draw (1.5,0) node[below] {$4$};
    \draw (2.5,0) node[below] {$1$};
    \draw (3.5,0) node[below] {$3$};
    \draw (4.5,0) node[below] {$2$};
    %%%%%%%%%%%%%%%%%%%%%%%%%%%%%%%%%%%%%%%%%
    \draw (6,0)--(6,1);
    \draw (6,1)--(11,1);
    \draw (11,1)--(11,0);
    \draw (11,0)--(6,0);
    \foreach \i in {7,...,10}{
        \draw [dashed, color = darkgray] (\i,0)--(\i,1);
        \draw (\i,0.05)--(\i,-0.05);
        \draw (\i,1.05)--(\i,0.95);
    }
    \draw (6,0.5) node[left] {$0$};
    \draw (11,0.5) node[right] {$0$};
    \foreach \i in {5,...,1}{
        \draw (12-\i-0.5,0) node[below] {$\i$};
        \draw (12-\i-0.5,0.5) node[color = darkgray, rotate = 180] {\footnotesize $\i$};
    }
    \draw (6.5,1) node[above] {$2$};
    \draw (7.5,1) node[above] {$3$};
    \draw (8.5,1) node[above] {$1$};
    \draw (9.5,1) node[above] {$4$};
    \draw (10.5,1) node[above] {$5$};
    \end{tikzpicture}
    \caption{Rotating the origami on the left by $\pi$ gives the origami on the right. Relabelling the sides according to $(1)(2,4)(3,5)$ recovers the original origami.}
    \label{fig:orientation-example}
\end{figure}
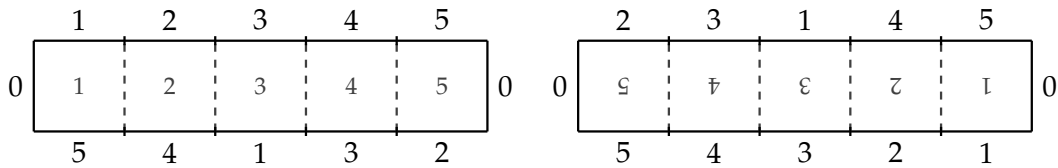

Before proving the proposition below, we will give an example that demonstrates the main considerations of the proof. 

\begin{example}\label{ex:double-cover:rotation}
Consider the genus 3 AMN origami on the left of Figure~\ref{fig:orientation-example}. The matrix $-I\in\SL(2,\R)$ acts on the polygons by rotating the picture upside down. The origami is then an orientation double cover if there is a choice of relabelling of the sides (and cut and paste of the squares) that gives the original origami back. For example, the origami on the right of Figure~\ref{fig:orientation-example} is the image under $-I$ and we see that the side relabelling given by the permutation $(1)(2,4)(3,5)$ returns (up to cut and paste of the squares) the original origami.

We remark that the action of $-I$ can always be seen as a central rotation around the centre of a square in the origami. Indeed, on the level of permutations, the origami $(\sigma_{g},\tau)$ is sent to $(\sigma_{g}^{-1},\tau^{-1})$ by $-I$. If this fixes the origami, then there exists some $\rho$ such that $\rho \sigma_{g}^{-1}\rho^{-1} = \sigma_{g}$ and $\rho \tau^{-1} \rho^{-1} = \tau$. That is, the action of $-I$ can be carried out by a relabelling of the squares of the origami. Since $(-I)^{2} = I$, $\rho^{2} = 1$. Since we have an odd number of squares, $\rho$ fixes some square $i$ and since $\rho$ conjugates $\sigma_{g} = (1,2,\ldots,2g-1)$ to $\sigma_{g}^{-1}$, we must have that $\rho = (i-1,i+1)(i-2,i+2)\cdots$. This is exactly the effect on the squares caused by rotation around the square $i$ and the correct relabelling of the sides corresponds exactly to the relabelling induced by the rotation.

Returning to our example, the origami in Figure~\ref{fig:orientation-example} corresponds to $((1,2,3,4,5),(1,3,4,2,5))$ and, in the language of the previous paragraph, $\rho = (1,3)(4,5)$ and so we see that the rotational symmetry can be realised by rotating around the centre of square 2. This can be seen in Figure~\ref{fig:canonical-rotation} and we recover the induced relabelling of the sides $(1)(2,4)(3,5)$. 
\end{example}

\begin{figure}[t]
    \centering
    \begin{tikzpicture}[scale = 1.2, line width = 0.3mm]
    \draw (0,0)--(0,1);
    \draw (0,1)--(5,1);
    \draw (5,1)--(5,0);
    \draw (5,0)--(0,0);
    \foreach \i in {1,2,...,4}{
        \draw [dashed, color = darkgray] (\i,0)--(\i,1);
        \draw (\i,0.05)--(\i,-0.05);
        \draw (\i,1.05)--(\i,0.95);
    }
    \draw (0,0.5) node[left] {$0$};
    \draw (5,0.5) node[right] {$0$};
    \draw (1-0.5,1) node[above] {$5$};
    \draw (1-0.5,0.5) node[color = darkgray] {\footnotesize $5$};
    \foreach \i in {1,2,3,4}{
        \draw (1+\i-0.5,1) node[above] {$\i$};
        \draw (1+\i-0.5,0.5) node[color = darkgray] {\footnotesize $\i$};
    }
    \draw (0.5,0) node[below] {$2$};
    \draw (1.5,0) node[below] {$5$};
    \draw (2.5,0) node[below] {$4$};
    \draw (3.5,0) node[below] {$1$};
    \draw (4.5,0) node[below] {$3$};
    \end{tikzpicture}
    \caption{The surface is symmetric by rotation around the centre of square 2 (after the side relabelling $(1)(2,4)(3,5)$ induced by the rotation).}
    \label{fig:canonical-rotation}
\end{figure}
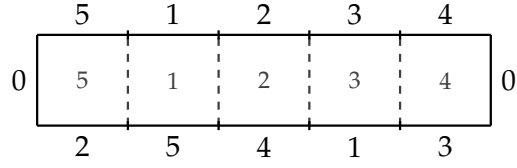

In the proof below, we will call this rotation in Example~\ref{ex:double-cover:rotation} the \emph{canonical rotation} and we will analyse the relabelling forced by such a canonical rotation in order to arrive at a contradiction.

We will also give the construction of $(g-3)!!$ many odd genus AMN origamis that are orientation double covers. Here, we give an example of this construction in genus 5 that is rotationally symmetric about the centre of square 2 (we will prove that this must be the centre of the rotation). 

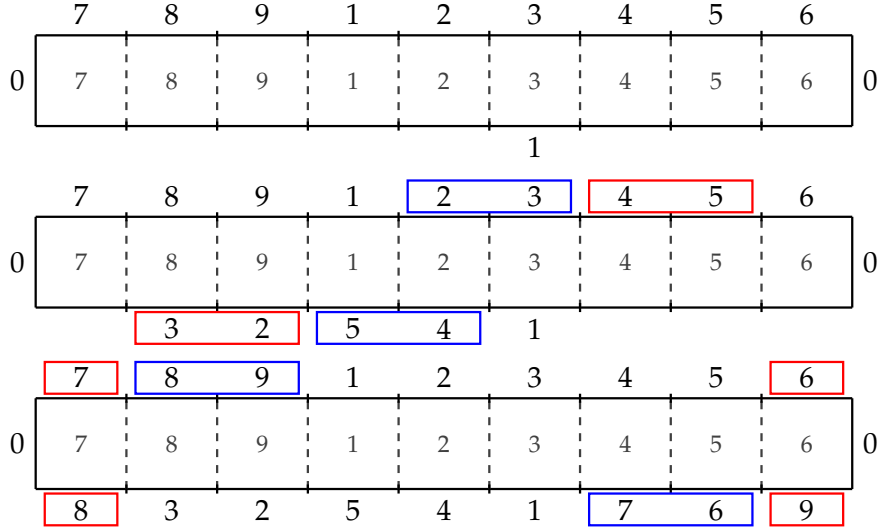
\begin{figure}[t]
    \centering
    \begin{tikzpicture}[scale = 1.2, line width = 0.3mm]
    \foreach \j in {0,2,4}{
    \draw (0,0-\j)--(0,1-\j);
    \draw (0,1-\j)--(9,1-\j);
    \draw (9,1-\j)--(9,0-\j);
    \draw (9,0-\j)--(0,0-\j);
    \foreach \i in {1,...,8}{
        \draw [dashed, color = darkgray] (\i,0-\j)--(\i,1-\j);
        \draw (\i,0.05-\j)--(\i,-0.05-\j);
        \draw (\i,1.05-\j)--(\i,0.95-\j);
    }
    \draw (0,0.5-\j) node[left] {$0$};
    \draw (9,0.5-\j) node[right] {$0$};
    \foreach \i in {7,...,9}{
        \draw (\i-6-0.5,1-\j) node[above] {$\i$};
        \draw (\i-6-0.5,0.5-\j) node[color = darkgray] {\footnotesize $\i$};
    }
    \foreach \i in {1,...,6}{
        \draw (3+\i-0.5,1-\j) node[above] {$\i$};
        \draw (3+\i-0.5,0.5-\j) node[color = darkgray] {\footnotesize $\i$};
    }
    \draw (5.5,0-\j) node[below] {$1$};
    }
    \draw (1.5,-2) node[below] {$3$};
    \draw (2.5,-2) node[below] {$2$};
    \draw (3.5,-2) node[below] {$5$};
    \draw (4.5,-2) node[below] {$4$};
    \draw[color = red] (1.1,-2.05) -- (2.9,-2.05) -- (2.9,-2.4) -- (1.1,-2.4) -- cycle;
    \draw[color = red] (6.1,-0.95) -- (7.9,-0.95) -- (7.9,-0.6) -- (6.1,-0.6) -- cycle;
    \draw[color = blue] (3.1,-2.05) -- (4.9,-2.05) -- (4.9,-2.4) -- (3.1,-2.4) -- cycle;
    \draw[color = blue] (4.1,-0.95) -- (5.9,-0.95) -- (5.9,-0.6) -- (4.1,-0.6) -- cycle;
    \draw (0.5,-4) node[below] {$8$};
    \draw (1.5,-4) node[below] {$3$};
    \draw (2.5,-4) node[below] {$2$};
    \draw (3.5,-4) node[below] {$5$};
    \draw (4.5,-4) node[below] {$4$};
    \draw (6.5,-4) node[below] {$7$};
    \draw (7.5,-4) node[below] {$6$};
    \draw (8.5,-4) node[below] {$9$};
    \draw[color = red] (0.1,-4.05) -- (0.9,-4.05) -- (0.9,-4.4) -- (0.1,-4.4) -- cycle;
    \draw[color = red] (0.1,-2.95) -- (0.9,-2.95) -- (0.9,-2.6) -- (0.1,-2.6) -- cycle;
    \draw[color = red] (8.1,-4.05) -- (8.9,-4.05) -- (8.9,-4.4) -- (8.1,-4.4) -- cycle;
    \draw[color = red] (8.1,-2.95) -- (8.9,-2.95) -- (8.9,-2.6) -- (8.1,-2.6) -- cycle;
    \draw[color = blue] (6.1,-4.05) -- (7.9,-4.05) -- (7.9,-4.4) -- (6.1,-4.4) -- cycle;
    \draw[color = blue] (1.1,-2.95) -- (2.9,-2.95) -- (2.9,-2.6) -- (1.1,-2.6) -- cycle;
    \end{tikzpicture}
    \caption{A construction of a genus 5 AMN origami that is an orientation double cover. It is symmetric by rotation around the centre of square 2.}
    \label{fig:genus-5-double}
\end{figure}

\begin{example}\label{ex:double-cover:pairs}
Consider the start of the construction with square 2 placed centrally, as in the top of Figure~\ref{fig:genus-5-double}. We first place the pair $(3,2)$ somewhere on the bottom. As usual, we cannot place it below $(1,2)$ as this will close the cycle. Suppose that we place it below $(8,9)$. We see that $(3,2)$ on the bottom is rotated to $(4,5)$ on the top. So, in order to preserve rotational symmetry, $(2,3)$ on the top must be rotated to $(5,4)$ on the bottom which forces us to place $(5,4)$ below $(1,2)$. Note that this process of preserving rotational symmetry prevents us from having placed $(3,2)$ under $(6,7)$. Indeed, in this case the rotational symmetry would have forced us to place $(7,6)$ under $(1,2)$ which would have closed the cycle. Next, since we placed $(3,2)$ below $(8,9)$, we place $(9,8)$. We cannot place it below $(4,5)$ as this would close the cycle, so we must place it below $(6,7)$. To preserve rotational symmetry, we place $(7,6)$ below $(5,4)$ and we are done.
\end{example}

\begin{theorem}
    For all odd $g\geq 3$, exactly $(g-3)!!$ of the $(g-2)!$ odd genus AMN origamis as constructed in Subsection~\ref{subsec:AMN-odd} are orientation double covers of quadratic differentials in the stratum $\mathcal{Q}_{\frac{g-1}{2}}(2g-3,-1^{3})$.

    None of the remaining origamis considered in this paper are orientation double covers.
\end{theorem}

\begin{proof}
For each $1\leq i\leq 2g-1$, we suppose that the canonical rotation is centred at square $i$. We will see if the induced relabelling of the sides gives rise to a contradiction.

\textbf{Case 1: $\boldsymbol{4\leq i\leq g}$.} If we have an odd genus AMN origami it can be checked that the side labelled 1 on the top gets rotated to an even number on the bottom, whereas the side labelled 1 on the bottom gets rotated to an odd number on the top. The images of 1 must therefore be different but this contradicts the existence of the relabelling induced by the canonical rotation. The same argument works for the generalised constructions of Section~\ref{sec:gen-const}. For the even AMN construction, rotation around square $g$ will send $1$ to $k$ (where $k$ is the choice of odd number used in the even genus construction), so the parity issue above cannot be used. Instead, it can be checked that $2g-1$ on the top gets sent to an odd number on the bottom, while $2g-1$ on the bottom gets sent to an even number on the top, so a similar contradiction is reached.

\textbf{Case 2: $\boldsymbol{i = 3}$.} For the AMN constructions in all genus, it can be checked that the parity of the images of 1 disagree as above. For the generalisations of Section~\ref{sec:gen-const}, it can be directly checked that the images of 1 are distinct as they lie in the fixed prefixes of those constructions.

\textbf{Case 3: $\boldsymbol{i = 2}$.} In this case, 1 on the top is sent to 1 on the bottom and so this provides no contradiction. In fact, no contradiction arises in the odd genus AMN construction. In the even genus AMN construction, it can be checked that $2g-2$ has images with differing parities and we get the desired contradiction. For the construction of Subsection~\ref{subsec:gen-odd-even}, a parity contradiction arises from the images of $2g-1$. For the construction of Subsection~\ref{subsec:gen-even-odd}, the images of 4 can be used. Finally, for the construction of Subsection~\ref{subsec:gen-even-even}, the images of $5$ can be used.

\textbf{Case 4: $\boldsymbol{i = 1}$ or $\boldsymbol{g+3\leq i \leq 2g-1}$.} The parity of the images of 1 can be used for the odd genus AMN construction, and the constructions of Subsections~\ref{subsec:gen-odd-even} and~\ref{subsec:gen-even-even}. For the even genus AMN construction, the parity of the images of 1 can be used apart from in the situation where 1 on the top is sent to $2g-2$ on the bottom (where 1 on the bottom is also sent to an even number on the top). However, in this case, the parity of the images of $k$ will differ. For the construction of Subsection~\ref{subsec:gen-even-odd}, if $i = 1$ then 1 on the top gets sent to 4 on the bottom but 1 on the bottom gets sent to $2g-2\neq 4$ on the top. For the remaining choices of $i$, the parity of the images of 1 can be used.

\textbf{Case 5: $\boldsymbol{i = g + 2}$.} In all cases, 1 on the bottom is sent to 2 on the top and 1 on the top gets sent to the side to the right of 1 on the bottom which is never equal to 2.

\textbf{Case 6: $\boldsymbol{i = g+1}$.} In all cases, 1 on the bottom is sent to $2g-1$ on the top and 1 on the top is sent to the side to the left of 1 on the bottom which is never $2g-1$.

So we see that a contradiction arises in all cases apart from $i = 2$ in the odd genus AMN construction. We will now argue that there are $(g-3)!!$ many of these origamis that are orientation double covers, using the same logic as in Example~\ref{ex:double-cover:pairs}.

Start building the odd genus AMN origami with square 2 in the centre. Add the pair $(3,2)$ as normal in the construction. We cannot place it under $(1,2)$ as this closes the cycle, and we cannot place it under $(g+1,g+2)$ as the forced rotational symmetry would place $(g+2,g+1)$ under $(1,2)$ and the cycle will be closed again. So we have $(g-3)$ choices of a pair $(2j,2j+1)$ under which to place $(3,2)$, and we then place the required pair to preserve rotational symmetry (using up another choice of pair). Next, we place the pair $(2j+1,2j)$. We have one choice forbidden as it would close a cycle, so we have $(g-5)$ choices. Again, we place the pair forced by rotational symmetry. This continues until we close the vertical cycle after placing the final pair. This process produces $(g-3)!!$ many origamis, as claimed.

Finally, we see that in the above construction there are 4 points fixed by the rotation around square 2. These points are the point at the centre of square 2, the point at the center of the side labelled 1, the point at the centre of the side labelled 0 (the vertical side), and the point corresponding to the vertices of the squares (the zero). It can then be checked (e.g., using Riemann-Hurwitz) that the quotient surface has genus $\frac{g-1}{2}$ and carries a quadratic differential with a zero of order $2g-3$ (the image of the zero on the origami) and three poles (the images of the other fixed points). Hence, the resulting surface lies in the stratum $\mathcal{Q}_{\frac{g-1}{2}}(2g-3,-1^{3})$.
\end{proof}

Since the property of being an orientation double cover is preserved by the action of $\SL(2,\Z)$, we obtain the following.

\begin{corollary}
    For $g\geq 5$, the odd genus AMN origamis lie in at least two $\SL(2,\Z)$-orbits. For $g = 3$, there is only one AMN origami and thus only one orbit.
\end{corollary}

%%%%%%%%%%%%%%%%%%%%%%%%%%%%%%%%%%%%%%%%%%%%%%%%%%%%%%

\section{$\SL(2,\Z)$-orbit conjectures}\label{sec:SL2Z}

The work of Section~\ref{sec:double-covers} shows that the odd genus AMN origamis lie in at least two $\SL(2,\Z)$-orbits. It also shows that the remaining origamis constructed in this paper have no $\SL(2,\Z)$-orbit constraints arising from being a double cover. Computer investigations carried out with the \texttt{surface\_dynamics}~\cite{surf-dyn} package of SageMath~\cite{Sage} suggest that there are no further restrictions on the $\SL(2,\Z)$-orbits. That is, the situation is described by Table~\ref{tab:SL2Z-orbit}.

\begin{table}[H]
    \centering
    \begin{tabular}{|c|c|c|}
    \hline
        Genus & no. of $\SL(2,\Z)$-orbits of  & no. of $\SL(2,\Z)$-orbits of \\
        & origamis constructed in $\odd(2g-2)$ & origamis constructed in $\even(2g-2)$ \\
    \hline
        3 & 1 & n/a \\
        4 & 1 & 1 \\
        5 & 2 & 1 \\
        6 & 1 & 1 \\
        7 & 2 & 1 \\
        8 & 1 & 1 \\
    \hline
    \end{tabular}
    \caption{The number of $\SL(2,\Z)$ orbits containing the AMN origamis and their generalisations constructed in this paper}
    \label{tab:SL2Z-orbit}
\end{table}

It is interesting that the even genus AMN origamis have the same $\SL(2,\Z)$-orbits as the generalisations we have introduced in this paper. One might hope to be able to show that the combinatorial moves of Proposition~\ref{prop:same-spin} are realisable inside $\SL(2,\Z)$. However, we know that this cannot be true in general due to the orientation double covers that exist in the case of the odd genus AMN origamis. It therefore seems difficult to prove that the pattern in Table~\ref{tab:SL2Z-orbit} continues in higher genus. However, we conjecture that this is indeed the case (see Conjecture~\ref{con:SL2Z}).

%%%%%%%%%%%%%%%%%%%%%%%%%%%%%%%%%%%%%%%%%%%%%%%%%%%%%%

\section{Simple monodromy groups}\label{sec:simple}

Below, $C_p$ denotes the cyclic group of order $p$ and $\text{AGL}(1,p)$ denotes the affine general linear group of transformations of $\mathbb{F}_{p}$ of the form $x\mapsto \alpha x + \beta$, $\alpha\in\mathbb{F}_{p}^{*},\beta\in\mathbb{F}_{p}$. The groups $\PSL(d,q)$ and $\PGL(d,q)$ are the projective special linear group and projective general linear group, respectively, of the projective space $\PG(d-1,q)$ of dimension $d-1$ over the finite field $\mathbb{F}_{q}$, where $q = p^{k}$ is some prime power. The group $\PGamL(d,q)$ is the projective semilinear group defined by
\[\PGamL(d,q) := \PGL(d,q) \rtimes \Gal(\mathbb{F}_{q}/\mathbb{F}_{p}).\]
The groups $M_{11}$ and $M_{23}$ are the Mathieu groups of degree 11 and 23, respectively.

We will prove Theorem~\ref{thm:simple} by establishing the following group theoretic result.

\begin{theorem}\label{thm:grp-simple}
    Let $G$ be a primitive permutation group of degree $n$ that is generated by two $n$-cycles whose commutator is also an $n$-cycle. Then $G$ is $\PGamL(2,8)$, the Mathieu group $M_{11}$ with $n = 11$, the Mathieu group $M_{23}$ with $n = 23$, $PSL(d,q)$ with $n = \frac{q^{d}-1}{q-1}$ and $\gcd(d,q-1) = 1$, or $\Alt_{n}$.

    In particular, with the exception of $\PGamL(2,8)$, $G$ is simple.
\end{theorem}

The computer investigations discussed below were performed using GAP~\cite{Gap} and the \texttt{surface\_dynamics}~\cite{surf-dyn} package of SageMath~\cite{Sage}.

\subsection{Primitive permutation groups containing a cycle}

We make use of the following theorem of Jones.

\begin{theorem}[{\cite[Theorem 3]{Jones1}}]
    A primitive permutation group $G$ of finite degree $n$ has a cyclic regular subgroup if and only if one of the following holds:
    \begin{enumerate}
        \item $C_{p} \leq G \leq \text{AGL}(1,p)$ where $n = p$ is prime;
        \item $G = \Sym_{n}$ for some $n\geq 2$, or $G = \Alt_{n}$ for some odd $n\geq 3$;
        \item $PGL(d,q) \leq G \leq \PGamL(d,q)$ where $n = \frac{q^{d}-1}{q-1}$ for some $d\geq 2$;
        \item $G = \PSL(2,11), M_{11},$ or $M_{23}$ where $n = 11, 11$ or $23$, respectively.
    \end{enumerate}
\end{theorem}

It can be checked that the $p$-cycles in $\text{AGL}(1,p)$ are contained in $C_{p}$. In $\PGL(d,q)$, there is a natural construction of an $n$-cycle called a Singer cycle (see~\cite{Singer} and~\cite{Hir}). Jones proved the following result that establishes that, with the exception of $\PGamL(2,8)$, the only $n$-cycles in $\PGamL(d,q)$ are the Singer cycles contained in $\PGL(d,q)$.

\begin{theorem}[{\cite[Theorem 1]{Jones1}}]
    Let $A$ be an abelian subgroup of order $n = \frac{q^{d}-1}{q-1}$ in $\PGamL(d,q)$, where $d\geq 2$ and $q$ is any prime power.
    \begin{enumerate}
        \item If $d\geq 3$, or if $d = 2$ and $q \neq 3$ or 8, then $A$ is a Singer subgroup.
        \item In $\PGamL(2,3)$, with $n = 4$, $A$ is a Singer subgroup, or a transitive normal Klein four-group contained in $\PSL(2,3)$, or one of three conjugate intransitive Klein four-groups.
        \item In $\PGamL(2,8)$, with $n = 9$, $A$ is a Singer subgroup, or one of a conjugacy class of cyclic regular subgroups not contained in $\PGL(2,8)$, or one of a conjugacy class of intransitive elementary abelian subgroups, also not contained in $\PGL(2,8)$.
    \end{enumerate}
\end{theorem}

So, with the aid of some computer checks for the sporadic examples, we obtain the following result.

\begin{theorem}\label{thm:n-cycle-gen}
    Let $G$ be a primitive permutation group of degree $n$ that is generated by two $n$-cycles. Then one of the following holds:
    \begin{enumerate}
        \item $G = C_{p}$ where $n = p$ is prime;
        \item $G = \Sym_{n}$ for some even $n\geq 2$, or $G = \Alt_{n}$ for some odd $n\geq 3$;
        \item $G = \PGL(d,q)$ where $n = \frac{q^{d}-1}{q-1}$ for some $d\geq 2$;
        \item $G = \PGamL(2,8)$ where $n = 9$;
        \item $G = \PSL(2,11), M_{11},$ or $M_{23}$ where $n = 11, 11$ or $23$, respectively.
    \end{enumerate}
\end{theorem}

This gives us the following result for monodromy groups of primitive $[1,1]$-origamis.

\begin{corollary}
    Let $O$ be a primitive $[1,1]$-origami of genus $g\geq 1$. Then one of the following holds:
    \begin{enumerate}
        \item $O$ has genus 1 and the monodromy group of $O$ is $C_{p}$;
        \item $O$ has genus $g\geq 2$ and has monodromy group given by one of the groups in items (2)-(5) of Theorem~\ref{thm:n-cycle-gen}.
    \end{enumerate}
\end{corollary}

\subsection{Adding the commutator condition}

Now we add the condition that the commutator of the generating $n$-cycles is also an $n$-cycle.

We see that $C_{p}$ and $\Sym_{n}$ are now ruled out. It can also be checked that $\PSL(2,11)$ can not be generated by two $n$-cycles whose commutator is also an $n$-cycle, but $\PGamL(2,8)$, $M_{11}$ and $M_{23}$ can be. Finally, we observe that if $h,v\in G = \PGL(d,q)$ are two generating $n$-cycles, then $[h,v]\in G' = \PSL(d,q)$. However, an element in $\PSL(d,q)$ has maximum order $\frac{n}{\gcd(d,q-1)}$. So, if $[h,v]$ is also an $n$-cycle then we require $\gcd(d,q-1) = 1$. Hence, $\PGL(d,q) \cong \PSL(d,q)$ in such a case. This completes the proof of Theorem~\ref{thm:grp-simple}, and the translation into the language of monodromy groups gives us Theorem~\ref{thm:simple}.

\subsection{Open questions}\label{subsec:finite-simple}

A modification of the proof of \cite[Proposition 5.14]{BF} in the recent work of Burness-Fusari demonstrates that $\PGL(d,q)$ can always be generated by two Singer cycles. Determining the possibilities for the cycle structure of the commutator of two generating Singer cycles is still wide open.

\begin{question}
    Let $\PGL(d,q)$ be generated by two Singer cycles $h$ and $v$. What can be said about the cycle structure of $[h,v]$? In particular, when $\gcd(d,q-1) = 1$, when can $[h,v]$ also be realised as a Singer cycle in $\PSL(d,q)\cong\PGL(d,q)$?
\end{question}

Note that $\PSL(2,2)\cong\Sym_{3}$ and $\PSL(3,2)\cong\PSL(2,7)\not\cong\PGL(2,7)$ cannot be generated by two Singer cycles whose commutator is also a Singer cycle. However, these may be the only examples of $\PSL(d,q)\cong\PGL(d,q)$ that cannot be generated by two $n$-cycles whose commutator is also an $n$-cycle. We have tested this for all $(d,q)$ in the set
\begin{multline*}
    \{(2,4),(2,8),(2,16),(2,32),(2,64),(2,128),(2,256),(3,8),(3,3),(3,9),(3,27), \\
    (3,5),(3,11),(3,17),(3,23),(4,2),(4,4),(4,8),(5,2),(5,4),(5,3),(5,5)\}.
\end{multline*}

Once we know that the monodromy group $\PGL(d,q)$ can be realised by some $[1,1]$-origami in a particular stratum, it is natural to ask which components of that stratum can be achieved.

For example, we can check that the origami
\[O_1 = ((1,2,3,4,5,6,7,8,9,10,11,12,13,14),(1,2,11,5,9,14,4,12,7,8,6,13,10,3))\]
in $\odd(6,6)$ and the origami
\[O_2 = ((1,2,3,4,5,6,7,8,9,10,11,12,13,14),(1,3,6,14,12,7,11,8,5,9,4,2,10,13))\]
in $\even(6,6)$ both have monodromy group isomorphic to $\PGL(2,13)$. So both non-hyperelliptic components are realised here.

Interestingly, however, it seems that when we restrict to minimal $[1,1]$-origamis in the minimal stratum then the component may be determined by $d$ and $q$. That is, it seems that the groups $\PSL(d,q)$ when realised as monodromy groups of primitive minimal $[1,1]$-origamis in $\calH(2g-2)$ are only achieved in one of the components. Indeed, testing some examples we found the following.

Let $X$ be a primitive minimal $[1,1]$-origami in $\calH(2g-2)$, then:
\begin{itemize}
    \item if $\Mon(X)\cong\PSL(2,8)$ then $X\in\odd(8)$;
    \item if $\Mon(X)\cong\PGamL(2,8)$ then $X\in\odd(8)$;
    \item if $\Mon(X)\cong M_{11}$ then $X\in\even(10)$;
    \item if $\Mon(X)\cong\PSL(3,3)$ then $X\in\even(12)$;
    \item if $\Mon(X)\cong\PSL(4,2)$ then $X\in\even(14)$;
    \item if $\Mon(X)\cong\PSL(2,16)$ then $X\in\odd(16)$;
    \item if $\Mon(X)\cong M_{23}$ then $X\in\odd(22)$;
    \item if $\Mon(X)\cong\PSL(3,5)$ then $X\in\odd(30)$;
    \item if $\Mon(X)\cong\PSL(2,32)$ then $X\in\odd(32)$.
\end{itemize}

We ask the following question:

\begin{question}
    Let $X$ be a primitive minimal $[1,1]$-origami in $\calH(2g-2)$ with $\Mon(X)\cong \PSL(d,q)$ for some $d$ and $q$. Is the spin parity of $X$ determined by $d$ and $q$?
\end{question}

%%%%%%%%%%%%%%%%%%%%%%%%%%%%%%%%%%%%%%%%%%%%%%%%%%%%%%

\section{Kontsevich-Zorich monodromy calculations}\label{sec:KZ-calcs}

Recall from Subsection~\ref{subsec:KZ-prep} that the Kontsevich-Zorich monodromy group of an origami $(X,\omega)$ is defined to be
\[\alpha(\Aff(X,\omega)):= \widetilde{\alpha}(\Aff(X,\omega))|_{H_{1}^{(0)}(X)},\]
where
\[\widetilde{\alpha}:\Aff(X,\omega)\to\Sp(H_{1}(X,\R))\cong\Sp(2g,\R)\]
is the symplectic representation of $\Aff(X,\omega)$ induced by its action on the fibers of the Hodge bundle.

We will calculate the Zariski closure $\overline{\alpha(\Aff(X,\omega))}^{Zarsiki}$. To do this, we will make use of the following criterion for Zariski density in $\Sp(2d,\R)$ due to Kany-Matheus~{\cite[Section 2.1]{KanMat}}.

Let $G\leq \Sp(2d,\R)$ for some $d\geq 1$. Suppose that
\begin{enumerate}
    \item there exists an element $A\in G$ that is \emph{Galois-pinching}; that is, the characteristic polynomial of $A$ is irreducible over $\Z$, splits over $\R$, and has Galois group of order $2^{d}\cdot d!$; and
    \item there exists a non-identity unipotent element $B\in G$ with $(B-I)$ having non-Lagrangian image;
\end{enumerate}
then $\overline{G}^{Zariski}\cong \Sp(2d,\R)$.

To study the action of the Veech groups of our origamis on their homology, we make use of the \texttt{sage-flatsurf} package~\cite{flatsurf}.

\subsection{Orientation double covers}

Recall that our orientation double covers of genus $g$ admit an anti-involution $\tau:X\to X$ with $\tau^{*}(\omega) = -\omega$ and $(X,\omega)/\tau$ being a quadratic differential in $Q_{\frac{g-1}{2}}(2g-3,-1^{3})$. This map $\tau$ then induces a further splitting of $H_{1}^{(0)}(X)$ as
\[H_{1}^{(0)}(X) = H_{1}^{+}(X)\oplus H_{1}^{-}(X),\]
where
\[H_{1}^{\pm}(X) = \{\gamma \in H_{1}^{(0)}(X)\,:\,\tau_{*}(\gamma) = \pm \gamma\}.\]

Note that, since the quotient surface has genus $\frac{g-1}{2}$, $H_{1}^{+}(X)$ has dimension $g-1$. Thus, since $H_{1}^{(0)}(X)$ has dimension $2g-2$, $H_{1}^{-}(X)$ has dimension $g-1$ as well. Therefore,
\[\alpha(\Aff(X,\omega))\leq \Sp(H_{1}^{+}(X))\times\Sp(H_{1}^{-}(X))\cong\Sp(g-1,\R)^{2}.\]
We will show that, in genus 3, 5 and 7, $\overline{\alpha(\Aff(X,\omega))}^{Zariski}\cong\Sp(g-1,\R)^{2}$. That is, the Kontsevich-Zorich monodromy is as large as possible.

By the results of Section~\ref{sec:SL2Z}, the orientation double covers in genus 3, 5 and 7 are contained in a single $\SL(2,\Z)$-orbit. Hence, it suffices to calculate the monodromy for a single surface in each orbit. Recall that the origami in Figure~\ref{fig:origami-anti-inv} is one of the odd genus AMN-origamis admitting an anti-involution.

\begin{figure}[b]
    \centering
    \begin{tikzpicture}[scale = 1.2, line width = 0.3mm]
    \draw (0,0)--(0,1);
    \draw (0,1)--(5,1);
    \draw [dotted] (5,1)--(6,1);
    \draw (6,1)--(11,1);
    \draw (11,1)--(11,0);
    \draw (11,0)--(6,0);
    \draw [dotted] (6,0)--(5,0);
    \draw (5,0)--(0,0);
    
    \foreach \i in {1,2,...,10}{
        \draw [dashed, color = darkgray] (\i,0)--(\i,1);
        \draw (\i,0.05)--(\i,-0.05);
        \draw (\i,1.05)--(\i,0.95);   
    }
    \foreach \i in {1,2,3,4,5}{
        \draw [color = darkgray] (\i-0.5,0.5) node {\footnotesize {${\i}$}};
        \draw [color = darkgray] (11.5-\i,0.5) node {\footnotesize ${{2g-\i}}$};
    }
    \draw (0,0.5) node[left] {$0$};
    \draw (11,0.5) node[right] {$0$};
    \foreach \i in {1,2,...,5}{
        \draw (\i-0.5,1) node[above] {$\i$};
        \draw (11-\i+0.5,1) node[above] {$2g-\i$};
    }
    \foreach \i in {1,2}{
        \draw (\i-0.5,0) node[below] {$2g-\i$};
    }
    \draw (2.5,0) node[below] {$1$};
    \foreach \i in {3,2}{
        \draw (6-\i+0.5,0) node[below] {$\i$};
    }
    \foreach \i in {4,6,8}{
        \draw (14-\i+0.5,0) node[below] {$2g-\i$};
    }
    \foreach \i in {3,5}{
        \draw (12-\i+0.5,0) node[below] {$2g-\i$};
    }
    \end{tikzpicture}
    \caption{A genus $g$ $[1,1]$-origami admitting an anti-involution.}\label{fig:origami-anti-inv}
\end{figure}
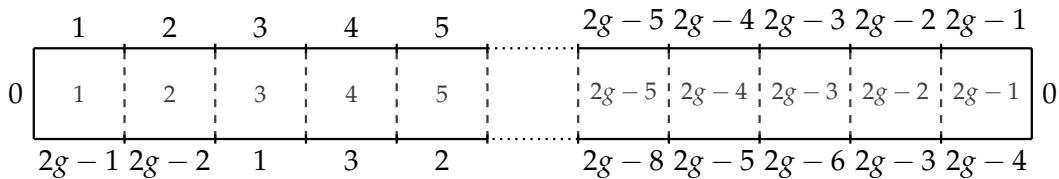

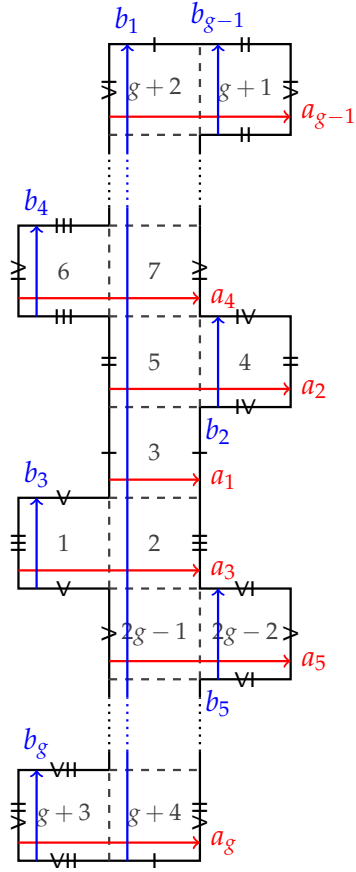
\begin{figure}[t]
    \centering
    \begin{tikzpicture}[scale = 1.2, line width = 0.3mm]
        \draw (0,0) -- (0,2) -- (-1,2) -- (-1,3) -- (0,3);
        \foreach \j in {0,-6}{
            \foreach \i in {0,1}{
                \draw (0+\i,3+\j) -- (0+\i,3.2+\j);
                \draw [dotted] (0+\i,3.2+\j) -- (0+\i,3.8+\j);
                \draw (0+\i,3.8+\j) -- (0+\i,4+\j);
            }
        }
        \draw (0,4) -- (0,5) -- (2,5) -- (2,4) -- (1,4);
        \draw (1,3) -- (1,2) -- (2,2) -- (2,1) -- (1,1) -- (1,-1) -- (2,-1) -- (2,-2) -- (1,-2);
        \draw (1,-3) -- (1,-4) -- (-1,-4) -- (-1,-3) -- (0,-3);
        \draw (0,-2) -- (0,-1) -- (-1,-1) -- (-1,0) -- (0,0);
        \draw [color = darkgray] (-0.5,-0.5) node {\footnotesize $1$};
        \draw [color = darkgray] (0.5,-0.5) node {\footnotesize $2$};
        \draw [color = darkgray] (0.5,0.5) node {\footnotesize $3$};
        \draw [color = darkgray] (0.5,1.5) node {\footnotesize $5$};
        \draw [color = darkgray] (1.5,1.5) node {\footnotesize $4$};
        \draw [color = darkgray] (-0.5,2.5) node {\footnotesize $6$};
        \draw [color = darkgray] (0.5,2.5) node {\footnotesize $7$};
        \draw [color = darkgray] (0.5,4.5) node {\footnotesize $g+2$};
        \draw [color = darkgray] (1.5,4.5) node {\footnotesize $g+1$};
        \draw [color = darkgray] (0.5,-1.5) node {\footnotesize $2g-1$};
        \draw [color = darkgray] (1.5,-1.5) node {\footnotesize $2g-2$};
        \draw [color = darkgray] (-0.5,-3.5) node {\footnotesize $g+3$};
        \draw [color = darkgray] (0.5,-3.5) node {\footnotesize $g+4$};
        \draw (0.5,5) node[node font = \sffamily] {\scriptsize I};
        \draw (0.5,-4) node[node font = \sffamily] {\scriptsize I};
        \draw (1.5,5) node[node font = \sffamily] {\scriptsize II};
        \draw (1.5,4) node[node font = \sffamily] {\scriptsize II};
        \draw (-0.5,3) node[node font = \sffamily] {\scriptsize III};
        \draw (-0.5,2) node[node font = \sffamily] {\scriptsize III};
        \draw (1.5,2) node[node font = \sffamily] {\scriptsize IV};
        \draw (1.5,1) node[node font = \sffamily] {\scriptsize IV};
        \draw (-0.5,0) node[node font = \sffamily] {\scriptsize V};
        \draw (-0.5,-1) node[node font = \sffamily] {\scriptsize V};
        \draw (1.5,-1) node[node font = \sffamily] {\scriptsize VI};
        \draw (1.5,-2) node[node font = \sffamily] {\scriptsize VI};
        \draw (-0.5,-3) node[node font = \sffamily] {\scriptsize VII};
        \draw (-0.5,-4) node[node font = \sffamily] {\scriptsize VII};
        \draw (0,0.5) node[node font = \sffamily, rotate = 90] {\scriptsize I};
        \draw (1,0.5) node[node font = \sffamily, rotate = 90] {\scriptsize I};
        \draw (0,1.5) node[node font = \sffamily, rotate = 90] {\scriptsize II};
        \draw (2,1.5) node[node font = \sffamily, rotate = 90] {\scriptsize II};
        \draw (-1,-0.5) node[node font = \sffamily, rotate = 90] {\scriptsize III};
        \draw (1,-0.5) node[node font = \sffamily, rotate = 90] {\scriptsize III};
        \draw (-1,2.5) node[node font = \sffamily, rotate = 90] {\scriptsize IV};
        \draw (1,2.5) node[node font = \sffamily, rotate = 90] {\scriptsize IV};
        \draw (0,-1.5) node[node font = \sffamily, rotate = 90] {\scriptsize V};
        \draw (2,-1.5) node[node font = \sffamily, rotate = 90] {\scriptsize V};
        \draw (0,4.5) node[node font = \sffamily, rotate = 90] {\scriptsize VI};
        \draw (2,4.5) node[node font = \sffamily, rotate = 90] {\scriptsize VI};
        \draw (-1,-3.5) node[node font = \sffamily, rotate = 90] {\scriptsize VII};
        \draw (1,-3.5) node[node font = \sffamily, rotate = 90]{\scriptsize VII};
        \draw (1,-1+0.2) node[right, color = red] {$a_{3}$};
        \draw (2,1+0.2) node[right, color = red] {$a_{2}$};
        \draw (2,-2+0.2) node[right, color = red] {$a_{5}$};
        \draw (1,2+0.2) node[right, color = red] {$a_{4}$};
        \draw (1,-4+0.2) node[right, color = red] {$a_{g}$};
        \draw (2,4+0.2) node[right, color = red] {$a_{g-1}$};
        \draw (-1+0.2,-3) node[above, color = blue] {$b_{g}$};
        \draw (1.2,-2) node[below, color = blue] {$b_{5}$};
        \draw (-1+0.2,0) node[above, color = blue] {$b_{3}$};
        \draw (1.2,1) node[below, color = blue] {$b_{2}$};
        \draw (-1+0.2,3) node[above, color = blue] {$b_{4}$};
        \draw (1.2,5) node[above, color = blue] {$b_{g-1}$};
        \foreach \i in {0,1,2}{
            \draw [dashed, color = darkgray] (0,-4+3*\i) -- (0,-3+3*\i);
            \draw [dashed, color = darkgray] (1,-2+3*\i) -- (1,-1+3*\i);
            \draw [color = red, ->] (-1,-4+3*\i+0.2) -- (1,-4+3*\i+0.2); 
            \draw [color = red, ->] (0,-2+3*\i+0.2) -- (2,-2+3*\i+0.2);
            \draw [color = blue, ->] (-0.8,-4+3*\i) -- (-0.8,-3+3*\i);
            \draw [color = blue, ->] (1.2,-2+3*\i) -- (1.2,-1+3*\i);
        }
        \foreach \i in {0,...,7}{
            \draw [dashed, color = darkgray] (0,-3+\i) -- (1,-3+\i);
        }
        \draw [color = red, ->] (0,0.2) -- (1,0.2) node[right, color = red]{$a_{1}$};
        \draw [color = blue] (0.2,-4) -- (0.2,-3+0.2);
        \draw [dotted, color = blue] (0.2,-3+0.2) -- (0.2,-2-0.2);
        \draw [color = blue] (0.2,-2-0.2) -- (0.2,3+0.2);
        \draw [dotted, color = blue] (0.2,3+0.2) -- (0.2,4-0.2);
        \draw [color = blue, ->] (0.2,4-0.2) -- (0.2,5) node[above, color = blue] {$b_{1}$};
    \end{tikzpicture}
    \caption{The orientation double cover origami under consideration. A homology basis is shown.}
    \label{fig:orientation-homology}
\end{figure}

For ease of calculation, we will instead move to the origami that is the image of this origami under $S\circ T$ (recall that $T$ and $S$ are the parabolic generators of $\SL(2,\Z)$ defined in Subsection~\ref{subsec:ori-prep}). The resulting origami, shown in Figure~\ref{fig:orientation-homology}, is given by $(h,v)$ with
\[h = (1,2)(3)(4,5)(6,7)\cdots(2g-2,2g-1)\]
and
\[v = (1)(2,3,5,7,\ldots,2g-1)(4)(6)\cdots(2g-2).\]
The cycles $\{a_{i},b_{i}\}_{i = 1}^{g}$ form a basis for the homology. The anti-involution $\tau$ acts as rotation by $\pi$ around the center of square 3. It has the following effect on the homology basis:
\[a_{1}\mapsto -a_{1}, b_{1}\mapsto -b_{1},\,\text{and, for }i\in\{2,4,\ldots,g-1\},\, a_{i}\mapsto-a_{i+1}, b_{i}\mapsto -b_{i+1}.\]

\begin{remark}
    We observe that the cycles $a_{i}$ for $i = 1,\ldots,g$ can be used in Forni's criterion~\cite{Forni-crit} to deduce that the Lyapunov spectrum has no zero exponents.
\end{remark}

We define
\begin{align*}
    \Sigma &:= a_{1}+a_{2}+\cdots+a_{g} \\
    Z &:= b_{1}+b_{2}+\cdots+b_{g}
\end{align*}
and, for $i\in\{2,4,\ldots,g-1\},$
\[\begin{aligned}
    & \left.\begin{aligned}
    \mu_{i} &:= a_{i}-a_{i+1} \\
    \nu_{i} &:= b_{i}-b_{i+1}
    \end{aligned}\,\,\,\right\rbrace\,\,\,\text{fixed by $\tau$} \\
    &\left.\begin{aligned}
    \eta_{i} &:= 4a_{1}-a_{i}-a_{i+1} \\
    \chi_{i} &:= 2b_{1} - g\cdot b_{i}-g\cdot b_{i+1}
    \end{aligned}\,\,\,\right\rbrace\,\,\,\text{negated by $\tau$.}
\end{aligned}\]
It can then be checked that we have the splitting
\[H_{1}(X,\R) = H_{1}^{st}(X)\oplus H_{1}^{+}(X)\oplus H_{1}^{-}(X) = \langle\Sigma,Z\rangle\oplus\langle\mu_{i},\nu_{i}\rangle\oplus\langle\eta_{i},\chi_{i}\rangle.\]

\subsubsection{Genus three}

In genus three, our surface is as shown in Figure~\ref{fig:genus-3-double-homology}.

\begin{figure}[t]
    \centering
    \begin{tikzpicture}[scale = 1.2, line width = 0.3mm]
        \draw (0,0) -- (0,2) -- (2,2) -- (2,1) -- (1,1) -- (1,-1) -- (-1,-1) -- (-1,0) -- (0,0);
        \draw [color = darkgray] (-0.5,-0.5) node {\footnotesize $1$};
        \draw [color = darkgray] (0.5,-0.5) node {\footnotesize $2$};
        \draw [color = darkgray] (0.5,0.5) node {\footnotesize $3$};
        \draw [color = darkgray] (0.5,1.5) node {\footnotesize $5$};
        \draw [color = darkgray] (1.5,1.5) node {\footnotesize $4$};
        \draw (0.5,2) node[node font = \sffamily] {\scriptsize I};
        \draw (0.5,-1) node[node font = \sffamily] {\scriptsize I};
        \draw (1.5,2) node[node font = \sffamily] {\scriptsize II};
        \draw (1.5,1) node[node font = \sffamily] {\scriptsize II};
        \draw (-0.5,0) node[node font = \sffamily] {\scriptsize III};
        \draw (-0.5,-1) node[node font = \sffamily] {\scriptsize III};
        \draw (0,0.5) node[node font = \sffamily, rotate = 90] {\scriptsize I};
        \draw (1,0.5) node[node font = \sffamily, rotate = 90] {\scriptsize I};
        \draw (0,1.5) node[node font = \sffamily, rotate = 90] {\scriptsize II};
        \draw (2,1.5) node[node font = \sffamily, rotate = 90] {\scriptsize II};
        \draw (-1,-0.5) node[node font = \sffamily, rotate = 90] {\scriptsize III};
        \draw (1,-1+0.2) node[right, color = red] {$a_{3}$};
        \draw (2,1+0.2) node[right, color = red] {$a_{2}$};
        \draw (-1+0.2,0) node[above, color = blue] {$b_{3}$};
        \draw (1.2,2) node[above, color = blue] {$b_{2}$};
        \foreach \i in {0,1}{
            \draw [dashed, color = darkgray] (0+\i,-1+2*\i) -- (0+\i,0+2*\i);
            \draw [dashed, color = darkgray] (0,0+\i) -- (1,0+\i);
            \draw [color = red, ->] (-1+\i,-1+2*\i+0.2) -- (1+\i,-1+2*\i+0.2);
            \draw [color = blue, ->] (-0.8+2*\i,-1+2*\i) -- (-0.8+2*\i,2*\i);
        }
        \draw [color = red, ->] (0,0.2) -- (1,0.2) node[right, color = red]{$a_{1}$};
        \draw [color = blue, ->] (0.2,-1) -- (0.2,2) node[above, color = blue] {$b_{1}$};
    \end{tikzpicture}
    \caption{The genus three orientation double cover and homology basis.}
    \label{fig:genus-3-double-homology}
\end{figure}
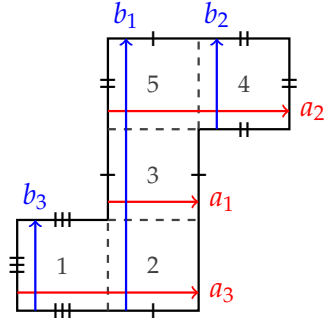

We work with the basis
\begin{align*}
    \Sigma &= a_{1}+a_{2}+a_{3} \\
    Z &= b_{1}+b_{2}+b_{3} \\
    \mu_{1} &= a_{2}-a_{3} \\
    \nu_{1} &= b_{2}-b_{3} \\
    \eta_{1} &= 4a_{1} -a_{2}-a_{3} \\
    \chi_{1} &= 2b_{1} - 3b_{2}-3b_{3}.
\end{align*}

Now, $\Aff(X,\omega) = \SL(X,\omega)$ contains the elements $A = \begin{bmatrix}
    1 & 2 \\
    0 & 1
\end{bmatrix}$ and $B = \begin{bmatrix}
    1 & 0 \\
    3 & 1
\end{bmatrix}$, and we can use \texttt{sage-flatsurf}~\cite{flatsurf} (although it is possible to calculate by hand in this small example) to see that
\[\alpha(A) = \begin{bmatrix}
    1 & 1 \\
    0 & 1
\end{bmatrix}\oplus\begin{bmatrix}
    1 & 1 \\
    0 & 1
\end{bmatrix} =: A^{+}\oplus A^{-},\]
and
\[\alpha(B) = \begin{bmatrix}
    1 & 0 \\
    3 & 1
\end{bmatrix}\oplus\begin{bmatrix}
    1 & 0 \\
    1 & 1
\end{bmatrix} =: B^{+}\oplus B^{-}.\]
We then see that the groups $\langle A^{+},B^{+}\rangle$ and $\langle A^{-},B^{-}\rangle$ are finite index subgroups of $\SL(2,\Z)$ and hence have Zariski closure $\SL(2,\R)\cong \Sp(2,\R)\cong\Sp(H_{1}^{\pm}(X))$. Thus, we have
\[\overline{\alpha(\Aff(X,\omega))}^{Zariski}\cong\Sp(2,\R)^{2}.\]
In fact, $(X,\omega)$ lies in the Prym eigenform locus of $\odd(4)$ (see Lanneau-Nguyen~\cite{LN14}) and so, by work of M\"oller~\cite{Moller} and Eskin-Kontsevich-Zorich~\cite{EKZ}, we know that the full Lyapunov spectrum is
\[1>\frac{2}{5} > \frac{1}{5} > -\frac{1}{5} > -\frac{2}{5} > -1.\]

\subsubsection{Genus five}

We do not draw the figure for this genus. However, the situation is as above using the splitting of homology discussed at the beginning of the section.

Here, $\Aff(X,\omega)\cong \SL(X,\omega)$ contains the elements
\[A = \begin{bmatrix}
    1 & 0 \\
    5 & 1
\end{bmatrix}, \,\,\,B = \begin{bmatrix}
    4215 & -2582 \\
    3038 & -1861
\end{bmatrix},\,\,\,\text{and}\,\,\,C = \begin{bmatrix*}[r]
    148 & -63 \\
    343 & -146
\end{bmatrix*}.\]
We can use \texttt{sage-flatsurf} to calculate that
\[\alpha(A) = A^{+}\oplus A^{-} := \begin{bmatrix}
    1 & 0 & 0 & 0 \\
    5 & 1 & 0 & 0 \\
    0 & 0 & 1 & 0 \\
    0 & 0 & 5 & 1 \\
\end{bmatrix}\oplus\begin{bmatrix}
    1 & 0 & 0 & 0 \\
    1 & 1 & 0 & 0 \\
    0 & 0 & 1 & 0 \\
    0 & 0 & 1 & 1 \\
\end{bmatrix},\]
\[\alpha(B) = B^{+}\oplus B^{-} := \begin{bmatrix*}[r]
    19 & -7 & -3 & 3 \\
    30 & -11 & -5 & 5 \\
    0 & 0 & 1 & -1 \\
    -5 & 2 & 6 & -5 \\
\end{bmatrix*}\oplus\begin{bmatrix*}[r]
    -3 & 17 & -3 & -3 \\
    2/5 & -9/5 & 3/5 & 1/5 \\
    2 & -14 & 1 & 1 \\
    3/5 & -16/5 & 2/5 & -1/5 \\
\end{bmatrix*},\]
and
\[\alpha(C) = C^{+}\oplus C^{-} := \begin{bmatrix*}[r]
    1 & 0 & 0 & 0 \\
    0 & 1 & 0 & 0 \\
    0 & 0 & 15 & -7 \\
    0 & 0 & 28 & -13 \\
\end{bmatrix*}\oplus\begin{bmatrix*}[r]
    13 & -24 & 0 & 6 \\
    28/5 & -51/5 & 0 & 14/5 \\
    -6 & 12 & 1 & -3 \\
    8/5 & 16/5 & 0 & 1/5 \\
\end{bmatrix*}.\]
It can then be checked that $B^{+}$ and $A^{-}B^{-}$ are Galois-pinching while $C^{\pm}$ are unipotent with
\[\dim_{\R}((C^{\pm}-I)H_{1}^{\pm}(X)) = 1\neq \frac{4}{2} = \frac{1}{2}\dim_{\R}H_{1}^{\pm}(X).\]
Hence, it follows from the criterion of Kany-Matheus that
\[\overline{\alpha(\Aff(X,\omega))}^{Zariski}\cong\Sp(4,\R)^{2}.\]

\subsubsection{Genus seven}

Again, we do not draw the figure here. Similar calculations to the above can be carried out. In particular, it is possible to use \texttt{sage-flatsurf} to find an element $M$ (whose entries are too large to write here) of $\Aff(X,\omega)\cong\SL(X,\omega)$ with $\alpha(M) =: M^{+}\oplus M^{-}$ such that $M^{\pm}$ are both Galois pinching. Letting $N$ be the multitwist in the $(3,4)$ direction so that $N = \begin{bmatrix}
    -251 & 189 \\
    -336 & 253
\end{bmatrix}$, we obtain $\alpha(N) =: N^{+}\oplus N^{-}$ with $N^{\pm}$ both being non-identity unipotents and 
\[\dim_{\R}((N^{\pm}-I)H_{1}^{\pm}(X)) = 1\neq \frac{6}{2} = \frac{1}{2}\dim_{\R}H_{1}^{\pm}(X).\]

Hence, we see that
\[\overline{\alpha(\Aff(X,\omega))}^{Zariski}\cong\Sp(6,\R)^{2}.\]

\subsection{Non-symmetric origamis}

Here, our origamis have no further symmetries and also no non-trivial automorphisms. As such, we get no further splitting of the homology beyond
\[H_{1}(X,\R) = H_{1}^{st}(X)\oplus H_{1}^{(0)}(X).\]

Given a square labelled $i$ we define $\sigma_{i}$ to be the cycle corresponding to the bottom side of the square oriented to the right and $\zeta_{i}$ to be left side of the square oriented upwards. See Figure~\ref{fig:square-cycles}. Since, for our origamis, every corner of every square corresponds to the single zero, these cycles are indeed absolute cycles. 

\begin{figure}[b]
    \centering
    \begin{tikzpicture}[scale = 1.2, line width = 0.3mm]
        \draw (0,0) -- node[below]{$\sigma_{i}$} (1,0) -- (1,1) -- (0,1) -- node[left]{$\zeta_{i}$}(0,0);
        \draw[->] (0,0) -- (0.55,0);
        \draw[->] (0,0) -- (0,0.55);
        \draw[color = darkgray] (0.5,0.5) node{\footnotesize $i$};
    \end{tikzpicture}
    \caption{The cycles $\sigma_{i}$ and $\zeta_{i}$.}
    \label{fig:square-cycles}
\end{figure}
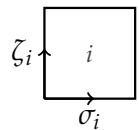

\begin{figure}[t]
    \centering
    \begin{tikzpicture}[scale = 1.2, line width = 0.3mm]
    \draw (0,0)--(0,1)--(7,1)--(7,0)--(0,0);
    \foreach \i in {1,2,...,6}{
        \draw [dashed, color = darkgray] (\i,0)--(\i,1);
        \draw (\i,0.05)--(\i,-0.05);
        \draw (\i,1.05)--(\i,0.95);
        \draw (\i-0.5,0.5) node[color = darkgray] {\footnotesize $\i$};
        \draw (\i-0.5,1) node[above] {$\i$};
    }
    \draw (0,0.5) node[left] {$0$};
    \draw (7,0.5) node[right] {$0$};
    \draw (6.5,1) node[above] {$7$};
    \draw (6.5,0.5) node[color = darkgray] {\footnotesize $7$};
    \draw (0.5,0) node[below] {$5$};
    \draw (1.5,0) node[below] {$4$};
    \draw (2.5,0) node[below] {$1$};
    \draw (3.5,0) node[below] {$6$};
    \draw (4.5,0) node[below] {$2$};
    \draw (5.5,0) node[below] {$7$};
    \draw (6.5,0) node[below] {$3$};
    \end{tikzpicture}
    \caption{A $[1,1]$-origami in $\odd(6)$.}
    \label{fig:H-odd-6}
\end{figure}

If the origami has $n$ squares, then the tautological plane is spanned by
\[\Sigma := \sigma_{1}+\cdots+\sigma_{n}\,\,\,\,\,\,\,\text{and}\,\,\,\,\,\,\,Z := \zeta_{1}+\cdots+\zeta_{n}.\]
We will continue to use this notation below. In each case, we will apply the criterion of Kany-Matheus to prove that
\[\overline{\alpha(\Aff(X,\omega))}^{Zariski}\cong\Sp(H_{1}^{(0)}(X))\cong\Sp(2g-2,\R).\]

\subsubsection{Genus four}

Consider the origami in $\odd(6)$ shown in Figure~\ref{fig:H-odd-6}. We define $\Sigma$ and $Z$ as above and further define
\begin{align*}
    \eta_{1} = \sigma_{1}-\sigma_{3},  &&\eta_{2} = \sigma_{2}-\sigma_{5}, &&&\eta_{3} = \sigma_{3}-\sigma_{7} \\
    \eta_{4} = \sigma_{4}-\sigma_{2},  &&\eta_{5} = \sigma_{5}-\sigma_{1}, &&&\eta_{6} = \sigma_{6}-\sigma_{4}.
\end{align*}
It can be checked that we have
\[H_{1}(X,\R) = H_{1}^{st}(X)\oplus H_{1}^{(0)}(X) = \langle\Sigma,Z\rangle\oplus\langle\eta_{i}\rangle_{i=1}^{6}.\]

The group $\Aff(X,\omega)\cong\SL(X,\omega)$ contains the elements
\[A = \begin{bmatrix*}[r]
    -40 & 11 \\
    -11 & 3
\end{bmatrix*},\,\,\,B = \begin{bmatrix*}[r]
    -16 & 7 \\
    -7 & 3
\end{bmatrix*},\,\,\,\text{and}\,\,\,C = \begin{bmatrix}
    4 & -3 \\
    3 & -2
\end{bmatrix}.\]
We can check with \texttt{sage-flatsurf} that
\[\alpha(A) = \begin{bmatrix*}[r]
    -1 & 1 & 0 & 1 & -2 & 0 \\
    -1 & 0 & 0 & 2 & -2 & 0 \\
    -1 & 0 & 0 & 1 & -1 & 0 \\
    -1 & 3 & -3 & 3 & -3 & 1 \\
    -1 & 0 & 1 & 1 & -2 & 0 \\
    -1 & 2 & -3 & 3 & -1 & 0
\end{bmatrix*},\,\,\,\,\,\,\alpha(B) = \begin{bmatrix*}[r]
    -3 & 1 & 0 & 1 & 0 & 0 \\
    0 & 0 & 0 & 0 & 0 & 1 \\
    0 & -1 & 1 & 0 & 0 & -1 \\
    -1 & 0 & 0 & 0 & 0 & 1 \\
    0 & -1 & 0 & 0 & 1 & -1 \\
    0 & -1 & 0 & 0 & 0 & 1
\end{bmatrix*},\]
and
\[\alpha(C) = \begin{bmatrix*}[r]
    2 & -1 & 0 & 1 & 0 & 0 \\
    0 & 1 & 0 & 0 & 0 & 1 \\
    0 & 0 & 1 & 0 & 0 & -1 \\
    -1 & 1 & 0 & 0 & 0 & 1 \\
    0 & 0 & 0 & 0 & 1 & -1 \\
    0 & 0 & 0 & 0 & 0 & 1
\end{bmatrix*}.\]
Moreover, $\alpha(A)\alpha(B)$ is Galois pinching and $\alpha(C)$ is unipotent with 
\[\dim_{\R}((\alpha(C)-I)H_{1}^{(0)}(X)) = 2\neq \frac{6}{2} = \frac{1}{2}\dim_{\R}H_{1}^{(0)}(X).\]
Hence, we have
\[\overline{\alpha(\Aff(X,\omega))}^{Zariski}\cong\Sp(6,\R).\]

\begin{figure}[t]
    \centering
    \begin{tikzpicture}[scale = 1.2, line width = 0.3mm]
    \draw (0,0)--(0,1)--(7,1)--(7,0)--(0,0);
    \foreach \i in {1,2,...,6}{
        \draw [dashed, color = darkgray] (\i,0)--(\i,1);
        \draw (\i,0.05)--(\i,-0.05);
        \draw (\i,1.05)--(\i,0.95);
        \draw (\i-0.5,0.5) node[color = darkgray] {\footnotesize $\i$};
        \draw (\i-0.5,1) node[above] {$\i$};
    }
    \draw (0,0.5) node[left] {$0$};
    \draw (7,0.5) node[right] {$0$};
    \draw (6.5,1) node[above] {$7$};
    \draw (6.5,0.5) node[color = darkgray] {\footnotesize $7$};
    \draw (0.5,0) node[below] {$7$};
    \draw (1.5,0) node[below] {$5$};
    \draw (2.5,0) node[below] {$1$};
    \draw (3.5,0) node[below] {$6$};
    \draw (4.5,0) node[below] {$4$};
    \draw (5.5,0) node[below] {$3$};
    \draw (6.5,0) node[below] {$2$};
    \end{tikzpicture}
    \caption{A $[1,1]$-origami in $\even(6)$.}
    \label{fig:H-even-6}
\end{figure}

Now consider the origami in $\even(6)$ shown in Figure~\ref{fig:H-even-6}. Define
\begin{align*}
    \eta_{1} = \sigma_{1}-\sigma_{3},  &&\eta_{2} = \sigma_{2}-\sigma_{7}, &&&\eta_{3} = \sigma_{3}-\sigma_{6} \\
    \eta_{4} = \sigma_{4}-\sigma_{5},  &&\eta_{5} = \sigma_{5}-\sigma_{2}, &&&\eta_{6} = \sigma_{6}-\sigma_{4}.
\end{align*}
It can again be checked that we have
\[H_{1}(X,\R) = H_{1}^{st}(X)\oplus H_{1}^{(0)}(X) = \langle\Sigma,Z\rangle\oplus\langle\eta_{i}\rangle_{i=1}^{6}.\]

The group $\Aff(X,\omega)\cong\SL(X,\omega)$ contains the elements
\[A = \begin{bmatrix*}[r]
    5 & 24 \\
    -4 & -19
\end{bmatrix*},\,\,\,B = \begin{bmatrix*}[r]
    -15 & -64 \\
    4 & 17
\end{bmatrix*},\,\,\,\text{and}\,\,\,C = \begin{bmatrix*}[r]
    -11 & -36 \\
    4 & 13
\end{bmatrix*}.\]
We can check with \texttt{sage-flatsurf} that
\[\alpha(A) = \begin{bmatrix*}[r]
    0 & -1 & 0 & 0 & 0 & 0 \\
    1 & 2 & 0 & 2 & 0 & -2 \\
    1 & -2 & -1 & 1 & 0 & 0 \\
    1 & 0 & -1 & 0 & -1 & 1 \\
    1 & 0 & -1 & 1 & -1 & 0 \\
    1 & -2 & -1 & -3 & 0 & 3
\end{bmatrix*},\,\,\,\,\,\,\,\,\alpha(B) = \begin{bmatrix*}[r]
    1 & -4 & -2 & 0 & 2 & 2 \\
    0 & 3 & 2 & 0 & 0 & -2 \\
    0 & -4 & -3 & 0 & 0 & 4 \\
    0 & -2 & -1 & 1 & 1 & 1 \\
    0 & 2 & 2 & 0 & 1 & -2 \\
    0 & -2 & -2 & 0 & 0 & 3
\end{bmatrix*},\]
and
\[\alpha(C) = \begin{bmatrix*}[r]
    1 & -2 & -2 & -3 & 2 & 2 \\
    0 & 3 & 4 & 3 & -2 & -4 \\
    0 & 3 & -2 & -3 & 3 & 3 \\
    0 & -1 & -1 & 0 & 1 & 1 \\
    0 & 1 & 3 & 1 & 0 & -3 \\
    0 & -3 & -3 & -3 & 3 & 4
\end{bmatrix*}.\]
Moreover, $\alpha(A)^{2}\alpha(B)$ is Galois pinching and $\alpha(C)^{2}$ is unipotent with 
\[\dim_{\R}((\alpha(C)^{2}-I)H_{1}^{(0)}(X)) = 2\neq \frac{6}{2} = \frac{1}{2}\dim_{\R}H_{1}^{(0)}(X).\]
Hence, we have
\[\overline{\alpha(\Aff(X,\omega))}^{Zariski}\cong\Sp(6,\R).\]

\subsubsection{Genus five}

Consider the origami $(h,v)$ in $\odd(8)$ defined by $h = (1,2,\ldots,9)$ and $v = (1,3,6,5,8,2,7,4,9)$. Define
\begin{align*}
    \eta_{1} = \sigma_{1}-\sigma_{3},  &&\eta_{2} = \sigma_{2}-\sigma_{7}, &&&\eta_{3} = \sigma_{3}-\sigma_{6}, &&\eta_{4} = \sigma_{4} - \sigma_{9} \\
    \eta_{5} = \sigma_{5}-\sigma_{8},  &&\eta_{6} = \sigma_{6}-\sigma_{5}, &&&\eta_{7} = \sigma_{7}-\sigma_{4}, &&\eta_{8} = \sigma_{8} - \sigma_{2}.
\end{align*}
It can again be checked that we have
\[H_{1}(X,\R) = H_{1}^{st}(X)\oplus H_{1}^{(0)}(X) = \langle\Sigma,Z\rangle\oplus\langle\eta_{i}\rangle_{i=1}^{8}.\]

The group $\Aff(X,\omega)\cong\SL(X,\omega)$ contains the elements 
\[A = \begin{bmatrix*}[r]
    -367 & 77 \\
    -143 & 30
\end{bmatrix*},\,\,\,B = \begin{bmatrix*}[r]
    -2809 & 609 \\
    -369 & 80
\end{bmatrix*},\,\,\,\text{and}\,\,\,C = \begin{bmatrix*}[r]
    596 & -175 \\
    2023 & -594
\end{bmatrix*}.\]
It can be checked that $\alpha(A)\alpha(B)$ is Galois pinching and $\alpha(C)$ is unipotent with 
\[\dim_{\R}((\alpha(C)-I)H_{1}^{(0)}(X)) = 2\neq \frac{8}{2} = \frac{1}{2}\dim_{\R}H_{1}^{(0)}(X).\]
Hence, we have
\[\overline{\alpha(\Aff(X,\omega))}^{Zariski}\cong\Sp(8,\R).\]

Now, consider the origami $(h,v)$ in $\even(8)$ defined by $h = (1,2,\ldots,9)$ and $v = (1,3,8,4,7,5,6,3,9)$. Define
\begin{align*}
    \eta_{1} = \sigma_{1}-\sigma_{3},  &&\eta_{2} = \sigma_{2}-\sigma_{9}, &&&\eta_{3} = \sigma_{3}-\sigma_{8}, &&\eta_{4} = \sigma_{4} - \sigma_{7} \\
    \eta_{5} = \sigma_{5}-\sigma_{6},  &&\eta_{6} = \sigma_{6}-\sigma_{2}, &&&\eta_{7} = \sigma_{7}-\sigma_{5}, &&\eta_{8} = \sigma_{8} - \sigma_{4}.
\end{align*}
As above, it can be checked that we have
\[H_{1}(X,\R) = H_{1}^{st}(X)\oplus H_{1}^{(0)}(X) = \langle\Sigma,Z\rangle\oplus\langle\eta_{i}\rangle_{i=1}^{8}.\]

The group $\Aff(X,\omega)\cong\SL(X,\omega)$ contains the elements
\[A = \begin{bmatrix*}[r]
    15175 & -4473 \\
    1737 & -512
\end{bmatrix*},\,\,\,B = \begin{bmatrix*}[r]
    -30853 & 9121 \\
    -3545 & 1048
\end{bmatrix*},\,\,\,\text{and}\,\,\,C = \begin{bmatrix*}[r]
    13 & -3 \\
    48 & -11
\end{bmatrix*}.\]
It can be checked that $\alpha(A)^{5}\alpha(B)$ is Galois pinching and $\alpha(C)$ is unipotent with 
\[\dim_{\R}((\alpha(C)-I)H_{1}^{(0)}(X)) = 2\neq \frac{8}{2} = \frac{1}{2}\dim_{\R}H_{1}^{(0)}(X).\]
Hence, we have
\[\overline{\alpha(\Aff(X,\omega))}^{Zariski}\cong\Sp(8,\R).\]

\subsubsection{Genus six}

Consider the origami $(h,v)$ in $\odd(10)$ defined by $h = (1,2,\ldots,11)$ and $v = (1,3,11,10,4,7,8,2,5,6,9)$. Define
\begin{align*}
    \eta_{1} = \sigma_{1}-\sigma_{3},  &&\eta_{2} = \sigma_{2}-\sigma_{5}, &&&\eta_{3} = \sigma_{3}-\sigma_{11}, &&\eta_{4} = \sigma_{4} - \sigma_{7} &&\eta_{5} = \sigma_{5} - \sigma_{6} \\
    \eta_{6} = \sigma_{6}-\sigma_{9},  &&\eta_{7} = \sigma_{7}-\sigma_{8}, &&&\eta_{8} = \sigma_{8}-\sigma_{2}, &&\eta_{9} = \sigma_{9} - \sigma_{1}, &&\eta_{10} = \sigma_{10}-\sigma_{4}.
\end{align*}
It can again be checked that we have
\[H_{1}(X,\R) = H_{1}^{st}(X)\oplus H_{1}^{(0)}(X) = \langle\Sigma,Z\rangle\oplus\langle\eta_{i}\rangle_{i=1}^{10}.\]

We can use \texttt{sage-flatsurf} to find an element $M$ in $\Aff(X,\omega)\cong\SL(X,\omega)$ with $\alpha(M)$ Galois-pinching (the entries are too large to write it here). Defining $N$ to be the $(1,1)$ multitwist so that $N = \begin{bmatrix}
    -11 & 12 \\
    -12 & 13
\end{bmatrix}$, it can be checked that $\alpha(N)$ is unipotent with 
\[\dim_{\R}((\alpha(N)-I)H_{1}^{(0)}(X)) = 2\neq \frac{10}{2} = \frac{1}{2}\dim_{\R}H_{1}^{(0)}(X).\]
Hence, we have
\[\overline{\alpha(\Aff(X,\omega))}^{Zariski}\cong\Sp(10,\R).\]

Now, consider the origami $(h,v)$ in $\even(10)$ defined by $h = (1,2,\ldots,11)$ and $v = (1,3,8,11,6,4,5,2,9,10,7)$. Define
\begin{align*}
    \eta_{1} = \sigma_{1}-\sigma_{3},  &&\eta_{2} = \sigma_{2}-\sigma_{9}, &&&\eta_{3} = \sigma_{3}-\sigma_{8}, &&\eta_{4} = \sigma_{4} - \sigma_{5} &&\eta_{5} = \sigma_{5} - \sigma_{2} \\
    \eta_{6} = \sigma_{6}-\sigma_{4},  &&\eta_{7} = \sigma_{7}-\sigma_{1}, &&&\eta_{8} = \sigma_{8}-\sigma_{11}, &&\eta_{9} = \sigma_{9} - \sigma_{10}, &&\eta_{10} = \sigma_{10}-\sigma_{7}.
\end{align*}
It can again be checked that we have
\[H_{1}(X,\R) = H_{1}^{st}(X)\oplus H_{1}^{(0)}(X) = \langle\Sigma,Z\rangle\oplus\langle\eta_{i}\rangle_{i=1}^{10}.\]

We can use \texttt{sage-flatsurf} to find an element $M$ in $\Aff(X,\omega)\cong\SL(X,\omega)$ with $\alpha(M)$ Galois-pinching (the entries are too large to write it here). Defining $N$ to be the $(-1,1)$ multitwist so that $N = \begin{bmatrix*}[r]
    10 & 9 \\
    -9 & -8
\end{bmatrix*}$, it can be checked that $\alpha(N)$ is unipotent with 
\[\dim_{\R}((\alpha(N)-I)H_{1}^{(0)}(X)) = 2\neq \frac{10}{2} = \frac{1}{2}\dim_{\R}H_{1}^{(0)}(X).\]
Hence, we have
\[\overline{\alpha(\Aff(X,\omega))}^{Zariski}\cong\Sp(10,\R).\]

This completes the proof of Theorem~\ref{thm:KZ-monodromy}.

%%%%%%%%%%%%%%%%%%%%%%%%%%%%%%%%%%%%%%%%%%%%%%%%%%%%%%

\appendix
\section{Topological realisations of the constructions}\label{app:alg}

Here we describe an algorithm for topologically constructing the filling pairs associated to the odd genus origamis constructed by the first author and Menasco-Nieland (those of Subsection~\ref{subsec:AMN-odd}). Similar algorithms can be formulated in order to construct the filling pairs associated to the generalised constructions of Section~\ref{sec:gen-const}, but we do not include them here. We also direct the reader to recent work of Chang-Menasco~\cite{CM} which gives a geometric construction of coherent minimally intersecting filling pairs (coherent here meaning that the filling pairs are those arising as core curves of $[1,1]$-origamis).

\begin{figure}[b]
    \centering
    \includegraphics{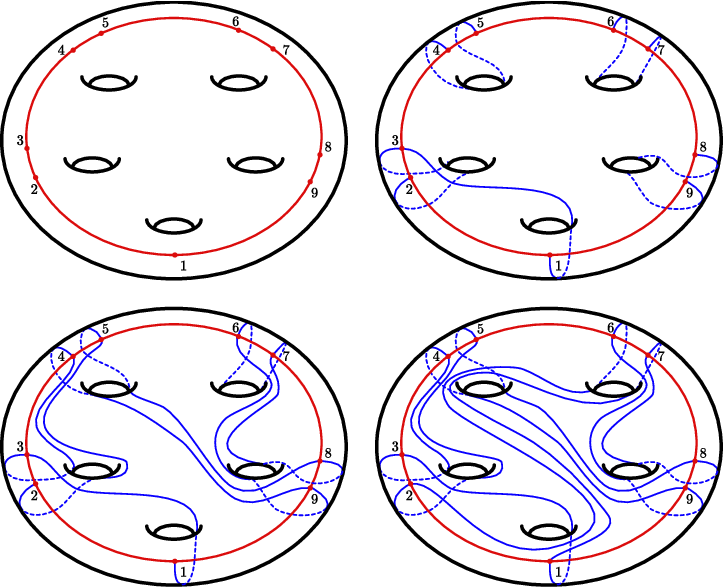}
    \caption{The steps of Algorithm~\ref{alg:draw} for $\tau = (1,3,4,9,6,2,5,8,7)$. The red curve corresponds to $\alpha$, and the blue curve in the bottom right diagram corresponds to $\beta$.}
    \label{fig:draw-alg}
\end{figure}

\begin{algorithm}\label{alg:draw}
Let $g\geq 3$ be odd and let $O = (\sigma_{g},\tau)$ be a $[1,1]$-origami constructed using the method of Subsection~\ref{subsec:AMN-odd}. The following steps will construct a filling pair corresponding to the core curves of the cylinders of $O$.

\begin{itemize}
\item[{\bf Step 1.}] Begin with a genus $g$ surface $S$ as shown in the top left of Figure~\ref{fig:draw-alg}. Draw a simple closed curve $\alpha$ enclosing all of the genus of the surface, and along $\alpha$ add $2g-1$ labelled vertices so that vertex 1 is below one handle, and each pair $i,i+1$, for even $2\leq i\leq 2g-2$, is placed next to another handle going clockwise around $\alpha$.

\item[{\bf Step 2.}] Next, draw a line from each vertex on $\alpha$ outwards around the back of the surface and through the nearest handle to represent a segment of $\beta$ on the underside of $S$. Also add a solid line from the dotted line leaving the vertex labelled 1 to the vertex labelled 3, as shown in the top right of Figure~\ref{fig:draw-alg}.

\item[{\bf Step 3.}] Now, start with the pair of dotted lines that originate from $3 = \tau(1)$ and $2 = \tau(1)-1$, and draw a pair of solid lines from each of the dotted lines (since now $\beta$ is using the handle to get to the top side of $S$) to vertices labelled by $\tau(3)$ and $\tau(2)=\tau(3)+1$, making sure that they fellow travel through the surface. We continue in this way until we reach $\tau^{-1}(1)$ and $\tau^{-1}(1)-1 = \tau^{-1}(2)$. See the bottom left of Figure~\ref{fig:draw-alg}.

\item[{\bf Step 4.}] Finally, draw two solid lines from the dotted lines starting at $\tau^{-1}(1)$ and $\tau^{-1}(2)$ to the dotted lines coming from $1$ and $2$, respectively. See the bottom right of Figure~\ref{fig:draw-alg}.
\end{itemize}
\end{algorithm}

We clearly have that $\alpha$ is a simple closed curve. Next, we observe that $\beta$ is forced to be a simple closed curve. Indeed, we do not introduce any self-intersections since the lines in Step 3 can always be drawn without intersections as the problem is equivalent to drawing lines from the boundary of a circle to points in the interior. The lines in Step 4 can always be drawn through the space between the vertex labelled 1 and its closest handle. Moreover, since $\tau$ is a $(2g-1)$-cycle, we have a single curve. Since $\alpha$ and $\beta$ have been constructed so that, given the correct orientation, they follow the permutations $\sigma_{g}$ and $\tau$, respectively, the fact that $[\sigma_{g},\tau]$ is a $(2g-1)$-cycle forces the complement of $\alpha$ and $\beta$ to be a single disk (it is in fact an $(8g-4)$-gon). Therefore, $\alpha$ and $\beta$ are a filling pair whose dual square-complex realises the origami $O$.

%%%%%%%%%%%%%%%%%%%%%%%%%%%%%%%%%%%%%%%%%%%%%%%%%%%%%%

\end{document}